\documentclass{amsart}
\usepackage{amsmath ,amssymb ,amsthm, mathrsfs, stmaryrd, dsfont}
\usepackage{bm}
\usepackage[all]{xy}
\usepackage[pdftex]{graphicx} 
\usepackage[pdftex ,pdfborder={0 0 0}]{hyperref} 
\usepackage{tikz, tikz-cd}
\usepackage{shadethm}
\usepackage{graphicx, xcolor}
\usepackage{url}
\DeclareMathAlphabet\mathbfcal{OMS}{cmsy}{b}{n}
\DeclareMathAlphabet{\mathpzc}{OT1}{pzc}{m}{it}

\theoremstyle{definition}
\newtheorem{defin}{Definition}[section]

\newtheorem{thm}[defin]{Theorem}
\newtheorem{conj}[defin]{Conjecture}
\newtheorem{lem}[defin]{Lemma}
\newtheorem{prop}[defin]{Proposition}
\newtheorem{cor}[defin]{Corollary}

\theoremstyle{remark}
\newtheorem{rem}[defin]{Remark}

\newtheorem{quest}[defin]{Question}

\newcommand{\Ric}{\mathrm{Ric}}
\newcommand{\cFut}{\check{\mathrm{F}}\mathrm{ut}}

\newcommand{\sqddbar}{\sqrt{-1} \partial \bar{\partial}}

\newcommand{\cmu}{\bm{\check{\mu}}}
\newcommand{\NAmu}{\bm{\check{\mu}}_{\mathrm{NA}}}

\newcommand{\pcH}{\mathcal{H}^{\mathbb{R}}_{\mathrm{NA}}}
\newcommand{\ibar}{\bar{\imath}}
\newcommand{\jbar}{\bar{\jmath}}

\newcommand{\cW}{\check{\mathcal{W}}}
\newcommand{\nH}{\mathcal{H}_{\mathrm{NA}}}

\title[I, Perelman's entropy and $\mu$-cscK metrics]{Entropies in $\mu$-framework of canonical metrics and K-stability, I -- Archimedean aspect: Perelman's entropy and $\mu$-cscK metrics}

\author{Eiji Inoue}

\address{Graduate School of Mathematical Sciences, the University of Tokyo \endgraf 3-8-1 Komaba, Meguro, Tokyo 153-8914, Japan. }

\email{eijinoe@ms.u-tokyo.ac.jp}

\begin{document}

\begin{abstract}
This is the first in a series of two papers (cf. \cite{Ino4}) studying $\mu$-cscK metrics and $\mu$K-stability from a new perspective, inspired by observations in \cite{Ino3} and in this first paper. 

This first paper is about a characterization of $\mu$-cscK metrics in terms of Perelman's $W$-entropy $\cW^\lambda$. 
We regard Perelman's $W$-entropy as a functional on the tangent bundle $T \mathcal{H} (X, L)$ of the space $\mathcal{H} (X, L)$ of K\"ahler metrics in a given K\"ahler class $L$. 
The critical points of $\cW^\lambda$ turn out to be $\mu^\lambda$-cscK metrics. 
When $\lambda \le 0$, the supremum along the fibres gives a smooth functional on $\mathcal{H} (X, L)$, which we call \textit{$\mu$-entropy}. 
Then $\mu^\lambda$-cscK metrics are also characterized as critical points of this functional, similarly as extremal metric is characterized as the critical points of Calabi functional. 

We also prove the $W$-entropy is monotonic along geodesics, following Berman--Berndtsson's subharmonicity argument. 
Then studying the limit of the $W$-entropy, we obtain a lower bound of the $\mu$-entropy. 
This bound is not just analogous, but indeed related to Donaldson's lower bound on Calabi functional by the extremal limit $\lambda \to -\infty$. 
\end{abstract}

\maketitle

\tableofcontents

\section{Introduction}

In this series, we unveil a new aspect of $\mu$-cscK metrics and $\mu$K-stability introduced in \cite{Ino2, Ino3} (see also \cite{Lah1}), which is concerned with optimal degeneration of polarized varieties. 
We also observe a universal aspect of various frameworks on optimal degeneration including K\"ahler--Ricci soliton and extremal metric. 

This work is inspired by the pioneering works of Donaldson \cite{Don}, Dervan \cite{Der2} and Xia \cite{Xia} in the context of extremal metric (see also \cite{Der1} and \cite[Theorem 1.3]{His16}), of He \cite{He}, Dervan--Sz\'ekelyhidi \cite{DS} and recent Han--Li \cite{HL2} in the context of K\"ahler--Ricci soliton (see also \cite[Theorem 4.3 \& Corollary 4.5]{Ber}, \cite{His} and \cite{CSW}), and of Chi Li \cite{Li1} on the normalized volume in the context of Sasaki--Einstein metric. 
As we review in section \ref{Observations}, one can think these studies originated from volume minimization \cite{TZ1, MSY}, and so does our study. 

\subsection{Backgrounds: $\mu$-cscK metrics and $\mu$K-stability}

\subsubsection{$\mu$-cscK metric}

Let $X$ be a compact K\"ahler manifold and $L$ be the K\"ahler class of an ample $\mathbb{Q}$-line bundle. 
We assume $X$ is smooth throughout this first article; we also study singular $X$ in the second article. 
Many of our arguments work also for transcendental $L$, but we restrict our interest to the rational case. 
In the rational case, we can adjust our study to non-archimedean pluripotential framework (cf. \cite{BJ1}), which we pursue in the subsequent article. 

For a parameter $\lambda \in \mathbb{R}$ and a real holomorphic vector field $\xi$ on $X$ which is of the form $\xi = \mathrm{Im} (\partial^\sharp \theta_\xi) = J \nabla \theta_\xi/2$ (cf. section \ref{convention}), a K\"ahler metric $\omega \in L$ is called a \textit{$\mu^\lambda$-cscK metric with respect to $\xi$} if $L_\xi \omega = 0$ and the \textit{$\mu^\lambda_\xi$-scalar curvature} 
\[ s^\lambda_\xi (\omega) := s (\omega) + \Delta \theta_\xi - |\partial^\sharp \theta_\xi|^2 - \lambda \theta_\xi \]
is constant. 
This curvature notion naturally (and in the unique manner) comes up from the moment map picture on K\"ahler--Ricci soliton \cite{Ino1, Ino2}. 
Indeed, K\"ahler--Ricci soliton is equivalent to $\mu^{2\pi \lambda}$-cscK metric in the K\"ahler class $-\lambda^{-1} K_X$ ($\lambda > 0$) on a Fano manifold $X$ (cf. \cite{Ino1}). 
The framework on $\mu$-cscK metrics encloses both the frameworks on K\"ahler--Ricci solitons and cscK metrics. 
We refer \cite{Ino2} for foundational aspects of $\mu$-cscK metrics and a more extensive work \cite{Lah1, Lah2} concerned with moment map picture. 

In the study \cite{Ino2}, we introduced the following \textit{$\mu$-volume functional}: 
\[ \log \mathrm{Vol}^\lambda (\xi) := \frac{\int_X s^\lambda_\xi e^{\theta_\xi} \omega^n}{\int_X e^{\theta_\xi} \omega^n} + \lambda \log \int_X e^{\theta_\xi} \omega^n. \]
The functional is defined on the Lie algebra $\mathfrak{k}$ of a compact subgroup $K \subset \mathrm{Aut} (X, L)$ and is independent of the choice of $K$-invariant metric $\omega \in L$. 
For Fano manifold, the functional is equivalent to Tian--Zhu's volume functional \cite{TZ1} appeared in the study of K\"ahler--Ricci soliton. 

This series is mainly devoted to a further exploration of the $\mu$-volume functional. 
We study a natural extension of (the minus log of) the $\mu$-volume functional defined on the space of non-archimedean metrics ($\supset$ test configurations), which we call the \textit{non-arhchimedean $\mu$-entropy}. 
The extension is introduced based on the equivariant cohomological nature of this $\mu$-volume functional. 

When our computation has equivariant cohomological background, it is convenient to compute with the moment map $\mu_\xi = -\theta_\xi/2$ rather than with the $\bar{\partial}$-potential $\theta_\xi$. 
It is thus convenient to assign a terminology for $\mu^\lambda_{-\xi/2}$-cscK metric. 
We call it \textit{$\check{\mu}^\lambda_\xi$-cscK metric}. 
A \textit{$\mu^\lambda$-cscK metric} is a $\mu^\lambda_\xi$/$\check{\mu}^\lambda_\xi$-cscK metric for some $\xi$. 

\subsubsection{Digression: phase transition and extremal limit}

Throughout this series, we study $\mu$-cscK metrics for a fixed parameter $\lambda$, especially for $\lambda \le 0$. 
It is observed in \cite{Ino2} that interesting phenomenon happens when varying $\lambda$, which inspires us to interpret the parameter $-\lambda$ as ``empirical temperature''. 
(It may be more appropriate regarding $\lambda^{-1}$ as ``absolute temperature''. 
The smaller $\lambda$, the higher temperature. 
Negative temperature is hotter than any positive $\lambda > 0$, in view of statistical role of the reverse temperature. )
Though it is off the topic of this series, we briefly describe it here as it is a fascinating aspect of $\mu$-cscK metrics. 

Firstly, a fall in ``temperature'' $-\lambda$ yields a ``phase transition phenomenon'': the possible states $\xi$ associated to some $\mu^\lambda$-cscK metrics branch off at some $-\lambda \ll 0$. 
Even on the simplest variety $\mathbb{C}P^1$, there appears a $\mu^\lambda_\xi$-cscK metric with respect to a non-trivial state $\xi \neq 0$ once the temperature $-\lambda$ gets across the ``phase transition point'' $-8\pi/\int_{\mathbb{C}P^1} \omega$. 
Such new \textit{non-trivial} $\mu^\lambda$-cscK metrics are unique modulo $\mathrm{Aut} (\mathbb{C}P^1)$ in this case, while the Fubini--Study metric gives a \textit{trivial} $\mu^\lambda$-cscK metric. 
For this example, the non-trivial state $\xi \neq 0$ is breaking the symmetry: $\{ g \in SU (2) ~|~ g_* \xi = \xi \} = U (1) \neq SU (2)$, where $SU (2) \subset \mathrm{Aut} (\mathbb{C}P^1)$ is a maximal compact group. 
This phenomenon is contrast to the case $\lambda \le 0$, in which case $\xi$ must be in the center of a maximal compact group (cf. \cite[Corollary 3.19]{Ino2}). 
In particular, non-trivial possible states $\xi \in \mathfrak{su} (2)$ associated to some $\mu^\lambda$-cscK metrics are note unique nor discrete; they form a sphere $S^2 = SU (2). \xi \subset \mathfrak{su} (2)$. 
In view of the $\mu$-volume functional, the nontrivial $\mu^\lambda_\xi$-cscK metric turns into a new ``stable state'', while the Fubini--Study metric transforms into ``supercooled state''. 

\begin{figure}[h]
\begin{minipage}[b]{0.45\linewidth}
\centering
\includegraphics[width=5cm]{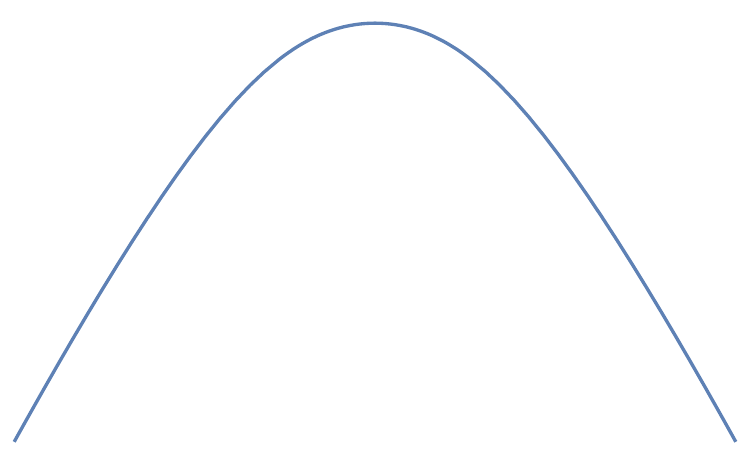}
\end{minipage}
\begin{minipage}[b]{0.45\linewidth}
\centering
\includegraphics[width=5cm]{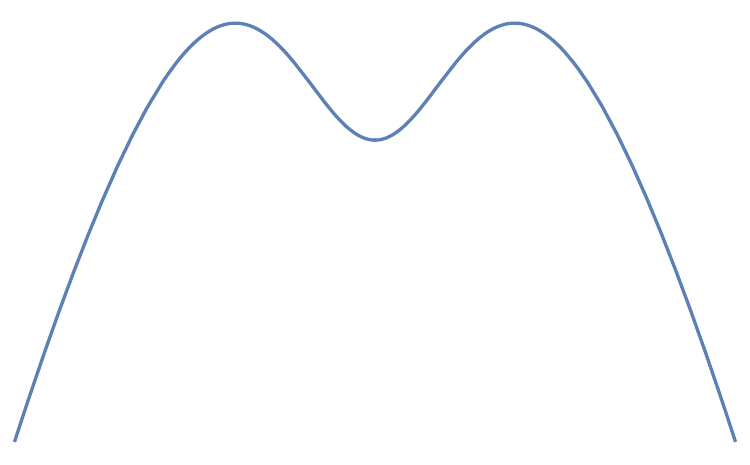}
\end{minipage}
\caption{The graphs of the $\mu$-entropy ($=$ the minus log of the $\mu$-volume) on $\mathfrak{u} (1) \subset \mathfrak{aut} (\mathbb{C}P^1)$ for $\lambda < \lambda_{\mathrm{ice}}$ and $\lambda > \lambda_{\mathrm{ice}}$, respectively. }
\end{figure}

Conversely, when the ``temperature'' is sufficiently high $-\lambda \gg 0$, the possible states $\xi$ associated to some $\mu^\lambda$-cscK metrics are uniquely determined for each ``temperature'' $-\lambda$ (cf. \cite[Theorem B (3)]{Ino2}), and is characterized as the minimizer of the $\mu$-volume functional. 
As ``temperature'' heats up $-\lambda \to + \infty$, the rescaled state $\lambda \xi$ converge to the extremal vector field $\xi_{\mathrm{ext}}$ (cf. \cite[Theorem D]{Ino2}). 
This implies the $\mu^\lambda$-scalar curvature converge to the extremal scalar curvature as $\lambda \to -\infty$: 
\[ s^\lambda_{\xi_\lambda} (\omega) ~ \longrightarrow ~ s (\omega) - \theta_{\xi_{\mathrm{ext}}}. \]
Based on this observation, we showed the following in \cite{Ino2} and \cite{Ino3} (the latter is essentially due to \cite{Lah1}), respectively. 
\begin{itemize}
\item If there is an extremal metric in $L$, we can construct by perturbation a family $\{ \omega_\lambda \}_{\lambda \ll 0}$ of $\mu^\lambda$-cscK metrics in $L$ converging to the extremal metric as $\lambda \to -\infty$. 

\item Conversely, if there are $\mu^\lambda$-cscK metrics in $L$ (or just $\mu^\lambda$K-semistable) for every $- \infty < \lambda \ll 0$, then $(X, L)$ is relatively K-semistable. 
\end{itemize}
Therefore, we can understand extremal metric (resp. relative K-stability) as the limit of $\mu^\lambda$-cscK metrics (resp. $\mu^\lambda$K-stability). 

We conjecture the following uniqueness for ``high temperature'' case $-\lambda \ge 0$, which fails for $-\lambda \ll 0$ as we already noted. 

\begin{conj}
\label{uniqueness}
For $\lambda \le 0$, $\mu^\lambda$-cscK metrics are unique modulo the automorphism group. 
\end{conj}

The main difficulty is that the $\mu$-volume functional is not convex in general, even when $\lambda \le 0$. 
As a partial evidence of the conjecture, we currently know the following. 

\begin{itemize}
\item \cite{Lah2}: For a fixed $\lambda \in \mathbb{R}$ and a \textit{fixed vector} $\xi$, $\mu^\lambda_\xi$-cscK metrics are unique modulo $\mathrm{Aut}_\xi (X, L)$. 

\item \cite{Ino2}: If $\mu^\lambda_\xi$-cscK metric exists for $\lambda \le 0$, then $\xi$ is a local minimizer of the $\mu^\lambda$-volume functional. 
Moreover, there are only finitely many local minimizers in the Lie algebra $\mathfrak{k}$ of any compact group $K \subset \mathrm{Aut} (X, L)$. 
\end{itemize}

As a consequence of our main theorem of this article, we will see $\xi$ is actually a global minimizer of the $\mu^\lambda$-volume functional, if $\mu^\lambda_\xi$-cscK metric exists for $\lambda \le 0$. 
Thus the problem reduces to the uniqueness of global minimizers of the $\mu^\lambda$-volume functional, which is equivalent to the uniqueness of the global maximizers of $\mu$-entropy for vectors which we introduce later. 

\subsubsection{Equivariant intersection}

Equivariant intersection is a basic language for describing $\mu$K-stability and our $\mu$-entropy of test configurations. 
It also brings us transparent understanding on some known results in K-stability (cf. \cite{Ino3, Leg}). 
The idea of using localization can be traced back to Futaki \cite{Fut}, although we apply localization in a slightly different way from his original work. 
We briefly explain the concept below. 
We refer to \cite{EG1, GS, GGK} and \cite{Ino3} for further information. 

In this first article, we compute equivariant intersection using equivariant differential form. 
For a compact Lie group $K$, an equivariant closed $2$-form $\alpha + \nu$ on a manifold $X$ is a pair of $K$-invariant $2$-form $\alpha$ and a smooth $K$-equivariant map $\nu: X \to \mathfrak{k}^\vee$ to the dual of the Lie algebra which satisfy $-d\nu_\xi = i_{\xi^\#} \omega$ for the pairing $\nu_\xi = \langle \nu, \xi \rangle$. 
It defines a second equivariant cohomology class $[\alpha + \nu] \in H^2_K (X, \mathbb{R})$ and conversely, every second equivariant cohomology class $M \in H^2_K (X, \mathbb{R})$ can be represented in this way. 
Here we identify the singular equivariant cohomology $H^*_K (X, \mathbb{R})$ with the deRham--Cartan model ${^\hbar H_{\mathrm{dR}, K}} (X)$ with $\hbar = 1$ which is described in \cite{Ino3}. 
We must choose positive $\hbar$ for our results. 

When $G$ is a complex Lie group with a maximal compact subgroup $K$, we have a natural isomorphism $H^*_G (X, \mathbb{Z}) \cong H^*_K (X, \mathbb{Z})$. 
Then for a real analytic function $f$ on $\mathbb{R}$, the equivariant intersection $(M. f (L); \xi) = {^1 (M. f (L); \xi)}$ is computed as the integration 
\[ (M. f(L); \xi) = \int_X (\alpha + \nu_\xi) f (\omega+ \mu_\xi) = \int_X \Big{(} \frac{1}{(n-1)!} f^{(n-1)} (\mu_\xi) \omega^n + \frac{1}{n!} \nu_\xi f^{(n)} (\mu_\xi) \omega^n \Big{)}. \]
When we consider $G = \mathbb{C}^\times$, we denote $(M. f (L); \rho. \eta)$ for the positive generator $\eta \in \mathfrak{u} (1)$ and $\rho \in \mathbb{R}$. 
In this series, we are especially interested in the equivariant intersection $(M. e^L; \rho)$ for $\rho \ge 0$. 

As we observed in \cite{Ino3}, the equivariant intersection on scheme is actually a purely algebro-geometric notion: the above description is just a one possible way to compute the equivariant intersection. 
In the second article, we are interested in such algebraic aspect.

\subsubsection{$\mu$K-stability}

Let $\xi$ be a real holomorphic vector field generating a closed torus $T = \overline{\exp (\mathbb{R}. \xi)}$ in $\mathrm{Aut} (X, L)$. 
We denote its complexification by the same symbol $T$. 
A polarized scheme $(X, L)$ is called \textit{$\check{\mu}^\lambda_\xi$K-semistable} if $\cFut^\lambda_\xi (\mathcal{X}, \mathcal{L}) \ge 0$ for every $T$-equivariant test configuration $(\mathcal{X}, \mathcal{L})$. 
It is shown in \cite{Lah1} and is restated in \cite{Ino3} that if a smooth polarized manifold $(X, L)$ admits a $\check{\mu}^\lambda_\xi$-cscK metric, then $(X, L)$ is $\check{\mu}^\lambda_\xi$K-semistable. 
Here the \textit{$\mu$-Futaki invariant} $\cFut^\lambda_\xi (\mathcal{X}, \mathcal{L}) := D_\xi \check{\mu} (\mathcal{X}, \mathcal{L}) + \lambda D_\xi \check{\sigma} (\mathcal{X}, \mathcal{L})$ for a normal test configuration $(\mathcal{X}, \mathcal{L})$ is defined by the following equivariant intersection: 
\begin{align*}
D_\xi \check{\mu} (\mathcal{X}, \mathcal{L}) 
&:= 2 \pi \frac{(K_{\bar{\mathcal{X}}/\mathbb{C}P^1}^T. e^{\bar{\mathcal{L}}_T}; \xi) \cdot (e^{L_T}; \xi)  - (K_X^T. e^{L_T}; \xi) \cdot (e^{\bar{\mathcal{L}}_T}; \xi) }{(e^{L_T}; \xi)^2}, 
\\
D_\xi \check{\sigma} (\mathcal{X}, \mathcal{L})
&:= \frac{(\bar{\mathcal{L}}_T. e^{\bar{\mathcal{L}}_T}; \xi) \cdot (e^{L_T}; \xi) - (L_T. e^{L_T}; \xi) \cdot (e^{\bar{\mathcal{L}}_T}; \xi) }{(e^{L_T}; \xi)^2} - \frac{(e^{\bar{\mathcal{L}}_T}; \xi)}{(e^{L_T}; \xi)}. 
\end{align*}
For general non-normal test configuration, we replace the equivariant canonical divisor class $K_{\bar{\mathcal{X}}/\mathbb{C}P^1}^T$ with an equivariant Chow class $\kappa_{\bar{\mathcal{X}}/\mathbb{C}P^1}^T$ as explained in \cite{Ino3}, which fits into equivariant Grothendieck--Riemann--Roch theorem for general scheme (cf. \cite{EG2}). 

In \cite{Ino3}, we introduce an equivariant characteristic class $\bm{\mu}^\lambda = \bm{\mu} + \lambda \bm{\sigma}$ for family of polarized schemes. 
The $\mu$-Futaki invariant is understood as ``\textit{the derivative of the $\mu$-character at $\xi$ along the test configuration $(\mathcal{X}, \mathcal{L})$}''. 
We justified this slogan by introducing a derivation $\mathcal{D}_\xi$ on equivariant cohomology and applying it to the $\mu$-character $\bm{\mu}^\lambda$. 
This idea applies not only to test configuration, but also to family of polarized schemes, which yields an analogue of CM line bundle in our $\mu$K-stability context. 

For a test configuration $(\mathcal{X}, \mathcal{L})$, the $\mu$-character $\bm{\mu}^\lambda_{\mathbb{C}^\times} (\mathcal{X}, \mathcal{L})$ lives in $\hat{H}_{\mathbb{C}^\times} (\mathbb{C}, \mathbb{R})$. 
Identifying $\hat{H}_{\mathbb{C}^\times} (\mathbb{C}, \mathbb{R})$ with the ring $\mathbb{R} \llbracket \eta \rrbracket$ of formal power series, we can regard $\bm{\mu}^\lambda_{\mathbb{C}^\times} (\mathcal{X}, \mathcal{L})$ as a formal power series. 
We can easily see by the method in \cite{Ino3} that the corresponding power series actually converges around the origin and extends to a real analytic function $\cmu^\lambda (\mathcal{X}, \mathcal{L}; \bullet)$ on $\mathbb{R}$. 
We are especially interested in the function $\cmu^\lambda (\mathcal{X}, \mathcal{L}; \bullet)$ restricted to the half line $[0, \infty)$. 

This original description is not important in this first article, while it turns to be essential in the second article \cite{Ino4}. 
We later adopt an equivariant intersection formula on the compactified space $\bar{\mathcal{X}}$ as a definition of the functional $\cmu^\lambda (\mathcal{X}, \mathcal{L}; \bullet)$ in this first article. 
These definitions can be compared by equivariant localization (see \cite{Ino4}). 

\subsection{Main concepts and results}

We firstly introduce main concepts of this article. 
The concepts turned up implicitly in the study \cite{Ino2, Ino3}. 

\subsubsection{Perelman's $W$-entropy as a functional on the tangent bundle}

Let $X$ be a compact K\"ahler manifold and $L$ be a K\"ahler class. 
We denote by $\mathcal{H} (X, L)$ the space of K\"ahler metrics in the K\"ahler class $L$. 

For a reference K\"ahler metric $\omega_{\mathrm{ref}} \in \mathcal{H} (X, L)$, we obtain an open embedding $\mathcal{H} (X, L) \hookrightarrow C^\infty (X)/\mathbb{R}$ by assigning the K\"ahler potential $\varphi$ of a metric $\omega_\varphi = \omega_{\mathrm{ref}} + \sqddbar \varphi$. 
For another choice of the reference $\omega'_{\mathrm{ref}} = \omega_\psi$, the coordinate change is just given by parallel translation $\varphi \mapsto \varphi - \psi$, so we can identify the tangent bundle $T \mathcal{H} (X, L)$ with the product $\mathcal{H} (X, L) \times C^\infty (X)/\mathbb{R}$ in the canonical way: 
\[ \mathcal{H} (X, L) \times C^\infty (X)/\mathbb{R} \xrightarrow{\sim} T \mathcal{H} (X, L): (\omega, f) \mapsto \frac{d}{dt}\Big{|}_{t=0} (\omega + t \sqddbar f). \]
We use this canonical identification to compute variations of functionals on $T\mathcal{H} (X, L)$. 
We sometimes call $f \in T_\omega \mathcal{H} = C^\infty (X)/\mathbb{R}$ a \textit{momentum} and a pair $(\omega, f)$ a \textit{state}. 

For a K\"ahler metric $\omega \in \mathcal{H} (X, L)$ and $f \in C^{0, 1} (X)$, the (normalized) \textit{$W$-entropy} is defined as 
\begin{equation}
\cW^\lambda (\omega, f) := - \frac{\int_X \big{(} s (\omega) + |\partial^\sharp f|^2 - \lambda (n+f) \big{)} e^f \omega^n}{\int_X e^f \omega^n} - \lambda \log  \int_X e^f \frac{\omega^n}{n!}. 
\end{equation}
We can write as $\cW^\lambda = \cW + \lambda \check{S}$ by putting
\begin{align}
\cW (\omega, f) 
&:= - \frac{\int_X \big{(} s (\omega) + |\partial^\sharp f|^2 \big{)} e^f \omega^n}{\int_X e^f \omega^n}, 
\\
\check{S} (\omega, f) 
&:= \frac{\int_X (n+f) e^f \omega^n}{\int_X e^f \omega^n} - \log  \int_X e^f \frac{\omega^n}{n!}. 
\end{align}
As $\cW^\lambda (\omega, f+c) = \cW^\lambda (\omega, f)$, we may regard $\cW^\lambda$ as a functional on the tangent bundle $T \mathcal{H} (X, L)$ (by restricting to $C^\infty (X) \subset C^{0,1} (X)$). 

We recall our $W$-entropy is equivalent to the minus half of Perelman's original definition \cite{Per} (modulo constant depending only on $n$ and $\tau$): 
\[ W (g, f, \tau) = \int_X \big{(} \tau (R (g) + |\nabla f|^2) - (2n - f) \big{)} e^{-f} \frac{\mathrm{vol}_g}{(4\pi \tau)^n} \]
for $f$ normalized as $\int_X e^{-f} \mathrm{vol}_g/(4\pi \tau)^n = 1$. 
Here $R (g)$ denotes the Riemannian scalar curvature for a general Riemannian metric, which is $2 s (\omega)$ when $g$ is K\"ahler. 
Though we consider the same functional, we have different viewpoint and motivation in this article. 
In Perelman's study on Ricci flow the parameter $\tau$ is positive and is regarded as a reverse time $\tau = T_0 -t$, while in our study of $\mu$-cscK metrics the negative case $\tau^{-1} = \lambda \le 0$ is more fruitful, mainly due to the invertibility of the operator $\Delta - \nabla f - \lambda$, and the parameter is fixed (similarly as we fix it in the study of \textit{normalized} K\"ahler--Ricci flow). 
The crucial difference is that we are interested in the functional restricted to the space of K\"ahler metrics, not the space of the whole Riemannian metrics, to encounter with holomorphy. 

As in \cite{Per}, we define the \textit{$\mu$-entropy} as 
\begin{equation}
\cmu^\lambda (\omega) := \sup_{f \in C^{0,1} (X)} \cW^\lambda (\omega, f). 
\end{equation}
We will see in section \ref{Perelman's W-functional and mu-entropy} that the supremum is indeed attained by a smooth function. 
The functional $\cmu^\lambda: \mathcal{H} (X, L) \to \mathbb{R}$ is lower semicontinuous by definition, while the smoothness, even the continuity, is not evident. 
In section \ref{mu-cscK metrics and Perelman's $W$-entropy}, we show the smoothness for $\lambda \le 0$. 

This functional is intensively studied especially for Fano manifold with $L = -K_X$ and $\lambda = 2\pi$ in many literatures including \cite{He, TZ2, TZZZ, Ber, DS}. 
We discuss this in section \ref{Observations, He}. 
In this article, we study this functional for a general polarized manifold. 

The functional plays a role analogous to the Calabi functional. 
We will see in section \ref{Observations, Calabi} that the $\mu$-entropy is indeed connected to the Calabi functional by the extremal limit $\lambda^{-1} \to 0$. 
Similarly to Calabi functional, the functional $\cmu^\lambda$ is always bounded from below: 
\begin{equation} 
\cmu^\lambda (\omega) \ge \cW^\lambda (\omega, 0) = 2\pi \frac{n (K_X. L^{\cdot n-1})}{L^{\cdot n}} + \lambda n - \lambda \log \frac{(L^{\cdot n})}{n!}. 
\end{equation}
Note this is contrast to the (weighted) Mabuchi/Ding functional as these are unbounded when (weighted) K/D-unstable. 

\subsubsection{Remarks on test configuration}

Schemes are of finite type over $\mathbb{C}$ throughout this series. 

In this series, a \textit{test configuration} of a polarized scheme $(X, L)$ is a triple $(\mathcal{X}, \mathcal{L}; \rho)$ of the following data. 
\begin{itemize}
\item $\mathcal{X}$ is a scheme endowed with a $\mathbb{C}^\times$-action, a proper flat surjective morphism $\varpi: \mathcal{X} \to \mathbb{C}$ which is $\mathbb{C}^\times$-equivariant with respect to the scaling action on $\mathbb{C}$, and a $\mathbb{C}^\times$-equivariant isomorphism $j_\circ : X \times \mathbb{C}^* \xrightarrow{\sim} \varpi^{-1} (\mathbb{C}^*)$. 

\item $\mathcal{L}$ is a $\mathbb{C}^\times$-equivariant relatively ample $\mathbb{Q}$-line bundle on $\mathcal{X}$ satisfying $j_\circ^* \mathcal{L} = p_X^* L$ on $X \times \mathbb{C}^*$. 
We often identify $\mathcal{L}$ with an element of the equivariant cohomology $H^2_{\mathbb{C}^\times} (\mathcal{X}, \mathbb{R})$. 
When we emphasize that $\mathcal{L}$ is a $\mathbb{C}^\times$-equivariant class, we write it as $\mathcal{L}_{\mathbb{C}^\times}$. 

\item $\rho$ is a non-negative real number. (scaling parameter)
\end{itemize}

\begin{rem}
In the non-archimedean pluripotential theory (cf. \cite{BHJ1, BJ1, BJ2}), a test configuration $(\mathcal{X}, \mathcal{L}; \rho)$ is understood as the rescaled non-archimedean metric $\phi_{(\mathcal{X}, \mathcal{L}), \rho}$ of $\phi_{(\mathcal{X}, \mathcal{L})}$. 
The scaling parameter $\rho$ was veiled in literatures as it can be dismissed thanks to the fact that `classical' invariants for test configuration, such as $L^p$-norms and Mabuchi invariant (a modification of Donaldson--Futaki invariant in the non-archimedean context), are homogeneous with respect to $\rho$. 

In analytic viewpoint, these `classical' invariants are homogeneous because these express the slope of some functionals \textit{on the space $\mathcal{H} (X, L)$ of K\"ahler potentials}. 
As we see in this article, the $\mu$-entropy $\bm{\mu}^\lambda (\mathcal{X}, \mathcal{L}; \rho)$ for test configurations introduced later can be understood as the limit of the $W$-entropy along geodesic rays. 
The $W$-entropy is defined \textit{on the tangent bundle $T \mathcal{H} (X, L)$ of the space of K\"ahler potentials}, not on $\mathcal{H} (X, L)$. 
This difference from the classical functional yields the non-linearlity on $\rho$ of the $\mu$-entropy of test configurations. 
\end{rem}

We denote by $\mathbb{C}_-$ the affine space $\mathbb{C}$ endowed with the reverse scaling $\mathbb{C}^\times$-action: $z.t = t^{-1} z$. 
We can compactify $\mathcal{X}$ by equivariantly gluing $X \times \mathbb{C}_-$ and $\mathcal{X}$ over $X \times \mathbb{C}^*$. 
Similarly we obtain a natural extension $\bar{\mathcal{L}} \to \bar{\mathcal{X}}$ of $\mathcal{L}$. 

We can shift the original $\mathbb{C}^\times$-action on $\mathcal{L}$ by a weight $m \in \mathbb{Z}$ of $\mathbb{C}^\times$. 
Under this shifting, the extension $\bar{\mathcal{L}}$ is replaced with $\bar{\mathcal{L}} + m [\mathcal{X}_0]$. 
Since $\bar{\mathcal{L}}$ is relatively ample, we can take $m$ so that $\bar{\mathcal{L}} + m [\mathcal{X}_0] = \bar{\mathcal{L}} + \varpi^* \mathcal{O} (m)$ is ample. 
As this shifting will not affect the interested invariant $\cmu^\lambda (\mathcal{X}, \mathcal{L}; \rho)$, we sometimes assume $\bar{\mathcal{L}}$ is ample by shifting. 
(We must pay attention that the moment map associated to $\bar{\mathcal{L}}$ also shifts, so that we are not free to choose the normalization of the moment map in the $U (1)$-equivariant class $\bar{\mathcal{L}}_{U (1)}$ when we make use of the ampleness of $\bar{\mathcal{L}}$. )

We call a test configuration $(\mathcal{X}, \mathcal{L}; \rho)$ \textit{normal (resp. smooth)} if $\mathcal{X}$ is normal (resp. smooth) and \textit{snc smooth} if it is smooth and the central fibre $\mathcal{X}_0$ is supported on an snc divisor on $\mathcal{X}$. 
We say a test configuration $(\mathcal{X}, \mathcal{L}; \rho)$ \textit{dominates} another test configuration $(\mathcal{X}', \mathcal{L}'; \rho)$ if the canonical birational map $j'_\circ \circ j_\circ^{-1}: \mathcal{X} \dashrightarrow \mathcal{X}'$ extends to a morphism $\mathcal{X} \to \mathcal{X}'$. 
The \textit{trivial configuration} is the product $(X \times \mathbb{C}, L \times \mathbb{C})$ endowed with the scaling $\mathbb{C}^\times$-action $(x, \tau). t = (x, \tau t)$. 

For a normal test configuration $(\mathcal{X}, \mathcal{L}; \rho)$, we denote by $(\mathcal{X}_d, \mathcal{L}_d; \rho)$ the normalization of the base change of $(\mathcal{X}, \mathcal{L})$ along $z^d: \mathbb{C} \to \mathbb{C}$. 
We have $\phi_{(\mathcal{X}_d, \mathcal{L}_d), \rho} = \phi_{(\mathcal{X}, \mathcal{L}), d \rho}$ for the associated non-archimedean metrics. 

\subsubsection{$\mu$-entropy of test configuration}

For a normal test configuration $(\mathcal{X}, \mathcal{L}; \rho)$, we define its \textit{$\mu^\lambda$-entropy} $\cmu^\lambda (\mathcal{X}, \mathcal{L}; \rho)$ by 
\begin{align} 
\cmu (\mathcal{X}, \mathcal{L}; \rho) 
&:= 2 \pi \frac{(K_X. e^L) - \rho. (K_{\bar{\mathcal{X}}/\mathbb{C}P^1}^{\mathbb{C}^\times}. e^{\bar{\mathcal{L}}_{\mathbb{C}^\times}}; \rho)}{(e^L) - \rho. (e^{\bar{\mathcal{L}}_{\mathbb{C}^\times}}; \rho)}, 
\\
\bm{\check{\sigma}} (\mathcal{X}, \mathcal{L}; \rho) 
&:= \frac{(L. e^L) - \rho. (\bar{\mathcal{L}}_{\mathbb{C}^\times}. e^{\bar{\mathcal{L}}_{\mathbb{C}^\times}}; \rho)}{(e^L) - \rho. (e^{\bar{\mathcal{L}}_{\mathbb{C}^\times}}; \rho)} - \log \Big{(} (e^L) - \rho. (e^{\bar{\mathcal{L}}_{\mathbb{C}^\times}}; \rho) \Big{)},
\\
\cmu^\lambda (\mathcal{X}, \mathcal{L}; \rho) 
&:= \cmu (\mathcal{X}, \mathcal{L}; \rho) + \lambda \bm{\check{\sigma}} (\mathcal{X}, \mathcal{L}; \rho). 
\end{align}

By equivariant localization (cf. \cite{GGK, Ino3}), we can write these invariants as 
\begin{align*}
\cmu (\mathcal{X}, \mathcal{L}; \rho) 
&= 2\pi \frac{(\kappa_{\mathcal{X}_0}. e^{\mathcal{L}|_{\mathcal{X}_0}}; \rho)}{(e^{\mathcal{L}|_{\mathcal{X}_0}}; \rho)}, 
\\
\bm{\check{\sigma}} (\mathcal{X}, \mathcal{L}; \rho) 
&= \frac{(\mathcal{L}|_{\mathcal{X}_0}. e^{\mathcal{L}|_{\mathcal{X}_0}}; \rho)}{(e^{\mathcal{L}|_{\mathcal{X}_0}}; \rho)} - \log (e^{\mathcal{L}|_{\mathcal{X}_0}}; \rho), 
\end{align*}
which relates the definition in this first article to the $\mu$-character introduced in \cite{Ino3} and that in the second article \cite{Ino4}. 
Here $\kappa_{\mathcal{X}_0} \in \mathrm{CH}^{\mathbb{C}^\times} (\mathcal{X}_0, \mathbb{Q})$ denotes the $\mathbb{C}^\times$-equivariant canonical Chow class of the central fibre, which can be defined for general scheme (cf. \cite{Ful, EG2} or \cite{Ino3}). 
In this first article, we only use these localized expression as a simplified expression of the equivariant intersection in the above definition. 

\subsubsection{$\mu$-entropy for vectors}

It is observed in \cite{Ino2, Ino3} that for the product configuration $(\mathcal{X}, \mathcal{L}) = (X_{\mathbb{C}, \xi}, L_{\mathbb{C}, \xi})$ associated to (the fundamental vector field $\xi$ of) a $\mathbb{C}^\times$-action on $X$, we have 
\[ \cmu^\lambda (X_{\mathbb{C}, \xi}, L_{\mathbb{C}, \xi}; \rho) = -\log \frac{\mathrm{Vol}^\lambda (-\rho \xi/2)}{(e^n. n!)^\lambda} = \cW^\lambda (\omega, \mu_{\rho \xi}). \]
Since the central fibre is reduced, we have $\cmu^\lambda (X_{\mathbb{C}, d \xi}, L_{\mathbb{C}, d \xi}; \rho) = \cmu^\lambda ((X_{\mathbb{C}, \xi}, L_{\mathbb{C}, \xi})_d; \rho) = \cmu^\lambda (X_{\mathbb{C}, \xi}, L_{\mathbb{C}, \xi}; d \rho)$ for non-negative integer $d$. 

Thus for a torus action $(X, L) \circlearrowleft T$, we can define the functional $\cmu^\lambda (X, L; \bullet): \mathbb{Q} \otimes N \to \mathbb{R}$ by $p/q \otimes \xi \mapsto \cmu^\lambda (X_{\mathbb{C}, \xi}, L_{\mathbb{C}, \xi}; p/q)$, where $N \subset \mathfrak{t}$ is the lattice of one parameter subgroups. 
We can extend this functional to the whole $\mathfrak{t}$ continuously by putting 
\begin{equation}
\label{mu-entropy for vector}
\cmu^\lambda (X, L; \xi) := \cW^\lambda (\omega, \mu_\xi) = -\log \frac{\mathrm{Vol}^\lambda (-\xi/2)}{(e^n. n!)^\lambda} 
\end{equation}
for general $\xi \in \mathfrak{t}$, using a $\xi$-invariant K\"ahler metric $\omega \in L$ and a moment map $\mu$ which satisfies $-d\mu_\xi = i_\xi \omega$. 
As observed in \cite{Ino2}, this is independent of the choice of the $\xi$-invariant K\"ahler metric $\omega \in L$. 
By this description, we obviously have $\cmu^\lambda (X, L; \xi) \le \cmu^\lambda (\omega)$ for any $\xi$-invariant K\"ahler metric $\omega \in L$. 
We will see this indeed holds for non-invariant general $\omega \in L$. 

\subsubsection{Main results}

Here we collect the main results of this article. 
The first result is about the characterization of $\mu^\lambda$-cscK metrics in terms of $W$-entropy. 

\begin{thm}
The following are equivalent for a state $(\omega, f) \in T\mathcal{H} (X, L)$. 
\begin{itemize}
\item The vector field $\partial^\sharp f$ is holomorphic and the metric $\omega$ is a $\mu^\lambda$-cscK metric with respect to $\xi = \mathrm{Im} \partial^\sharp f$. 

\item The state $(\omega, f)$ is a critical point of $\cW^\lambda: T \mathcal{H} (X, L) \to \mathbb{R}$. 
\end{itemize}
\end{thm}

When $\lambda \le 0$, we can reinterpret this variational result in terms of $\mu$-entropy. 

\begin{thm}
\label{main theorem on mu-entropy}
When $\lambda \le 0$, the $\mu$-entropy $\cmu^\lambda$ is smooth on $\mathcal{H} (X, L)$. 
In this case, the following are equivalent for a K\"ahler metric $\omega \in \mathcal{H} (X, L)$. 
\begin{enumerate}
\item[(a)] $\omega$ is a $\mu^\lambda$-cscK metric. 

\item[(b)] $\omega$ is a minimizer of $\cmu^\lambda: \mathcal{H} (X, L) \to \mathbb{R}$. 

\item[(b')] $\omega$ is a critical point of $\cmu^\lambda: \mathcal{H} (X, L) \to \mathbb{R}$. 

\item[(d)] There is $\xi$ such that $\cmu^\lambda (X, L; \xi) = \cmu^\lambda (\omega)$. 
\end{enumerate}
\end{thm}

We firstly give a proof for the implication (a) $\iff$ (b') $\iff$ (d) $\Longleftarrow$ (b) in section \ref{mu-cscK metrics and Perelman's $W$-entropy}. 
The rest implication (a) $\Longrightarrow$ (b) is concluded in section \ref{W-functional along geodesic rays} as a consequence of the following. 


\begin{thm}
\label{inequality}
For $\lambda \in \mathbb{R}$, we have 
\[ \sup_{(\mathcal{X}, \mathcal{L}), \rho \ge 0} \cmu^\lambda (\mathcal{X}, \mathcal{L}; \rho) \le \inf_{\omega_\varphi \in \mathcal{H} (X, L)} \cmu^\lambda (\omega_\varphi), \]
where $(\mathcal{X}, \mathcal{L}; \rho)$ runs over all test configurations. 
\end{thm}

This inequality is analogous to Donaldson's lower bound \cite{Don} (resp. \cite{He, DS}) for Calabi functional (resp. $H$-functional). 
These are indeed related by extremal limit $\lambda \to - \infty$ as explained in section \ref{Observations, Calabi}. 
We will reformulate this result in the second article \cite{Ino4}. 

The first two theorems are proved in section \ref{mu-cscK metrics and Perelman's $W$-entropy}. 
The last theorem is proved in section \ref{Slope formula}, based on the monotonicity established in section \ref{Monotonicity}. 
We employ Berman--Berndtsson's subharmonicity argument \cite{BB} there. 

\subsubsection{Interpretation of the result: Minimax picture}

Our proof of the Theorem \ref{inequality} inspires us to interpret the result as a minimax theorem on Perelman's $W$-entropy $\check{W}^\lambda (\omega, f)$. 
It is then natural to regard test configurations as dual concepts of K\"ahler metrics. 
As proposed in \cite{BHJ1}, one can identify test configurations with non-archimedean metrics. 
This perspective provides us the following understanding: $W$-entropy gives a duality between K\"ahler metrics and non-archimedean metrics. 

With this in mind, let us rewrite the inequality in Theorem \ref{inequality} as follows: 
\begin{equation}
\label{minimax inequality} 
\sup_{\varphi \in \nH (X, L)} \NAmu^\lambda (\varphi) \le \inf_{\omega_\phi \in \mathcal{H} (X, L)} \cmu^\lambda (\omega_\phi). 
\end{equation}
Then the following corollaries can be interpreted as a duality. 

\begin{cor}
Suppose $\lambda \le 0$. 
If $\omega$ is a $\check{\mu}^\lambda_\xi$-cscK metric for some $\xi$, then 
\[ \cmu^\lambda (\omega) = \NAmu^\lambda (\varphi_\xi) = \sup_{\varphi \in \nH (X, L)} \NAmu^\lambda (\varphi) = \inf_{\omega_\phi \in \mathcal{H} (X, L)} \cmu^\lambda (\omega_\phi), \]
for the associated non-archimedean metric $\varphi_\xi$. 
\end{cor}

\begin{proof}
If $\omega$ is a $\check{\mu}^\lambda_\xi$-cscK metric for $\lambda \le 0$, then by the second main theorem, we get $\cmu^\lambda (\omega) = \NAmu^\lambda (\varphi_\xi)$. 
Thus by the inequality (\ref{minimax inequality}), we obtain 
\[ \cmu^\lambda (\omega) = \NAmu^\lambda (\varphi_\xi) \le \sup_{\varphi \in \nH (X, L)} \NAmu^\lambda (\varphi) \le \inf_{\omega_\phi \in \mathcal{H} (X, L)} \cmu^\lambda (\omega_\phi) \le \cmu^\lambda (\omega). \]
\end{proof}

\begin{cor}
\label{mu-cscK characterization}
Suppose $\lambda \le 0$. 
The following are equivalent for a K\"ahler metric $\omega \in \mathcal{H} (X, L)$. 
\begin{enumerate}
\item[(a)] $\omega$ is a $\mu^\lambda$-cscK metric. 

\item[(b)] $\cmu^\lambda (\omega) = \inf_{\omega_\varphi \in \mathcal{H} (X, L)} \cmu^\lambda (\omega_\varphi)$. 

\item[(c)] $\cmu^\lambda (\omega) = \sup_{\phi \in \pcH (X, L)} \NAmu^\lambda (\phi)$. 

\item[(d)] $\cmu^\lambda (\omega) = \cmu^\lambda_{\mathrm{NA}} (\phi_\xi)$ for some vector $\xi$. 
\end{enumerate}
\end{cor}

\begin{proof}
As we already see (a) $\iff$ (b) $\iff$ (d), it suffices to show (c) $\Rightarrow$ (b) and (d) $\Rightarrow$ (c). 
The implication (c) $\Rightarrow$ (b) follows by the inequality (\ref{minimax inequality}): 
\[ \inf_{\omega_\varphi \in \mathcal{H} (X, L)} \cmu^\lambda (\omega_\varphi) \le \cmu^\lambda (\omega) = \sup_{\phi \in \pcH (X, L)} \NAmu^\lambda (\phi) \le \inf_{\omega_\varphi \in \mathcal{H} (X, L)} \cmu^\lambda (\omega_\varphi). \]
As for (d) $\Rightarrow$ (c), we again use the inequality (\ref{minimax inequality}): 
\[ \sup_{\phi \in \pcH (X, L)} \NAmu^\lambda (\phi) \le \cmu^\lambda (\omega) = \NAmu^\lambda (\phi_\xi) \le \sup_{\phi \in \pcH (X, L)} \NAmu^\lambda (\phi). \]
\end{proof}

Now we explain the minimax picture. 
The following diagram illustrates the relation of Perelman's $W$-entropy and $\mu$-entropies for K\"ahler metrics and test configurations. 
\[ \begin{tikzcd} 
~
& \check{W}^\lambda (\omega_\phi, \varphi) \ar{dr}{\sup_\varphi} \ar[swap]{dl}{\inf_{\omega_\phi}}
&~
\\
\NAmu^\lambda (\varphi) 
& ~
& \bm{\check{\mu}}^\lambda (\omega_\phi)
\end{tikzcd} \]
Here we define $\check{W}^\lambda (\omega_\phi, \varphi)$ as follows. 
For $\omega_\phi \in \mathcal{H} (X, L)$ and $\varphi \in \nH (X, L)$, we have a unique $C^{1,1}$-geodesic ray $\omega_{\phi_t} \in \mathcal{H}^{1,1} (X, L)$ subordinate to $\varphi$ emanating from $\omega_\phi$. 
We then put 
\[ \check{W}^\lambda (\omega_\phi, \varphi) := \check{W}^\lambda (\omega_\phi, \dot{\phi}_0). \]

By the definition of $\bm{\check{\mu}}^\lambda (\omega_\phi)$, we have 
\[ \bm{\check{\mu}}^\lambda (\omega_\phi) \ge \sup_{\varphi \in \nH (X, L)} \check{W}^\lambda (\omega_\phi, \varphi). \]
On the other hand, as we will prove, by the monotonicity and the slope formula, we have $\NAmu^\lambda (\varphi) \le \check{W}^\lambda (\omega_\phi, \varphi)$ for $\varphi \in \nH (X, L)$ associated to a smooth test configuration $(\mathcal{X}, \mathcal{L})$. 
Thus we have 
\[ \NAmu^\lambda (\varphi) \le \inf_{\omega_\phi \in \nH (X, L)} \check{W}^\lambda (\omega_\phi, \varphi). \]
Since we obviously have 
\[ \sup_{\varphi \in \mathcal{H} (X, L)} \inf_{\omega_\phi \in \nH (X, L)} \check{W}^\lambda (\omega_\phi, \varphi) \le \inf_{\omega_\phi \in \mathcal{H} (X, L)} \sup_{\varphi \in \nH (X, L)} \check{W}^\lambda (\omega_\phi, \varphi), \]
we obtain 
\[ \sup_{\varphi \in \nH (X, L)} \NAmu^\lambda (\varphi) \le \inf_{\omega_\phi \in \mathcal{H} (X, L)} \bm{\check{\mu}}^\lambda (\omega_\phi). \]
Therefore, Theorem \ref{inequality} can be interpreted as a minimax inequality. 
This interpretation leads us to the following conjecture. 

\begin{conj}
We have the following equalities: 
\begin{gather*} 
\bm{\check{\mu}}^\lambda (\omega_\phi) = \sup_{\varphi \in \mathcal{H} (X, L)} \check{W}^\lambda (\omega_\phi, \varphi), 
\\
\NAmu^\lambda (\varphi) = \inf_{\omega_\phi \in \nH (X, L)} \check{W}^\lambda (\omega_\phi, \varphi). 
\end{gather*}
Moreover, for $\lambda \le 0$, we have 
\[ \sup_{\varphi \in \nH (X, L)} \NAmu^\lambda (\varphi) = \inf_{\phi \in \mathcal{H} (X, L)} \bm{\check{\mu}}^\lambda (\omega_\phi). \]
\end{conj}

This is analogous to Sion's minimax theorem: let $C$ be a compact convex subset of a topological vector space and $H$ be a convex subset of a topological vector space. 
Suppose a function $W: H \times C \to \mathbb{R}$ satisfies the following conditions: 
\begin{itemize}
\item For every $\varphi \in C$, $W (\cdot, \varphi): H \to \mathbb{R}$ is lsc and quasi-convex. 

\item For every $\phi \in H$, $W (\phi, \cdot): C \to \mathbb{R}$ is usc and quasi-concave. 
\end{itemize}
Then we have 
\[ \max_{\varphi \in C} \inf_{\phi \in H} W (\phi, \varphi) = \inf_{\phi \in H} \max_{\varphi \in C} W (\phi, \varphi). \]

\subsubsection{Conventions for K\"ahlerian tensor calculus}
\label{convention}

Here we fix our convention for K\"ahlerian tensor calculus used throughout this series. 

Let $X$ be a complex manifold. 
We put $d^c := \frac{\sqrt{-1}}{2} (\bar{\partial} - \partial)$ so that it is a real operator satisfying $dd^c = \sqddbar$. 
Usually, $d^c$ is divided by $\pi$ or $2\pi$ from ours, though, we prefer to use our convention as this difference affects to the geodesic equation and hence to our computation on equivariant integration: for $a \in \mathbb{R}$, a (smooth) path of K\"ahler metrics $\omega + d (a. d^c) \phi_t$ is a geodesic iff $\ddot{\phi}_t - a. |\bar{\partial} \dot{\phi}_t|^2 = 0$. 

When we consider $(1,1)$-forms in a cohomology class $[\alpha]$, we often identify a smooth function $\varphi$ and the associated $(1,1)$-form $\alpha_\varphi = \alpha + \sqddbar \varphi = \alpha + dd^c \varphi$. 
We use the notation ``$\varphi \in C^\infty (X, \alpha)$'' to clarify this stance: ``we are now identifying $\varphi \in C^\infty (X)$ with the $(1,1)$-form $\alpha_\varphi$''. 

As usual, for a K\"ahler form $\omega$, we denote by $\mathcal{H} (X, \omega)$ the set of smooth functions satisfying $\omega_\varphi > 0$. 
We often denote by $L$ a K\"ahler class and by $\mathcal{H} (X, L)$ the set of K\"ahler metrics in $L$, which is sometimes implicitly assumed to be integral to simplify arguments. 

Just in order not to make a confusion, we avoid the following simple convention familiar in pluri-potential context: denote by $(1,1)$-form $dd^c \varphi$ for $\varphi$ representing a collection of local potentials $\varphi = \{ \varphi_\alpha \}_\alpha$ or a metric $h = e^{-\varphi}$ on a line bundle. 
Instead, we always fix a reference $\alpha$ and use the notation $dd^c_\alpha \varphi := \alpha_\varphi = \alpha + dd^c \varphi$ for $\varphi \in C^\infty (X, \alpha)$. 

A \textit{real holomorphic vector field} is a real vector field $\xi$ such that $\xi^J := J\xi + \sqrt{-1} \xi$ is holomorphic ($\Leftrightarrow L_\xi J =0$). 
Let $\mathfrak{h} (X)$ denote the space of real holomorphic vector fields. 
We put 
\begin{align}
\mathfrak{h}_0 (X) &:= \{ \xi \in \mathfrak{h} (X) ~|~ \exists \theta \in C^\infty_{\mathbb{C}} (X) \text{ s.t. } \xi = \mathrm{Im} \partial^\sharp \theta \}, 
\\
\mathfrak{h}_c (X, L) &:= \{ \xi \in \mathfrak{h}_0 (X) ~|~ \xi \in \mathfrak{t} \text{ for a closed torus } T \subset \mathrm{Aut} (X, L) \}, 
\\
{^\nabla \mathfrak{isom}} (X, \omega) &:= \{ \xi \in \mathfrak{h}_0 (X) ~|~ \exists \theta \in C^\infty_{\mathbb{R}} (X) \text{ s.t. } \xi = \mathrm{Im} \partial^\sharp \theta \}. 
\end{align}
The subspaces $\mathfrak{h}_0 (X), {^\nabla \mathfrak{isom}} (X, \omega) \subset \mathfrak{h}$ are linear. 
On the other hand, $\mathfrak{h}_c (X, L)$ is not linear but just the orbit of a linear subspace by the conjugate action. 
Since ${^\nabla \mathfrak{isom}} (X, \omega) \subset \mathfrak{isom} (X, \omega)$, we have ${^\nabla \mathfrak{isom}} (X, \omega) \subset \mathfrak{h}_c (X, L)$. 
In this series, we take a vector from $\mathfrak{h}_c (X, L)$ unless we specify the domain. 

For the complex vector fields
\[ \partial^\sharp f = g^{i \bar{\jmath}} f_{\bar{\jmath}} \frac{\partial}{\partial z^i} = \frac{1}{2} (\nabla f - \sqrt{-1} J\nabla f), \quad \bar{\partial}^\sharp f = g^{i \bar{\jmath}} f_i \frac{\partial}{\partial \bar{z}^{\jmath}} = \frac{1}{2} (\nabla f + \sqrt{-1} J\nabla f), \]
we frequently use the following fact: for $f \in C^\infty_\mathbb{R} (X)$ and $u, v \in C^\infty_{\mathbb{C}} (X)$, we have 
\begin{align*} 
\int_X (\bar{\Box} - \bar{\partial}^\sharp f) u. \bar{v} ~e^f \omega^n 
&= \int_X (\bar{\partial} u, \bar{\partial} v) e^f \omega^n 
= \int_X u. (\Box - \partial^\sharp f) \bar{v} ~e^f \omega^n, 
\end{align*}
where $(, )$ denotes the hermitian metric associated to the metric $\omega$.

\subsubsection*{Acknowledgement}

I wish to thank Tomoyuki Hisamoto for his interest and frequent helpful discussions since I had been in Kyoto. 
I would like to express my deep gratitude to Abdellah Lahdili for his interest and stimulating discussions, which motivates me greatly and accelerates this study. 
I would express my sincere gratitude to Laszlo Lempert for his interest and insightful comments on `Lagrangian experiment'. 
I am grateful to Ruadha\'i Dervan, Akito Futaki, Genki Hosono, Chi Li, Yuji Odaka for their interest, helpful comments and discussions. 
This work is supported by JSPS KAKENHI Grant Number 18J22154, MEXT, Japan.

\section{Perelman's $\mu$-entropy and $\mu$-cscK metrics}

\subsection{Perelman's $W$-entropy and $\mu$-entropy}
\label{Perelman's W-functional and mu-entropy}

We firstly study $\cW^\lambda$ restricted to the fibre direction $T_\omega \mathcal{H} \subset T \mathcal{H}$. 

\subsubsection{Variational formula on critical momentum}

We put 
\[ s^\lambda_f (\omega) := (s (\omega)+ \bar{\Box} f) + (\bar{\Box} f - |\partial^\sharp f|^2) -\lambda f \]
for $f \in C^\infty (X)$ and 
\[ \bar{s}^\lambda_f (\omega) := \int_X (s (\omega) +|\bar{\partial}^\sharp f|^2 -\lambda f) e^f \omega^n \Big{/} \int_X e^f \omega^n \]
for $f \in C^{0,1} (X)$. 
We note $\bar{s}^\lambda_f (\omega) = \int_X s^\lambda_f (\omega) e^f \omega^n / \int_X e^f \omega^n$ for $f \in C^\infty (X)$. 

The following is firstly observed in \cite{Per} and in \cite{TZ2} for Fano manifold. 

\begin{lem}
\label{Variational formula}
Fix a smooth K\"ahler metric $\omega$ on $X$. 
We have the following. 
\begin{enumerate}
\item A momentum $f \in C^{0,1} (X)$ is a critical point of $\cW^\lambda (\omega, \cdot)$ iff it is smooth and satisfies $s^\lambda_f (\omega) = \bar{s}^\lambda_f (\omega)$. 

\item If $f \in C^\infty (X)$ is a critical point of $\cW^\lambda (\omega, \cdot)$, then we have 
\[ \frac{d^2}{dt^2}\Big{|}_{t=0} \cW^\lambda (\omega, f + t u) = -\int_X (2|\partial^\sharp u|^2 - \lambda u^2) e^f \omega^n \Big{/} \int_X e^f \omega^n \]
for every $u \in C^\infty (X)$ with $\int_X u e^f \omega^n = 0$, which is negative if $\lambda$ is less than the first eigenvalue of $\Delta - \nabla f$, especially when $\lambda \le 0$. 
\end{enumerate}
\end{lem}

\begin{proof}
We compute 
\begin{align*} 
\frac{d}{dt}\Big{|}_{t=0} \cW^\lambda (\omega, f + t u) 
&= - \int_X \Big{(} 2(\bar{\partial}^\sharp f, \bar{\partial}^\sharp u) + (s(\omega) + |\bar{\partial}^\sharp f|^2 -\lambda f ) u \Big{)} e^f \omega^n \Big{/} \int_X e^f \omega^n 
\\
&\qquad+ \bar{s}^\lambda_f \int_X u e^f \omega^n \Big{/} \int_X e^f \omega^n, 
\end{align*}
which can be arranged as 
\[ - \int_X (s^\lambda_f (\omega) - \bar{s}^\lambda_f (\omega)) u ~e^f \omega^n \Big{/} \int_X e^f \omega^n \]
when $f$ is smooth, using integration by parts. 
For a less regular critical point $f \in C^{0,1} (X)$, we have 
\[ (s (\omega) + \bar{\Box} f) + (\bar{\Box} f - |\partial^\sharp f|^2) - \lambda f = \bar{s}^\lambda_f \]
in the distributional sense. 
The elliptic bootstrap argument shows that $f$ is indeed smooth and satisfies $s^\lambda_f (\omega) = \bar{s}^\lambda_f (\omega)$ in the usual sense. 

For the second claim, we firstly compute 
\begin{align*} 
\frac{d}{dt}\Big{|}_{t=0} \bar{s}^\lambda_{f + t u} 
&= \left( \int_X (\bar{\Box} u - \lambda u) e^ f \omega^n + \int_X (s (\omega) + \bar{\Box} f - \lambda f) u ~ e^f \omega^n \right) \Big{/} \int_X e^f \omega^n - \bar{s}^\lambda_f \int_X u e^f \omega^n \Big{/} \int_X e^f \omega^n 
\\
&= \int_X (s^\lambda_f - \bar{s}^\lambda_f) u ~ e^f \omega^n \Big{/} \int_X e^f \omega^n - \lambda \int_X u e^f \omega^n \Big{/} \int_X e^f \omega^n. 
\end{align*}
We exhibit the second variation for general $u$ as an independent interest: 
\begin{align*} 
-\frac{d^2}{dt^2}\Big{|}_{t=0} 
&\cW^\lambda (\omega, f + t u) = \frac{d}{dt}\Big{|}_{t=0} \int_X (s^\lambda_{f + t u} - \bar{s}^\lambda_{f + t u}) u ~ e^{f + t u} \omega^n \Big{/} \int_X e^{f + t u} \omega^n 
\\
&= \left( \int_X (2 \bar{\Box} u - 2 \bar{\partial}^\sharp f (u) - \lambda u) u ~e^f \omega^n + \int_X s^\lambda_f u^2 ~e^f \omega^n \right) \Big{/} \int_X e^f \omega^n
\\
&\qquad - \int_X s^\lambda_f u ~e^f \omega^n \Big{/} \int_X e^f \omega^n \cdot \int_X u~ e^f \omega^n \Big{/} \int_X e^f \omega^n 
\\
&\quad - \frac{d}{dt}\Big{|}_{t=0} \bar{s}^\lambda_{f + t u} \int_X u ~e^f \omega^n \Big{/} \int_X e^f \omega^n
\\
&\qquad - \bar{s}^\lambda_f \int_X u^2 ~e^f \omega^n \Big{/} \int_X e^f \omega^n + \bar{s}^\lambda_f  \Big{(} \int_X u ~e^f \omega^n \Big{/} \int_X e^f \omega^n \Big{)}^2
\\
&= \int_X \Big{(} 2 |\bar{\partial}^\sharp u|^2 + (s^\lambda_f - \bar{s}^\lambda_f - \lambda) (u - \bar{u}_f)^2 \Big{)} ~e^f \omega^n \Big{/} \int_X e^f \omega^n 
\end{align*}
where we put $\bar{u}_f = \int_X u e^f \omega^n / \int_X e^f \omega^n$. 
This proves the claim by the first result. 
\end{proof}

\subsubsection{Maximal momentum}

The following is due to \cite{Rot}, which covers the case $\lambda > 0$ but also works for the case $\lambda \le 0$ with a minor change. 
Here we draw the proof to clarify the difference and to make our arguments self-contained. 
The case $\lambda \le 0$ is slightly simpler than the case $\lambda > 0$. 

\begin{thm}
For each $\lambda \in \mathbb{R}$ and $\omega \in \mathcal{H} (X, L)$, there exists a smooth function $f \in C^\infty (X)$ which attains the maximum of the functional $\cW^\lambda (\omega, \cdot): C^{0, 1} (X) \to \mathbb{R}$. 
In fact, every maximizer is smooth. 
\end{thm}

\begin{proof}
We consider the following functional: 
\[ \mathcal{L}^\lambda (u) := - \int_X \big{(} s (\omega) u^2 + 2 |\nabla u|^2 - \lambda n u^2 - \lambda u^2 \log u^2 \big{)} \omega^n \Big{/} \int_X u^2 \omega^n - \lambda \log \Big{(} \frac{1}{n!} \int_X u^2 \omega^n \Big{)}. \]
We have $\cW^\lambda (\omega, f) = \mathcal{L}^\lambda (e^{f/2})$ and $\mathcal{L}^\lambda (c u) = \mathcal{L}^\lambda (u)$ for $c \in \mathbb{R}^\times$. 
We can see as follows that $\mathcal{L}^\lambda$ gives a continuous function on $L^2_1 (X) \setminus 0$. 
For $u, v \neq 0$, we have a measurable function $\theta: X \to [0,1]$ such that 
\[ u^2 \log u^2 - v^2 \log v^2 = \Big{(} 2 u_\theta + 4 u_\theta \log u_\theta \Big{)} (|u| - |v|) \]
for $u_\theta = \theta |u| + (1- \theta) |v|$ by the mean value theorem and the measurable selection theorem. 
Then we have 
\begin{align*} 
\left| \int_X u^2 \log u^2 \omega^n - \int_X v^2 \log v^2 \omega^n \right| 
&\le \int_X \left| \Big{(} 2 u_\theta + 4 u_\theta \log u_\theta \Big{)} (|u| - |v|) \right| \omega^n
\\
&\le \int_X \Big{(} 2 u_\theta + 4 \max (e^{-1}, C u_\theta^{1+\delta}) \Big{)} \Big{|}|u| - |v|\Big{|} \omega^n
\\
&\le C' \max (e^{-1}, \| u \|_{L^{2+ 2\delta}}^{1+\delta}, \| v \|_{L^{2+ 2\delta}}^{1+\delta} ) \Big{\|} |u| - |v| \Big{\|}_{L^2}, 
\end{align*}
where we use the convexity of $x^{1+\delta}$ and H\"older's inequality in the last line. 
It follows by the Sobolev embedding $L^2_1 \hookrightarrow L^{2n/(n-1)} \subset L^{2+2 \delta}$ that the functional $\mathcal{L}^\lambda$ is continuous on the sphere $\{ u \in L^2_1 (X) ~|~ \| u \|_{L^2} = 1 \}$, hence is continuous on $L^2_1 (X) \setminus 0$ by the scaling invariance. 

To obtain the claim, we firstly see that there exists a non-negative maximizer $u$ of $\mathcal{L}^\lambda$ and then show that $u$ is a strictly positive smooth function, by which we obtain a smooth maximizer $f = 2 \log u$ of $\cW^\lambda$. 

Firstly, the functional $\mathcal{L}^\lambda$ is bounded from above. 
To see this, we may normalize $u$ as $\| u \|_{L^2} =1$. 
When $\lambda \le 0$, the boundedness follows by Jensen's inequality applied for $\varphi (x) = x \log x$: 
\[ \lambda \int_X u^2 \log u^2 \omega^n \le \lambda V \varphi (V^{-1}). \]
When $\lambda > 0$, we can bound 
\[ - \int_X (|\nabla u|^2 - \lambda u^2 \log u^2) \omega^n \]
from above as in \cite{Rot}. 

Secondly, we construct a non-negative maximizer $u$. 
A similar argument as above shows that we have a uniform constant $c$ such that 
\[ \int_X |\nabla u|^2 \omega^n \le -\mathcal{L}^\lambda (u) + c \]
for every $u$ with $\| u \|_{L^2} = 1$. 
Take a sequence $u_i \in L^2_1$ so that $\mathcal{L}^\lambda (u_i) \to \sup \mathcal{L}^\lambda$ and $\| u_i \|_{L^2} =1$. 
Then $u_i$ is bounded in $L^2_1$-norm by the above inequality. 
By Sobolev embedding, we can take a subsequence $u_i$ converging weakly to $u$ in $L^2_1$ and strongly in $L^{2n/(n-1) - \epsilon}$. 
The weak convergence implies $\liminf \int_X |\nabla u_i |^2 \omega^n \ge \int_X |\nabla u |^2 \omega^n$, so we get 
\[ \sup \mathcal{L}^\lambda = \lim \mathcal{L}^\lambda (u_i) \le \mathcal{L}^\lambda (u). \]
Thus the limit $u$ gives a maximizer of $\mathcal{L}^\lambda$. 
Since $|\nabla |u|| \le |\nabla u|$ in general, we obtain a non-negative maximizer. 

Since $\mathscr{L}^\lambda (t) = \mathcal{L}^\lambda (u+t \varphi)$ is $C^1$ for smooth $\varphi$, for the maximizer $u$, we have 
\begin{align*} 
0 
&= D_u \mathcal{L}^\lambda (\varphi) 
\\
&=- 2 \int_X \Big{(} (s (\omega) u - \lambda u \log u^2) \varphi + 2(\nabla u, \nabla \varphi) \Big{)} \omega^n \Big{/} \int_X u^2 \omega^n 
\\
&\qquad + 2 \bar{s}^\lambda_{(u)} \cdot \int_X u \varphi \omega^n \Big{/} \int_X u^2 \omega^n, 
\end{align*}
where we put 
\[ \bar{s}^\lambda_{(u)} := \int_X \big{(} s (\omega) u^2 + 2 |\nabla u|^2 - \lambda u^2 \log u^2 \big{)} \omega^n \Big{/} \int_X u^2 \omega^n. \]
This gives us the following distributional differential equation on $u$: 
\[ s (\omega) u - \lambda u \log u^2 + 2 \Delta u - \bar{s}^\lambda_{(u)} u = 0. \]
Then as in \cite{Rot}, we can see the maximizer $u$ is indeed in $C^{2,1} (X)$ and satisfies the above differential equation in the ordinary sense. 

Now we assume $u$ takes zero at a point $p \in X$. 
We will show $u$ must be zero around $p$ and hence is zero on $X$ by the connectedness, which makes a contradiction to the fact $\| u \|_{L^2} =1$ and thus implies that the assumption was absurd. 
Once this is proved, we easily see the smoothness of $f = 2\log u$ by the elliptic bootstrap argument, so that we get the desired claim. 
Take a normal polar coordinate $(r, \theta) = (r, \theta_1, \ldots, \theta_{2n-1})$ at $p$. 
We put $A (r, \theta) := \omega^n/dr \wedge d\theta$, $S (r) = \int A (r, \theta) d\theta$ and $\sigma (r, \theta) := A (r, \theta) /S(r)$. 
We have $(\partial/\partial r) \log \sigma (r, \theta) \le C r$ as noted in \cite{Rot}. 
It suffices to see the function 
\[ F (r) := \int u (r, \theta) \sigma (r, \theta) d\theta \]
is zero on $r \in (0, R)$ for sufficiently small $R$. 
To see this, we set up an induction on the decay rate: if we have $F (r) \le r^k$ on $(0, R)$, then we obtain $F (r) \le r^{k+1/2}$. 
We put 
\begin{align*}
G (r) 
&:= \int \frac{\partial}{\partial r} u(r, \theta) \sigma (r, \theta) d\theta, 
\\
L (r) 
&:= \int u \log u^2 \sigma d\theta, 
\\
K (r)
&:= \int (s (\omega) u - \bar{s}^\lambda_{(u)} u) \sigma d\theta. 
\end{align*}
For a smooth function $\varphi (r, \theta) = \varphi (r)$ which is compactly supported on the geodesic ball $B (p, R)$ and depends only on the variable $r$, we have 
\[ D_u \mathcal{L}^\lambda (\varphi) = \int_0^R \Big{(} 2 G (r) S (r) \frac{d}{dr} \varphi (r) - \lambda L (r) S (r) \varphi (r) + K (r) S (r) \varphi (r) \Big{)} dr = 0 \]
by $\int (\nabla u, \nabla \varphi) \omega^n = \int_0^R G (r) S (r) (d/dr) \varphi (r) dr$. 
Running all possible $\varphi$, we get 
\[ 2\frac{d}{dr} (G (r) S (r)) = K (r) S (r) - \lambda L (r) S (r) \]
for $r \in (0, R)$. 
Since $u (p) = 0$ and $u$ is non-negative and continuous, we may assume $u \log u^2 \le 0$ on a small ball $B (p, R)$. 
So when $\lambda \le 0$, we obtain 
\[ 2\frac{d}{dr} (G (r) S (r)) \le K (r) S (r) \le C' F (r) S (r). \]
Therefore, we get 
\begin{align*} 
F (r) 
&= \int_0^r F' (s) ds = \int_0^r G (s) ds + \int_0^r ds \int u (s, \theta) \Big{(} \frac{\partial}{\partial s} \log \sigma (s, \theta) \Big{)} \sigma (s, \theta) d\theta 
\\
&\le C' \int_0^r ds S (s)^{-1} \int_0^s F (t) S (t) dt + C \int_0^r s F (s) ds
\\
&\le C'' \left( \int_0^r \frac{ds}{s^{n-1}} \int_0^s F (t) t^{n-1} dt + \int_0^r s F (s) ds \right), 
\end{align*}
which provides us a machinery for running the induction as desired. 
We refer \cite{Rot} for the rest detail. 
\end{proof}

\subsection{$\mu$-cscK metrics and Perelman's $W$-entropy}
\label{mu-cscK metrics and Perelman's $W$-entropy}

In this subsection we show that the critical points of $\cW^\lambda$ precisely correspond to $\mu^\lambda$-cscK metrics as stated in the first main theorem.

\subsubsection{Variational formula on state}

\begin{thm}
\label{Critical points of W}
Let $f$ be a critical point of the functional $\cW^\lambda (\omega, \cdot)$. 
For a perturbation $\omega_t = \omega + \sqrt{-1} \partial \bar{\partial} \phi_t$ with $\dot{\phi}_0 = \varphi$, we have 
\[ \frac{d}{dt}\Big{|}_{t=0} \cW^\lambda (\omega_t, f) = \int_X \mathrm{Re} (\mathcal{D} \varphi, \mathcal{D} f) e^f \omega^n \Big{/} \int_X e^f \omega^n. \]
\end{thm}

We prepare some convenient formulas. 

\begin{lem}
We have 
\begin{align*}
(\sqrt{-1} \partial \bar{\partial} f, \sqrt{-1} \partial \bar{\partial} \varphi) 
&= - \bar{\Box} (\bar{\partial}^\sharp f (\varphi)) + \bar{\partial}^\sharp f (\bar{\Box} \varphi) - g^{k \bar{l}} ((g^{i \jbar} f_i)_k \varphi_{\jbar})_{\bar{l}}, 
\\
(\sqrt{-1} \partial f \wedge \bar{\partial} f, \sqrt{-1} \partial \bar{\partial} \varphi) 
&= \bar{\partial}^\sharp f (\partial^\sharp f (\varphi)) - g^{i \jbar} f_i (g^{k \bar{l}} f_{\bar{l}})_{\jbar} \varphi_k
\end{align*}
and 
\begin{align*} 
\mathcal{D}^{f *} \mathcal{D} \varphi 
&= (\bar{\Box} - \bar{\partial}^\sharp f)^2 \varphi + (\Ric (\omega) - \sqrt{-1} \partial \bar{\partial} f, \sqrt{-1} \partial \bar{\partial} \varphi) + (\bar{\partial}^\sharp s^0_f (\omega), \nabla \varphi) 
\\
&\qquad + \Big{(} g^{i \jbar} (g^{k \bar{l}} f_k)_{i \bar{l}} \varphi_{\jbar} + g^{i \jbar} g^{k \bar{q}} g_{k \bar{l}, \bar{q}} (g^{p \bar{l}} f_p)_i \varphi_{\jbar} + g^{i \jbar} (g^{k \bar{l}} f_k)_i f_{\bar{l}} \varphi_{\jbar} \Big{)}. 
\end{align*}
\end{lem}

\begin{proof}
Using $g^{k \bar{l}} {g^{i \jbar}}_{, \bar{l}} = {g^{k \jbar}}_{, \bar{l}} g^{i \bar{l}}$, we compute 
\begin{align*}
(\sqrt{-1} \partial \bar{\partial} f, \sqrt{-1} \partial \bar{\partial} \varphi) 
&= g^{i \jbar} g^{k \bar{l}} f_{i \bar{l}} \varphi_{k \jbar}
\\
&= g^{k \bar{l}} (g^{i \jbar} f_i \varphi_{k \jbar})_{\bar{l}} - g^{k \bar{l}} f_i (g^{i \jbar} \varphi_{k \jbar})_{\bar{l}}
\\
&= g^{k \bar{l}} (g^{i \jbar} f_i \varphi_{\jbar})_{k \bar{l}} - g^{k \bar{l}} ((g^{i \jbar} f_i)_k \varphi_{\jbar})_{\bar{l}} - g^{k \bar{l}} f_i (g^{i \jbar} \varphi_{k \jbar})_{\bar{l}}
\\
&= - \bar{\Box} (\bar{\partial}^\sharp f (\varphi)) - g^{k \bar{l}} ((g^{i \jbar} f_i)_k \varphi_{\jbar})_{\bar{l}} - g^{k \bar{l}} f_i ({g^{i \jbar}}_{, \bar{l}} \varphi_{k \jbar} + g^{i \jbar} \varphi_{k \jbar \bar{l}})
\\
&= - \bar{\Box} (\bar{\partial}^\sharp f (\varphi)) - g^{k \bar{l}} ((g^{i \jbar} f_i)_k \varphi_{\jbar})_{\bar{l}} - g^{k \bar{l}} f_i ({g^{i \jbar}}_{, \bar{l}} \varphi_{k \jbar} + (g^{i \jbar} \varphi_{k \bar{l}})_{\jbar} - {g^{i \jbar}}_{, \jbar} \varphi_{k \bar{l}})
\\
&= - \bar{\Box} (\bar{\partial}^\sharp f (\varphi)) - g^{k \bar{l}} ((g^{i \jbar} f_i)_k \varphi_{\jbar})_{\bar{l}} - g^{k \bar{l}} f_i ({g^{i \jbar}}_{, \bar{l}} \varphi_{k \jbar} + (g^{i \jbar} \varphi_{k \bar{l}})_{\jbar}) 
\\
&\qquad+ f_i (g^{k \bar{l}} g^{i \jbar} \varphi_{k \bar{l}})_{\jbar} - f_i g^{i \jbar} (g^{k \bar{l}} \varphi_{k \bar{l}})_{\jbar}
\\
&= - \bar{\Box} (\bar{\partial}^\sharp f (\varphi)) + \bar{\partial}^\sharp f (\bar{\Box} \varphi) - g^{k \bar{l}} ((g^{i \jbar} f_i)_k \varphi_{\jbar})_{\bar{l}} - g^{k \bar{l}} f_i {g^{i \jbar}}_{, \bar{l}} \varphi_{k \jbar} 
\\
&\qquad - g^{k \bar{l}} f_i  (g^{i \jbar} \varphi_{k \bar{l}})_{\jbar} + f_i (g^{k \bar{l}} g^{i \jbar} \varphi_{k \bar{l}})_{\jbar} 
\\
&= - \bar{\Box} (\bar{\partial}^\sharp f (\varphi)) + \bar{\partial}^\sharp f (\bar{\Box} \varphi) - g^{k \bar{l}} ((g^{i \jbar} f_i)_k \varphi_{\jbar})_{\bar{l}} 
\\
&\qquad - f_i {g^{k \jbar}}_{, \bar{l}} g^{i \bar{l}} \varphi_{k \jbar} + f_i {g^{k \bar{l}}}_{, \jbar} g^{i \jbar} \varphi_{k \bar{l}} 
\\
&= - \bar{\Box} (\bar{\partial}^\sharp f (\varphi)) + \bar{\partial}^\sharp f (\bar{\Box} \varphi) - g^{k \bar{l}} ((g^{i \jbar} f_i)_k \varphi_{\jbar})_{\bar{l}}. 
\end{align*}

The second equality is simple: 
\begin{align*}
(\sqrt{-1} \partial f \wedge \bar{\partial} f, \sqrt{-1} \partial \bar{\partial} \varphi)
&= g^{i \jbar} g^{k \bar{l}} f_i f_{\bar{l}} \varphi_{k \jbar} 
\\
&= g^{i \jbar} f_i (g^{k \bar{l}} f_{\bar{l}} \varphi_k )_{\jbar} - g^{i \jbar} f_i (g^{k \bar{l}} f_{\bar{l}})_{\jbar} \varphi_k
\\
&= \bar{\partial}^\sharp f (\partial^\sharp f (\varphi)) - g^{i \jbar} f_i (g^{k \bar{l}} f_{\bar{l}})_{\jbar} \varphi_k. 
\end{align*}

As for the last one, we refer the computation in Proposition 3.4 in \cite{Ino2} for the detail. 
In that proof, we used the assumption that $\partial^\sharp f$ is holomorphic only to eliminate the last term 
\[ \Big{(} g^{i \jbar} (g^{k \bar{l}} f_k)_{i \bar{l}} \varphi_{\jbar} + g^{i \jbar} g^{k \bar{q}} g_{k \bar{l}, \bar{q}} (g^{p \bar{l}} f_p)_i \varphi_{\jbar} + g^{i \jbar} (g^{k \bar{l}} f_k)_i f_{\bar{l}} \varphi_{\jbar} \Big{)}. \]
\end{proof}

We put 
\begin{align} 
A_f (\varphi) 
&:= g^{k \bar{l}} ((g^{i \jbar} f_i)_k \varphi_{\jbar})_{\bar{l}} + g^{i \jbar} f_i (g^{k \bar{l}} f_{\bar{l}})_{\jbar} \varphi_k, 
\\ 
B_f (\varphi) 
&:= g^{i \jbar} (g^{k \bar{l}} f_k)_{i \bar{l}} \varphi_{\jbar} + g^{i \jbar} g^{k \bar{q}} g_{k \bar{l}, \bar{q}} (g^{p \bar{l}} f_p)_i \varphi_{\jbar} + g^{i \jbar} (g^{k \bar{l}} f_k)_i f_{\bar{l}} \varphi_{\jbar}, 
\end{align}
which are globally defined complex valued functions by the above lemma.

\begin{prop}
We have 
\begin{equation} 
\int_X \mathrm{Re} A_f (\varphi) e^f \omega^n = 0 
\end{equation}
and 
\begin{equation} 
\mathrm{Re} (A_f (\varphi) - B_f (\varphi)) = \mathrm{Re} (\mathcal{D} \varphi, \mathcal{D} f). 
\end{equation}
\end{prop}

\begin{proof}
Put $\omega_t := \omega + t \sqddbar \varphi$. 
To see the first claim, we compute the derivative of $\int_X (\bar{\Box}_t f - |\bar{\partial}_t^\sharp f|^2) e^f \omega_t^n = 0$. 
Using the above lemma, we compute  
\begin{align*}
0 
&= \frac{d}{dt}\Big{|}_{t=0} \int_X (\bar{\Box}_t f - |\bar{\partial}_t^\sharp f|^2) e^f \omega_t^n
\\
&= \int_X \Big{(} (\sqddbar \varphi ,\sqddbar f) + (\sqddbar \varphi, \sqrt{-1} \partial f \wedge \bar{\partial} f) - (\bar{\Box} f - |\bar{\partial}^\sharp f|^2) \bar{\Box} \varphi \Big{)} e^f \omega^n
\\
&= \int_X \Big{(} (\sqddbar \varphi ,\sqddbar f) + (\sqddbar \varphi, \sqrt{-1} \partial f \wedge \bar{\partial} f) - \bar{\partial}^\sharp f (\bar{\Box} \varphi) \Big{)} e^f \omega^n
\\
&= \int_X \Big{(} (\sqddbar \varphi ,\sqddbar f) + (\sqddbar \varphi, \sqrt{-1} \partial f \wedge \bar{\partial} f) - \bar{\Box}^2 \varphi \Big{)} e^f \omega^n
\\
&= - \int_X A_f (\varphi) e^f \omega^n - \int_X (\bar{\Box} - \bar{\partial}^\sharp f) (\bar{\Box} + \bar{\partial}^\sharp f) (\varphi) e^f \omega^n + \int_X  \bar{\partial}^\sharp f ((\partial^\sharp f -\bar{\partial}^\sharp f) (\varphi)) e^f \omega^n, 
\\
&= - \int_X A_f (\varphi) e^f \omega^n - \int_X (\bar{\Box} - \bar{\partial}^\sharp f) (\bar{\Box} + \bar{\partial}^\sharp f) (\varphi) e^f \omega^n + \int_X  \bar{\partial}^\sharp f ((\partial^\sharp f -\bar{\partial}^\sharp f) (\varphi)) e^f \omega^n
\\
&= - \int_X \mathrm{Re} A_f (\varphi) e^f \omega^n - \sqrt{-1} \Big{(} \int_X \mathrm{Im} A_f (\varphi) e^f \omega^n + \int_X \bar{\Box} (J \nabla f (\varphi)) e^f \omega^n \Big{)}, 
\end{align*}
which proves the first claim.

As for the second claim, we compare them at $p \in X$ on a normal coordinate. 
We have 
\begin{align*} 
\mathrm{Re} (A_f (\varphi) 
&- B_f (\varphi) ) (p)
\\
&=\mathrm{Re} \Big{[} \Big{(} \sum_{i, k, \jbar} {g^{i \jbar}}_{, k \bar{k}} f_i \varphi_{\jbar} + \sum_{i, k} f_{i k \bar{k}} \varphi_{\ibar} + \sum_{i, k} f_{ik} \varphi_{\ibar \bar{k}} + \sum_{i, k} f_i f_{\bar{i} \bar{k}} \varphi_k \Big{)} 
\\
&\qquad - \Big{(} \sum_{i, k, l} {g^{k \bar{l}}}_{, i \bar{l}} f_k \varphi_{\ibar} + \sum_{i, k} f_{k i \bar{k}} \varphi_{\ibar} + \sum_{i, k} f_{k i} f_{\bar{k}} \varphi_{\ibar} \Big{)} \Big{]}. 
\end{align*}
This reduces to $\mathrm{Re} \sum_{i, k} f_{ik} \varphi_{\bar{i} \bar{k}} = \mathrm{Re} (\mathcal{D} \varphi, \mathcal{D} f) (p)$ by 
\[ \sum_{i,k,l} {g^{k \bar{l}}}_{, i \bar{l}} f_k \varphi_{\ibar} = - \sum_{i, k, l, p, q} g^{k \bar{q}} g^{p \bar{l}} g_{p \bar{q}, i \bar{l}} f_k \varphi_{\ibar} = \sum_{i, k, l} R_{l \bar{k} i \bar{l}} f_k \varphi_{\ibar} = \sum_{i, k, l} R_{i \bar{k} l \bar{l}} f_k \varphi_{\ibar} = \sum_{i, k, l} {g^{k \ibar}}_{, l \bar{l}} f_k \varphi_{\ibar} \]
and 
\[ \mathrm{Re} (\sum_{i, k} f_i f_{\bar{i} \bar{k}} \varphi_k - \sum_{i, k} f_{k i} f_{\bar{k}} \varphi_{\ibar}) = 0. \]
\end{proof}

Now we show Theorem \ref{Critical points of W}. 

\begin{proof}[Proof of Theorem \ref{Critical points of W}]
We exhibit the first variation for general $f$ as an independent interest. 
Using the above proposition, we compute 
\begin{align*}
\frac{d}{dt}\Big{|}_{t=0} \cW^\lambda (\omega_t, f) 
&= \frac{-1}{\int_X e^f \omega^n} \int_X \Big{[} - \mathcal{D}^* \mathcal{D} \varphi + ({\bar{\partial}}^\sharp s, \nabla \varphi) + (\sqddbar \varphi, \sqddbar f) 
\\
&\qquad \qquad \qquad - \Big{(} (s (\omega) + \bar{\Box} f -\lambda (n+f)) - (\bar{s}^\lambda_f (\omega) -\lambda n - \lambda) \Big{)} \bar{\Box} \varphi \Big{]} ~e^f \omega^n 
\\
&= \frac{-1}{\int_X e^f \omega^n} \int_X \Big{[} - \mathcal{D}^* \mathcal{D} \varphi + ({\bar{\partial}}^\sharp s, \nabla \varphi) + (\sqddbar \varphi, \sqddbar f) 
\\
&\qquad \qquad \qquad + (\bar{\Box} f - |\partial^\sharp f|^2 - \lambda) \bar{\Box} \varphi - (s^\lambda_f - \bar{s}^\lambda_f) \bar{\Box} \varphi \Big{]} ~e^f \omega^n  
\\
&= \frac{-1}{\int_X e^f \omega^n} \int_X \Big{[} - \bar{\Box}^2 \varphi - (\Ric (\omega) - \sqddbar f, \sqddbar \varphi) 
\\
&\qquad \qquad \qquad + (\bar{\Box} f - |\partial^\sharp f|^2) \bar{\Box} \varphi - \lambda (\bar{\partial}^\sharp f, \nabla \varphi) - (s^\lambda_f - \bar{s}^\lambda_f) \bar{\Box} \varphi \Big{]} ~e^f \omega^n  
\\
&= \frac{-1}{\int_X e^f \omega^n} \int_X \Big{[} - (\bar{\Box} - \bar{\partial}^\sharp f)^2 \varphi - (\Ric (\omega) - \sqddbar f, \sqddbar \varphi) - (\bar{\partial}^\sharp s^0_f, \nabla \varphi) 
\\
&\qquad \qquad \qquad - \Big{(} \bar{\Box} (\bar{\partial}^\sharp f (\varphi)) + \bar{\partial}^\sharp f (\bar{\Box} \varphi) - \bar{\partial}^\sharp f (\bar{\partial}^\sharp f (\varphi)) \Big{)} + \bar{\partial}^\sharp f (\bar{\Box} \varphi)
\\
&\qquad \qquad \qquad \quad + (\bar{\partial}^\sharp s^\lambda_f, \nabla \varphi) - (s^\lambda_f - \bar{s}^\lambda_f) \bar{\Box} \varphi \Big{]} ~e^f \omega^n 
\\
&= \frac{-1}{\int_X e^f \omega^n} \int_X \Big{[} -\mathcal{D}^{f *} \mathcal{D} \varphi + B_f (\varphi)  - (\bar{\Box} - \bar{\partial}^\sharp f) (\bar{\partial}^\sharp f (\varphi)) - (s^\lambda_f - \bar{s}^\lambda_f) (\bar{\partial}^\sharp f, \nabla \varphi) \Big{]} ~e^f \omega^n 
\\
&= \frac{-1}{\int_X e^f \omega^n} \int_X B_f (\varphi) ~e^f \omega^n + \frac{1}{\int_X e^f \omega^n} \int_X (s^\lambda_f - \bar{s}^\lambda_f) (\bar{\partial}^\sharp f, \nabla \varphi) ~e^f \omega^n
\\
&= \frac{-1}{\int_X e^f \omega^n} \mathrm{Re} \int_X B_f (\varphi) ~e^f \omega^n + \frac{1}{2 \int_X e^f \omega^n} \int_X (s^\lambda_f - \bar{s}^\lambda_f) (\nabla f, \nabla \varphi) ~e^f \omega^n
\\
&= \frac{1}{\int_X e^f \omega^n} \int_X \mathrm{Re} (A_f (\varphi) - B_f (\varphi)) ~e^f \omega^n + \frac{1}{2 \int_X e^f \omega^n} \int_X (s^\lambda_f - \bar{s}^\lambda_f) (\nabla f, \nabla \varphi) ~e^f \omega^n
\\
&= \frac{1}{\int_X e^f \omega^n} \int_X \mathrm{Re} (\mathcal{D} \varphi, \mathcal{D} f) ~e^f \omega^n + \frac{1}{2 \int_X e^f \omega^n} \int_X (s^\lambda_f - \bar{s}^\lambda_f) (\nabla f, \nabla \varphi) ~e^f \omega^n. 
\end{align*}
\end{proof}

\begin{cor}
Critical points of $\cW^\lambda: T \mathcal{H} \to \mathbb{R}$ correspond to $\mu^\lambda$-cscK metrics. 
Namely, a state $(\omega, f) \in T \mathcal{H} (X, L) = \mathcal{H} (X, L) \times C^\infty (X)/\mathbb{R}$ is a critical point of $\cW^\lambda$ if and only if $\xi' = \partial^\sharp f$ gives a holomorphic vector field and the K\"ahler metric $\omega$ is a $\mu^\lambda$-cscK metric with respect to $\xi = \mathrm{Im} \xi'$. 
\end{cor}

\begin{proof}
We apply the above theorem for $\phi_t = t f$. 
It follows that a state $(\omega, f)$ is a critical point of $\cW^\lambda$ if and only if $f$ satisfies $s^\lambda_f (\omega) - \bar{s}^\lambda_f (\omega) = 0$ and $\mathcal{D} f = 0$, which is nothing but the condition in the claim. 
\end{proof}

\subsubsection{$\mu$-entropy in high temperature}

When $\lambda \le 0$, the critical momentum $f$ turns out to be a unique maximizer of the functional $\cW^\lambda (\omega, \cdot): C^{0,1} (X) \to \mathbb{R}$, which allows us to reduce all the information of the critical points of $\cW^\lambda$ on $T \mathcal{H}$ to that of a functional $\cmu^\lambda$ on $\mathcal{H} (X, L)$, analogous to the Calabi functional. 

\begin{thm}
\label{Uniqueness of moment}
Suppose $\lambda \le 0$, then we have the following. 
\begin{enumerate}
\item The functional $\cW^\lambda (\omega, \cdot): C^\infty (X) \to \mathbb{R}$ admits a unique critical point $f$ modulo constant for every K\"ahler metric $\omega$, which automatically maximizes $\cW^\lambda (\omega, \cdot)$. 

\item The functional $\cmu^\lambda: \mathcal{H} (X, L) \to \mathbb{R}$ is smooth. 
In this case, the following are equivalent for a K\"ahler metric $\omega$: 
\begin{itemize}
\item $\omega$ is a $\mu^\lambda$-cscK metric 

\item $\omega$ is a critical point of $\cmu^\lambda$ 

\item There is $\xi \in {^\nabla \mathfrak{isom}} (X, \omega)$ such that $\NAmu^\lambda (X, L; \xi) = \cmu^\lambda (\omega)$. 
\end{itemize}
In one of the above cases, $\omega$ minimizes $\cmu^\lambda$ among all $\xi$-invariant metrics. 
\end{enumerate}
\end{thm}

\begin{proof}
Suppose there are two critical points $f, g \in C^\infty (X)$ of $W^\lambda (\omega, \cdot)$. 
We may normalize $f, g$ so that $\int_X (f-g) e^{(f+g)/2} \omega^n = 0$. 
We have 
\[ \Delta f - \frac{1}{2} |\nabla f|^2 - \lambda f = \bar{s}^\lambda_f -s (\omega) = \Delta g - \frac{1}{2} |\nabla g|^2 - \lambda g + (\bar{s}^\lambda_f - \bar{s}^\lambda_g). \]
Putting $u= f- g$ and $h = (f+g)/2$, we can arrange this as 
\[ \Delta u - (\nabla h, \nabla u) = \lambda u + (\bar{s}^\lambda_f - \bar{s}^\lambda_g). \]
Then we get 
\[ \int_X |\nabla u|^2 e^h \omega^n = \int_X (\Delta u - (\nabla h, \nabla u)) u~ e^h \omega^n = \lambda \int_X u^2 e^h \omega^n, \]
which implies $u = f - g$ is zero when $\lambda \le 0$ under the normalization condition. 

Next we show the second claim. 
If $(\omega, f)$ is a critical point of $\cW^\lambda$, then for any smooth perturbation $\omega_t$, we have a smooth family $f_t$ satisfying $s^\lambda_{f_t} (\omega_t) = \bar{s}^\lambda_{f_t} (\omega_t)$. 
To see this, we compute the derivative 
\begin{align*} 
D_0 \mathcal{S}^\lambda (0, u) 
&=\Delta u - (\nabla f, \nabla u) - \lambda u
\end{align*}
of the smooth map 
\[ \mathcal{S}^\lambda: C^{k+4, \alpha}_f (X) \times C^{k+2, \alpha}_f (X) \to C^{k, \alpha}_f (X): (\phi, u) \mapsto s^\lambda_{f+u} (\omega_\phi) - \bar{s}^\lambda_{f+u} (\omega_\phi). \]
Here we put $C^{k, \alpha}_f (X) = \{ u \in C^{k, \alpha} (X) ~|~ \int_X u e^f \omega^n = 0 \}$. 
The above expression shows that $D_0 \mathcal{S}^\lambda|_{\{ 0 \} \times C^{k+2, \alpha}_f (X)}: C^{k+2, \alpha}_f (X) \to C^{k, \alpha}_f (X)$ has a right inverse when $\lambda \le 0$, so that $\mathrm{Ker} D_0 \mathcal{S}^\lambda$ maps onto $C^{k+4, \alpha}_f (X)$ by the projection to the first factor. 
Thus by the implicit function theorem, we get the desired smooth family $f_t$. 
By the uniqueness of the critical momentum, we have $\cmu^\lambda (\omega_t) = \cW^\lambda (\omega_t, f_t)$ for this $f_t$, which shows the smoothness of $\cmu^\lambda$. 

As for the equivalence, if $\omega$ is a critical point of $\cmu^\lambda$, then $(\omega, f)$ gives a critical point of $\cW^\lambda$ for the unique critical momentum $f$ as 
\[ \frac{d}{dt}\Big{|}_{t=0} \cW^\lambda (\omega_t, f) = \frac{d}{dt}\Big{|}_{t=0} \cW^\lambda (\omega_t, f_t) = \frac{d}{dt}\Big{|}_{t=0} \cmu^\lambda (\omega_t) = 0, \] 
so that it is a $\mu^\lambda_\xi$-cscK metric for the real holomorphic vector field $\xi = \mathrm{Im} \partial^\sharp f$ thanks to the above corollary. 
Conversely, if $\omega$ is a $\mu^\lambda$-cscK metric, then it is a critical point of $\cW^\lambda$, so that it is a critical point of $\cmu^\lambda$. 

By the uniqueness of the critical momentum, we have $f = \mu_\xi$ for a $\check{\mu}^\lambda_\xi$-cscK metric $\omega$, so that we have $\NAmu^\lambda (X, L; \xi) = \cmu^\lambda (\omega)$. 
Conversely, if $\NAmu^\lambda (X, L; \xi) = \cmu^\lambda (\omega)$ for $\xi \in {^\nabla \mathfrak{isom}} (X, \omega)$, then since $\NAmu^\lambda (X, L; \xi) = \cW^\lambda (\omega, \mu_\xi)$, we have $\cmu^\lambda (\omega) = \cW^\lambda (\omega, \mu_\xi)$. 
Thus $\mu_\xi$ is a critical momentum, so that $\omega$ is a $\check{\mu}^\lambda_\xi$-cscK metric. 

As for the last claim, we have 
\[ \cmu^\lambda (\omega) = \cW^\lambda (\omega, \mu_{-2\xi}) = \NAmu^\lambda (\phi_{-2\xi}) = \cW^\lambda (\omega_\varphi, \mu_{-2\xi}^\varphi) \le \cmu^\lambda (\omega_\varphi) \]
for $\xi$-invariant metric $\omega_\varphi$, so that $\omega$ minimizes $\cmu^\lambda$ among all $\xi$-invariant metrics. 
\end{proof}

We will refine the last claim in the next section as we claimed in Theorem \ref{main theorem on mu-entropy}.

\section{$W$-entropy along geodesic rays}
\label{W-functional along geodesic rays}

In this section, we study $W$-functional along geodesic rays to prove the rest of our main theorems: (a) $\Rightarrow$ (b) in Theorem \ref{main theorem on mu-entropy} and Theorem \ref{inequality}. 
We imitate the proof of the corresponding inequality on $H$-functional (cf. \cite{DS}). 
The following are the key materials for the proof: 
\begin{itemize}
\item The monotonicity of the $W$-entropy along $C^{1,1}$-geodesic rays.  

\item The non-archimedean $\mu$-entropy of a test configuration is the limit of $W$-entropy along the geodesic ray subordinate to the test configuration. 
\end{itemize}

We must clarify the meaning of the $W$-entropy $\cW^\lambda (\omega_{\phi_t}, -\dot{\phi}_t)$ along $C^{1,1}$-geodesic ray as the curvature term $s (\omega_\phi) - \bar{\Box} \dot{\phi}$ of the integrand does not make sense for $C^{1,1}$-regular $\phi$. 
Integration by parts shows that the action functional $- \int_0^t \cW^\lambda (\omega_{\phi_t}, -\dot{\phi}_t) dt$ has a curvature free expression, so that we can introduce the action functional even for $C^{1,1}$-geodesic rays. 
We actually study the convexity and the slope of this action functional, instead of studying the monotonicity and the limit of the $W$-entropy.

\subsection{$W$-entropy is monotonic along geodesics}
\label{Monotonicity}

\subsubsection{First observation: monotonicity along smooth geodesic}

We firstly show the monotonicity along smooth geodesic rays as a motivative observation. 

\begin{prop}
Let $\{ \phi_t \}$ be a smooth geodesic. 
Then we have 
\begin{equation} 
\frac{d}{dt} \cW^\lambda (\omega_{\phi_t}, -\dot{\phi}_t) = - \int_X |\mathcal{D}_{\phi_t} \dot{\phi}_t|^2 e^{- \dot{\phi}_t} \omega_{\phi_t}^n \Big{/} \int_X e^{-\dot{\phi}_t} \omega_{\phi_t}^n \le 0. 
\end{equation}
\end{prop}

\begin{proof}
Put $\theta_t = -\dot{\phi}_t$. 
We have $\dot{\theta}_t = -|\bar{\partial}^\sharp_{\phi_t} \dot{\phi}_t|^2 = (\bar{\partial}^\sharp_{\phi_t} \theta_t, \nabla \dot{\phi}_t)$ as $\{ \phi_t \}$ is a geodesic. 
We easily see 
\[ \frac{d}{dt} \int_X e^{\theta_t} \omega_{\phi_t}^n = 0, \quad \frac{d}{dt} \int_X \theta_t e^{\theta_t} \omega_{\phi_t}^n = 0. \]

We compute 
\begin{align*}
\frac{d}{dt} \int_X
& (s (\omega_{\phi_t}) + \bar{\Box}_{\phi_t} \theta_t) e^{\theta_t} \omega_{\phi_t}^n
\\
&= \int_X \Big{(} - \mathcal{D}_{\phi_t}^* \mathcal{D}_{\phi_t} \dot{\phi}_t + (\bar{\partial}^\sharp s (\omega_{\phi_t}), \nabla \dot{\phi}_t) + (\sqddbar \dot{\phi}_t, \sqddbar \theta_t) + \bar{\Box}_{\phi_t} \dot{\theta}_t \Big{)} e^{\theta_t} \omega_{\phi_t}^n
\\
&\qquad- \int_X (s (\omega_{\phi_t}) + \bar{\Box}_{\phi_t} \theta_t) (\bar{\Box} \dot{\phi}_t - \dot{\theta}_t) e^{\theta_t} \omega_{\phi_t}^n
\\
&= \int_X \Big{(} - \bar{\Box}^2 \dot{\phi}_t - (\mathrm{Ric} (\omega_{\phi_t}), \sqddbar \dot{\phi}_t) + (\sqddbar \dot{\phi}_t, \sqddbar \theta_t) + \bar{\Box}_{\phi_t} (\bar{\partial}^\sharp \theta_t (\dot{\phi}_t)) \Big{)} e^{\theta_t} \omega_{\phi_t}^n
\\
&\qquad+ \int_X \bar{\Box}_{\phi_t} (\dot{\theta}_t - \bar{\partial}^\sharp \theta_t (\dot{\phi}_t)) e^{\theta_t} \omega_{\phi_t}^n
\\
&\qquad- \int_X (s (\omega_{\phi_t}) + \bar{\Box}_{\phi_t} \theta_t) (\bar{\Box} - \bar{\partial}^\sharp_{\phi_t} \theta_t) (\dot{\phi}_t) e^{\theta_t} \omega_{\phi_t}^n
\\
&\qquad \quad + \int_X (s (\omega_{\phi_t}) + \bar{\Box}_{\phi_t} \theta_t) (\dot{\theta}_t - \bar{\partial}^\sharp \theta_t (\dot{\phi}_t)) e^{\theta_t} \omega_{\phi_t}^n
\\
&= - \int_X \Big{(} \bar{\Box}_{\phi_t} (\bar{\Box}_{\phi_t} - \bar{\partial}^\sharp \theta_t) \dot{\phi}_t + (\Ric (\omega_{\phi_t}) - \sqddbar \theta_t, \sqddbar \dot{\phi}_t) \Big{)} e^{\theta_t} \omega_{\phi_t}^n
\\
&\qquad- \int_X (\bar{\partial}^\sharp (s (\omega_{\phi_t}) + \bar{\Box}_{\phi_t} \theta_t), \nabla \dot{\phi}_t) e^{\theta_t} \omega_{\phi_t}^n
\\
&= - \int_X \Big{(} (\bar{\Box}_{\phi_t} - \bar{\partial}^\sharp \theta_t)^2 \dot{\phi}_t + (\Ric (\omega_{\phi_t}) - \sqddbar \theta_t, \sqddbar \dot{\phi}_t) + (\bar{\partial}^\sharp s^0_{\theta_t} (\omega_{\phi_t}) , \nabla \dot{\phi}_t) \Big{)} e^{\theta_t} \omega_{\phi_t}^n
\\
&= - \int_X \Big{(} \mathcal{D}_{\phi_t}^{\theta_t *} \mathcal{D}_{\phi_t} \dot{\phi}_t - B_{\theta_t} (\dot{\phi}_t) \Big{)} e^{\theta_t} \omega_{\phi_t}^n
\\
&= \mathrm{Re} \int_X B_{\theta_t} (\dot{\phi}_t) e^{\theta_t} \omega_{\phi_t}^n
= - \int_X \mathrm{Re} (A_{\theta_t} (\dot{\phi}_t) - B_{\theta_t} (\dot{\phi}_t)) e^{\theta_t} \omega_{\phi_t}^n
\\
&= - \int_X \mathrm{Re} (\mathcal{D}_{\phi_t} \dot{\phi}_t, \mathcal{D}_{\phi_t} \theta_t) e^{\theta_t} \omega_{\phi_t}^n 
= \int_X |\mathcal{D}_{\phi_t} \theta_t|^2 e^{\theta_t} \omega_{\phi_t}^n. 
\end{align*}
Thus we get 
\[ \frac{d}{dt} \cW^\lambda (\omega_{\phi_t}, - \dot{\phi}_t) = - \int_X |\mathcal{D}_{\phi_t} \theta_t|^2 e^{\theta_t} \omega_{\phi_t}^n \Big{/} \int_X e^{\theta_t} \omega_{\phi_t}. \]
\end{proof}

\subsubsection{Relaxed action functional and integration by parts}

For smooth rays $\bm{\phi} = \{ \phi_s \}_{s \in [0,\infty)} \subset C^\infty (X, \omega)$ and $\bm{\psi} = \{ \psi_s \}_{s \in [0, \infty)} \subset C^\infty (X, \sigma)$ of $(1,1)$-forms, we put 
\[ \mathcal{A}^{\bm{\psi}}_{\bm{\phi}} (t) := - \int_0^t ds \int_X (dd^c_\sigma \psi_s - \dot{\psi}_s) e^{dd^c_\omega \phi_s -\dot{\phi}_s}. \]
When $\omega_{\phi_s} = dd^c_\omega \phi_s$ is positive, $\bm{\psi} = \{ \log (\omega_{\phi_s}^n/\omega^n) \}_{s \in [0, \infty)} \subset C^\infty (X, -\mathrm{Ric} (\omega))$ gives a smooth ray, and we have 
\[ \mathcal{A}^{\bm{\psi}}_{\bm{\phi}} (t) = \frac{1}{n!} \int_0^t ds \int_X (s (\omega_{\phi_s}) - \bar{\Box}_{\phi_s} \dot{\phi}_s) e^{-\dot{\phi}_s} \omega_{\phi_s}^n. \]
We simply write this as $\mathcal{A}_{\bm{\phi}} (t)$. 


For $A_t := \{ \tau \in \mathbb{C}_- ~|~ 1 \le |\tau| < e^t \}$, we consider functionals $\Phi (x, \tau) := \phi_{\log |\tau|} (x)$ and $\Psi (x, \tau) := \psi_{\log |\tau|} (x)$ on $X \times A_t$. 
The integration by parts shows the following. 

\begin{prop}
We have 
\[ \mathcal{A}^{\bm{\psi}}_{\bm{\phi}} (t) = \int_X \psi_t e^{-\dot{\phi}_t} \frac{(dd^c_\omega \phi_t)^n}{n!} + \frac{1}{\pi} \int_{X \times A_t} \Psi e^{-\dot{\Phi}} \frac{(dd^c_\omega \Phi)^{n+1}}{(n+1)!} - \int_0^t ds \int_X \sigma \wedge e^{-\dot{\phi}_s} \frac{(dd^c_\omega \phi_s)^{n-1}}{(n-1)!}. \]
\end{prop}

\begin{proof}
By definition, we have
\[ n!. \mathcal{A}^{\bm{\psi}}_{\bm{\phi}} (t) = - \int_0^t ds \int_X n \sigma \wedge e^{-\dot{\phi}} (dd^c_\omega \phi)^{n-1} + \int_0^t ds \int_X n dd^c \psi \wedge e^{- \dot{\phi}} (dd^c_\omega \phi)^{n-1} + \int_0^t ds \int_X \dot{\psi} e^{-\dot{\phi}} (dd^c_\omega \phi)^n. \]

We can compute the integrand of the last term as 
\begin{align*} 
\int_X \dot{\psi} e^{- \dot{\phi}} (dd^c_\omega \phi)^n 
&= \frac{d}{ds} \int_X \psi e^{-\dot{\phi}} (dd^c_\omega \phi)^n + \int_X \psi (\ddot{\phi} e^{-\dot{\phi}} (dd^c_\omega \phi)^n - n d \dot{\phi} \wedge d^c \dot{\phi} \wedge e^{-\dot{\phi}} (dd^c_\omega \phi)^{n-1}) 
\\
&\qquad- \int_X n dd^c \psi \wedge e^{-\dot{\phi}} (dd^c_\omega \phi)^{n-1} 
\end{align*}
by using 
\begin{align*}
\frac{d}{dt} \Big{(} \psi e^{-\dot{\phi}} (dd^c_\omega \phi)^n \Big{)} 
&=  \dot{\psi} e^{- \dot{\phi}} (dd^c_\omega \phi)^n - \psi (\ddot{\phi} e^{-\dot{\phi}} (dd^c_\omega \phi)^n - n d \dot{\phi} \wedge d^c \dot{\phi} \wedge e^{-\dot{\phi}} (dd^c_\omega \phi)^{n-1}) 
\\
&\qquad + n \psi (dd^c \dot{\phi} - d \dot{\phi} \wedge d^c \dot{\phi}) e^{-\dot{\phi}} \wedge (dd^c_\omega \phi)^{n-1} 
\\
&=  \dot{\psi} e^{- \dot{\phi}} (dd^c_\omega \phi)^n - \psi (\ddot{\phi} e^{-\dot{\phi}} (dd^c_\omega \phi)^n - n d \dot{\phi} \wedge d^c \dot{\phi} \wedge e^{-\dot{\phi}} (dd^c_\omega \phi)^{n-1}) 
\\
&\qquad + n \psi dd^c (e^{- \dot{\phi}} (dd^c_\omega \phi)^{n-1}). 
\end{align*}

%
We can arrange the integration of the second term as 
\[ \int_0^t ds \int_X \psi (\ddot{\phi} e^{-\dot{\phi}} (dd^c_\omega \phi)^n - n d \dot{\phi} \wedge d^c \dot{\phi} \wedge e^{-\dot{\phi}} (dd^c_\omega \phi)^{n-1}) = \frac{1}{\pi} \int_{X \times A_t} \Psi e^{-\dot{\Phi}} (dd_\omega^c \Phi)^{n+1}, \] 
which shows the claim. 
\end{proof}

\subsubsection{Subgeodesic and momentum}

We put $A_t := \{ \tau \in \mathbb{C}_- ~|~ 1 \le |\tau| < e^t \}$, where $\mathbb{C}_-$ denotes $\mathbb{C}$ endowed with the reversed scaling action $z. t = t^{-1} z$. 
A \textit{subgeodesic ray} $\Phi$ on $X$ is a $U (1)$-invariant $\omega$-plurisubharmonic function on $X \times A_\infty$. 
We also use the notation $\bm{\phi} = \{ \phi_t (x) = \Phi (x, e^t) \}_{t \in [0, \infty)}$ to denote the subgeodesic $\Phi$. 
In this article, we restrict our interest to locally bounded $\Phi$ to simplify arguments on the \textit{momentum} $\dot{\Phi}$ introduced later. 
A (locally bounded) \textit{geodesic ray} is a locally bounded subgeodesic ray satisfying 
\[ (dd_\omega^c \Phi)^{n+1} = 0 \]
in the sense of Bedford--Taylor. 

For a normal test configuration $(\mathcal{X}, \mathcal{L})$ and any initial metric $\omega$, there exists a unique geodesic ray $\Phi$ emanating from $\omega$ such that the $\omega$-psh $\Phi (x, t^{-1})$ on $X \times \Delta^* \subset \mathcal{X}_\Delta$ extends to a locally bounded psh metric on the line bundle $\mathcal{L}_\Delta$ on $\mathcal{X}_\Delta$ (cf. \cite{PS1, BBJ}). 
We write it as $\Phi_{(\mathcal{X}, \mathcal{L})}$ or $\bm{\phi}_{(\mathcal{X}, \mathcal{L})}$. 
We say $\Phi$ is \textit{subordinate to} a normal test configuration $(\mathcal{X}, \mathcal{L}; \rho)$ if $\Phi (x, \tau) = \Phi_{(\mathcal{X}, \mathcal{L})} (x, \rho \tau)$. 
If we shift the $\mathbb{C}^\times$-action on $\mathcal{L}$ by a weight $m$, the associated geodesic ray shifts by $m \log |\tau|$. 

It is shown by \cite{CTW2} (cf. \cite{PS2}) that a geodesic ray $\Phi$ subordinate to some test configuration is in $C^{1,1}_{\mathrm{loc}} (X \times A_\infty, \omega)$. 
A \textit{$C^{1,1}$-(sub)geodesic ray} is a (sub)geodesic ray in $C^{1,1}_{\mathrm{loc}} (X \times A_\infty, \omega)$. 
We note functions in this class are regular than $C^{1, \bar{1}}_{\mathrm{loc}}$ consisting of $L^p_2$-functions with locally bounded Laplacian ($\forall p < \infty$). 

We can identify the space $C^{1,1}_{\mathrm{loc}}$ with the Sobolev space $L^\infty_{2, \mathrm{loc}}$. 
The Bedford--Taylor product $dd^c_\omega \phi_1 \wedge \dotsb \wedge dd^c_\omega \phi_k$ for $C^{1,1}$-regular $\omega$-psh $\phi_1, \ldots, \phi_k$ coincides with the usual product of differential forms $dd^c_\omega \phi_i$ with $L^\infty$-coefficients. 

We note the closure of $C^\infty \subset C^{k,1}$ in the $C^{k,1}$-norm is $C^{k+1} \subsetneq C^{k,1}$, so a purely $C^{k,1}$-function cannot be approximated by smooth functions in the strong $C^{k,1}$-topology. 
Still, every $C^{k,1}$-function can be approximated by smooth functions in the following weak topology, which is slightly weaker than the usual weak topology in functional analysis: Identifying $C^{k, 1}$ with a subspace of $C^k \times L^\infty$ by $f \mapsto (f, \nabla^{k+1} f)$, our weak topology on $C^{k,1}$ is defined as the topology induced from the strong topology of $C^k$ and the weak$^*$ topology of $L^\infty = (L^1)^*$. 
In this topology, $f_i$ converges to $f$ in $C^{k,1}$ if and only if it converges in $C^k$ and $\int_X \nabla^{k+1} f_i g d\mu \to \int_X \nabla^{k+1} f g d\mu$ for every function $g$ integrable with respect to the Lebesgue measure $\mu$. 
(The usual weak topology further assumes the convergence of $\langle \nabla^{k+1} f_i, G \rangle \to \langle \nabla^{k+1} f, G \rangle$ for every $G \in (L^\infty)^* \supsetneq L^1$. )
For $f \in C^{k,1}$, the convolution $\rho_\epsilon * f$ with a mollifyer $\rho_\epsilon$ gives a desired approximation, thanks to the uniform boundedness 
\[ \| \nabla^{k+1} (\rho_\epsilon * f) \|_{L^\infty} = \| \rho_\epsilon * \nabla^{k+1} f \|_{L^\infty} \le \| \rho_\epsilon \|_{L^1} \cdot \| \nabla^{k+1} f \|_{L^\infty} = \| \nabla^{k+1} f \|_{L^\infty}. \]
Here note we have sequential Banach--Alaoglu theorem since $L^\infty$ is the dual of separable space $L^1$ (while $L^\infty$ is not separable). 
We refer to approximation in this weak topology as \textit{approximation in $C^{k,1}$}. 
If $f_i \to f$ and $g_i \to g$ are approximation in $C^{k,1}$, then $f_i g_i \to fg$ is also approximation in $C^{k,1}$. 

We call attention to the difference between ``everywhere'' and ``almost everywhere'' for two reasons:
\begin{itemize}
\item When dealing with locally bounded psh, the Monge--Amp\`ere measures $(dd_\omega \phi_t)^n$ are not necessarily absolutely continuous with respect to the Lebesgue measure. 
So ``equal almost everywhere with respect to the Lebesgue measure'' does not imply ``equal almost everywhere with respect to $(dd^c \phi_t)^n$''. 

\item For a smooth metric $\omega$ and a test configuration $(\mathcal{X}, \mathcal{L}; \rho)$, we will construct a (possibly uncontinuous) monotonically decreasing functional $\cW^+$ on $[0, \infty)$ such that (1) $\cmu (\omega) \ge \cW^+ (0)$ and (2) $\lim_{t \to \infty} \cW^+ (t) = \NAmu (\mathcal{X}, \mathcal{L}; \rho)$, so that $\cmu (\omega) \ge \NAmu (\mathcal{X}, \mathcal{L}; \rho)$. 
If $\cW^+$ is just a functional satisfying (1) and (2), and $\cW^+ = \tilde{\mathcal{W}}^+$ almost everywhere for a monotonically decreasing functional $\tilde{\mathcal{W}}^+$, we fail to conclude $\cmu (\omega) \ge \NAmu (\mathcal{X}, \mathcal{L}; \rho)$ as we may have $\tilde{\mathcal{W}}^+ (0) \ge \cmu (\omega)$. 
\end{itemize}

When we clarify this point, we use the notation $\mathfrak{B} (X)$ rather than $L^\infty (X)$ to denote the space of (locally) bounded Borel functions. 
(Usually, $L^\infty (X)$ denotes the quotient of $\mathfrak{B} (X)$ which identifies two functions equal almost everywhere with respect to the Lebesgue measure. )
We use the notation $\mathfrak{B} (X, \sigma)$ similarly as $C^\infty (X, \sigma)$. 
We call a family $\{ \psi_t \}$ of locally bounded Borel functions \textit{equicontinuous} if the family $\{ \psi_t (x) \}_{x \in X}$ of functions on $t$ is equicontinuous. 
Namely, for every $t \in [0, \infty)$ and $\varepsilon > 0$, there is $\delta > 0$ such that $\sup_{x \in X} |\psi_t (x) - \psi_{t'} (x)| < \epsilon$ for every $t' \in [0, \infty)$ with $|t-t'| < \delta$. 

For a subgeodesic $\Phi$, for each $x \in X$, $\Phi (x, \tau)$ is a $U (1)$-invariant psh on $\tau$. 
This implies $\phi_t (x)$ is convex on $t$ for each $x \in X$. 
By the convexity, $\delta^{-1} (\phi_{t+\delta} (x) - \phi_t (x))$ is monotonically decreasing with respect to $\delta$, so that we get the right derivative: 
\begin{equation} 
\dot{\phi}_t (x) := \lim_{\delta \to +0} \frac{\phi_{t+\delta} (x) - \phi_t (x)}{\delta}. 
\end{equation}
Moreover, $\dot{\phi}_t$ is monotonically increasing with respect to $t$. 
We put $\dot{\Phi} (x, \tau) := \dot{\phi}_{\log |\tau|} (x)$. 
This function is Borel as the decreasing limit of Borel functions is Borel. 

For $t > 0$, we have 
\[ \frac{\phi_t (x) - \phi_{t-\delta} (x)}{\delta} \le \dot{\phi}_t (x) \le \frac{\phi_{t+\delta} (x) - \phi_t (x)}{\delta} \]
for any small $\delta > 0$ and every $x \in X$. 
Thus, suppose $\Phi$ is locally bounded, $\dot{\phi}_t$ is bounded on $X$ for each $t > 0$ and bounded from above for $t=0$. 
Moreover, by the monotonicity, $\dot{\Phi}|_{X \times A^\circ}$ is locally bounded when $\Phi$ is so. 

Furthermore, by 
\[ -\delta C \le \delta \dot{\phi}_t \le \phi_{t+ \delta} - \phi_t \le \delta \dot{\phi}_{t+\delta} \le \delta \dot{\phi}_{t+\delta_0} \le \delta C, \]
the family $\{ \phi_t \}$ is equicontinuous in the above sense. 
In particular, for $t_i \to t > 0$, we have the weak convergence $(dd^c_\omega \phi_{t_i})^n \to (dd^c \phi_t)^n$ of measures as the Bedford--Taylor product is continuous with respect to the pointwise uniform convergence. 
In particular, for a $C^{1,1}$-geodesic $\Phi$, the $L^\infty$-differential forms $(dd^c_\omega \phi_{t_i})^n$ weakly converge to $(dd^c_\omega \phi_t)^n$ as currents.

The following is well-known (cf. \cite{His16, His}). 

\begin{lem}
\label{affine functional}
For a $C^{1,1}$-geodesic ray $\bm{\phi}$, the integrations 
\[ \int_X e^{-\dot{\phi}_t} (dd^c_\omega \phi_t)^n, \quad \int_X \dot{\phi}_t e^{-\dot{\phi}_t} (dd^c_\omega \phi_t)^n \] 
are independent of $t$. 
\end{lem}

\begin{proof}
One can show this by firstly computing the derivative for smooth (not necessarily subgeodesic) ray, then approximating the geodesic ray by smooth rays in $C^{1,1}_{\mathrm{loc}}$. 
As the derivative consists of integration with respect to $(dd_\omega^c \Phi)^{n+1}$, it vanishes. 
As we will exhibit and repeat such arguments later, we omit the detail for this well-known case. 
\end{proof}

\subsubsection{Relaxed action functional along subgeodesic ray}

Here we study the relaxed action functional $\mathcal{A}_{\bm{\phi}}^{\bm{\psi}}$ for weakly regular $\bm{\phi}$ and $\bm{\psi}$. 
Though we only deal with $C^{1,1}$-geodesic ray $\bm{\phi}$ except for this subsection, we introduce the functional for locally bounded subgeodesics $\bm{\phi}$ in order to clarify a sufficient condition for continuity. 

\begin{defin}
For a locally bounded subgeodesic ray $\Phi$ and a locally bounded Borel function $\Psi \in \mathfrak{B} (X \times A_\infty, \sigma)$, for each $t \in (0, \infty)$, we put 
\[ \mathcal{A}_\Phi^\Psi (t) := \int_X \psi_t e^{-\dot{\phi}_t} \frac{(dd^c_\omega \phi_t)^n}{n!} + \frac{1}{\pi} \int_{X \times A_t} \Psi e^{-\dot{\Phi}} \frac{(dd_\omega^c \Phi)^{n+1}}{(n+1)!} - \int_0^t ds \int_X \sigma \wedge e^{-\dot{\phi}_s} \frac{(dd^c_\omega \phi_s)^{n-1}}{(n-1)!}. \]
If moreover $\dot{\phi}_0$ is bounded from below, we put $\mathcal{A}_\Phi^\Psi (0) := \int_X \psi_0 e^{-\dot{\phi}_0} \omega^n/n!$. 
\end{defin}

For a smooth $\Phi$ and $\Psi$, we have $\mathcal{A}_\Phi^\Psi = \mathcal{A}_{\bm{\phi}}^{\bm{\psi}}$ for $\mathcal{A}_{\bm{\phi}}^{\bm{\psi}}$ in the previous subsection. 
The relaxed action $\mathcal{A}_\Phi^\Psi$ gives a locally bounded Borel function on $(0, \infty) \text{ or } [0, \infty)$. 
We give a sufficient condition for continuity. 
As the measure $(dd^c \Phi)^{n+1}$ may charge the boundary $\partial A_t$, we focus on geodesic ray. 

\begin{lem}
For a geodesic ray $\Phi$ with continuous $\dot{\Phi}$ and for a continuous $\Psi$, the relaxed action $\mathcal{A}_\Phi^\Psi$ is continuous on $[0, \infty)$. 
\end{lem}

\begin{proof}
As we already noted, we have $(dd_\omega^c \phi_{t_i})^k \to (dd_\omega^c \phi_t)^k$ weakly as currents of order $0$. 
By our assumption, $\psi_{t_i}$ (resp. $\dot{\phi}_{t_i}$) are continuous and converge uniformly to $\psi_t$ (resp. $\dot{\phi}_t$) as $t_i \to t$. 
The desired continuity follows from the following fact: for a uniformly convergent sequence $f_i \to f$ of continuous functions and a weakly convergent sequence $\mu_i \to \mu$ of Radon measures on $X$, we have $\int_X f_i \mu_i \to \int_X f \mu$. 
Indeed, we have 
\begin{align*} 
\left| \int_X f_i \mu_i - \int_X f \mu \right| 
&\le \left| \int_X f \mu_i - \int_X f \mu \right| + \int_X |f- f_i| \mu_i 
\\
&\le \left| \int_X f \mu_i - \int_X f \mu \right| + \| \mu_i \| \| f- f_i \|_{\sup}. 
\end{align*}
Since $\mu_i$ is weakly convergent, the uniform boundedness principle on $\| \mu_i \|$ shows $\| \mu_i \| \le C$. 
This shows the convergence. 
\end{proof}


\begin{prop}
For a $C^{1,1}$-ray $\Phi$ (not necessarily a subgeodesic) and a smooth ray $\Psi$ (up to the boundary), $\mathcal{A}_\Phi^\Psi$ is $C^1$ (up to the boundary). 
In this case, we have 
\[ \frac{d}{dt} \mathcal{A}^\Psi_\Phi (t) = -\int_X (dd^c_\sigma \psi_t - \dot{\psi}_t) e^{dd^c_\omega \phi_t - \dot{\phi}_t}. \]
\end{prop}

\begin{proof}
The claim holds for smooth $\Phi$ and $\Psi$. 
Take a $C^{1,1}_{\mathrm{loc}}$-approximation $\{ \Phi_i \}$ of $\Phi$ by smooth rays. 
The claim follows from the following pointwise locally uniform convergence: 
\begin{gather*}
\mathcal{A}^\Psi_{\Phi_i} \to \mathcal{A}^\Psi_\Phi, 
\\
\int_X (dd^c_\sigma \psi - \dot{\psi}) e^{dd^c_\omega \phi_i - \dot{\phi}_i} \to \int_X (dd^c_\sigma \psi - \dot{\psi}) e^{dd^c_\omega \phi - \dot{\phi}}. 
\end{gather*}
\end{proof}

We denote $\mathcal{A}^\Psi_\Phi$ by $\mathcal{A}_{\bm{\phi}}^{\bm{\psi}}$ in the rest of this article. 
We note for a $C^{1,1}$-regular $\phi$, ``almost everywhere with respect to the Lebesgue measure'' implies ``almost every where with respect to the measure $(dd^c \phi)^n$''. 

\subsubsection{The action functional}

For a $C^{1,1}$-geodesic ray $\bm{\phi} = \{ \phi_t \}_{t \in [0, \infty)}$ ($C^{1, \bar{1}}$-regularity suffices in this section), $v_t := \omega_{\phi_t}^n/\omega^n$ is a non-negative $L^\infty$-function, so that $v_t \log v_t$ gives an $L^\infty$-function. 
Then we put 
\begin{equation} 
\mathcal{A}_{\bm{\phi}} (t) := \frac{1}{n!} \int_X v_t \log v_t e^{- \dot{\phi}_t} \omega^n + \frac{1}{n!} \int_0^t ds \int_X n \mathrm{Ric} (\omega) \wedge e^{-\dot{\phi}_s} \omega_{\phi_s}^{n-1}. 
\end{equation}
Using $\mu_t = e^{-\dot{\phi}_t} \omega_{\phi_t}^n$, we get the following expression of the first part: 
\[ \int_X v_t \log v_t e^{- \dot{\phi}_t} \omega^n = \int_X \frac{d\mu_t}{d\mu_0} \log \frac{d\mu_t}{d\mu_0} d\mu_0 - \int_X (\dot{\phi}_t - \dot{\phi}_0) e^{-\dot{\phi}_t} \omega_{\phi_t}^n. \]
The term $\int_X \frac{d\mu_t}{d\mu_0} \log \frac{d\mu_t}{d\mu_0} d\mu_0$ is known to be lower semi-continuous with respect to the weak convergence of measures. 
The rest terms are $C^1$ when $\dot{\Phi}$ is continuous (up to the boundary). 
Therefore, $\mathcal{A}_{\bm{\phi}}$ is lower semi-continuous (up to the boundary). 

Since the entropy part is non-negative, we have the following. 

\begin{lem}
\label{W plus}
Let $\bm{\phi} = \{ \phi_s \}_{s \in [0, \infty)}$ be a $C^{1, 1}$-geodesic ray emanating from a smooth initial metric $\omega$. 
Then we have 
\[ -\frac{\frac{d}{dt}_+\Big{|}_{t =0} \mathcal{A}_{\bm{\phi}} (t)}{\int_X e^{- \dot{\phi}_0} \omega^n/n!} \le \cW (\omega, -\dot{\phi}_0). \]
\end{lem}

\begin{proof}
As the entropy $\int_X \frac{d\mu_t}{d\mu_0} \log \frac{d\mu_t}{d\mu_0} d\mu_0$ is non-negative and is zero when $t= 0$, we have 
\[ \frac{d}{dt}_+\Big{|}_{t=0} \int_X v_t \log v_t e^{- \dot{\phi}_t} \omega^n \ge \frac{d}{dt}_+\Big{|}_{t=0} \int_X \dot{\phi}_0 e^{-\dot{\phi}_t} \omega_{\phi_t}^n. \]
Take a $C^{1,1}_{\mathrm{loc}}$-approximation $\{ \Phi_i \}$ of the geodesic $\Phi$ by smooth rays. 
Each $f_i (t) = \int_X \dot{\phi}_{i, 0} e^{-\dot{\phi}_{i, t}} \omega_{\phi_{i, t}}^n$ is $C^\infty$ on $[0, \infty)$ and its derivative is given by 
\[ f_i' (t) = -\int_X \dot{\phi}_{i, 0} e^{-\dot{\phi}_{i, t}} (\ddot{\phi}_{i, t} \omega_{\phi_{i, t}}^n - n d\dot{\phi}_{i, t} \wedge d^c \dot{\phi}_{i, t} \wedge \omega_{\phi_{i, t}}^{n-1}) + \int_X n d \dot{\phi}_{i, 0} \wedge d^c \dot{\phi}_{i, t} e^{-\dot{\phi}_{i, t}} \omega_{\phi_{i, t}}^{n-1}. \]
We have the pointwise locally uniform convergence 
\begin{gather*}
f_i (t) \to \int_X \dot{\phi}_0 e^{-\dot{\phi}_t} \omega_{\phi_t}^n,
\\
f_i' (t) \to \int_X n d \dot{\phi}_0 \wedge d^c \dot{\phi}_t e^{-\dot{\phi}_t} \omega_{\phi_t}^{n-1}. 
\end{gather*}
Thus the functional $\int_X \dot{\phi}_0 e^{-\dot{\phi}_t} \omega_{\phi_t}^n$ is $C^1$ on $[0, \infty)$ and 
\[ \frac{d}{dt}_+ \Big{|}_{t=0} \int_X \dot{\phi}_0 e^{-\dot{\phi}_t} \omega_{\phi_t}^n = \int_X |\partial^\sharp \dot{\phi}_0|^2 e^{-\dot{\phi}_0} \omega^n. \]
We similarly obtain 
\[ \frac{d}{dt}_+ \Big{|}_{t=0} \int_0^t ds \int_X n \mathrm{Ric} (\omega) \wedge e^{-\dot{\phi}_s} \omega_{\phi_s}^{n-1} = \int_X s (\omega) e^{-\dot{\phi}_0} \omega^n, \]
which shows the claim. 
\end{proof}

\subsubsection{Proof of the convexity}

The rest of this section is devoted to the proof of the following. 

\begin{thm}
For a $C^{1,1}$-geodesic ray $\bm{\phi}$, the action functional $\mathcal{A}_{\bm{\phi}}$ is continuous and pointwise convex on $[0, \infty)$
\end{thm}

To show the pointwise convexity of $\mathcal{A}_{\bm{\phi}}$, we make use of the following convergence similarly as in \cite{BB}. 

\begin{itemize}
\item If a locally uniformly bounded sequence $\psi_{i, t}$ converges to $\psi_t$ almost everywhere with respect to the Lebesgue measure on $X$ for each $t$, we have $\mathcal{A}_{\bm{\phi}}^{\bm{\psi}_i} (t) \to \mathcal{A}_{\bm{\phi}}^{\bm{\psi}} (t)$ by the dominated convergence theorem. 

\item If $\psi_{i, t}$ is a decreasing sequence of bounded functions pointwisely converging to $\log v_t$, we have $\mathcal{A}_{\bm{\phi}}^{\bm{\psi}_i} (t) \to \mathcal{A}_{\bm{\phi}} (t)$ again by the dominated convergence theorem, which we can apply as $v_t \psi_{1, t} \ge v_t \psi_{i, t} \ge v_t \log v_t \ge -e^{-1}$. 
\end{itemize}

Note also the following generalities on convex function: 
\begin{itemize}
\item An upper semi-continuous function $f$ on $[0, \infty)$ is convex (and hence continuous in the interior) iff $F (z) := f (\log |z|)$ satisfies $dd^c F \ge 0$ on $A^\circ$ in the weak sense of current. (weakly subharmonic)

\item A pointwise limit of convex functions on $[0, \infty)$ is automatically convex (possibly not continuous on the boundary). 

\item A pointwise convex function on $[0, \infty)$ is continuous in the interior and upper-semi continuous on the boundary. 
\end{itemize}

By these, the convexity and the continuity of $\mathcal{A}_{\bm{\phi}}$ up to the boundary reduces to the construction of a collection $\{ \Psi_{B, j} \}_{B, j \in \mathbb{N}}$ of $U(1)$-invariant continuous functions on $X \times A$ (up to the boundary) with the following properties: 
\begin{itemize}
\item $\mathscr{A}_{\bm{\phi}}^{\bm{\psi}_{B, j}} (\tau) := \mathcal{A}_{\bm{\phi}}^{\bm{\psi}_{B, j}} (\log |\tau|)$ is weakly subharmonic for each $B, j \in \mathbb{N}$. 

\item For each $B$, the sequence $\{ \Psi_{B, j} \}_j$ is locally uniformly bounded and converges to a locally bounded function $\Psi_B$ almost everywhere with respect to the Lebesgue measure on $X \times \{ t \}$ for each $t$. 

\item $\Psi_B$ is a decreasing sequence converging to $\Psi (z) := \log (\omega_{\phi_{\log |z|}}^n/\omega^n)$ pointwisely. 
\end{itemize}

Now, we check the sequence constructed in \cite{BB} is the desired one. 
Indeed, they constructed a collection $\{ \Psi_{B, j} \}_{B, j \in \mathbb{N}}$ of \textit{Lipschitz functions} satisfying the last two conditions, plus the following property:  
\[ dd^c_\sigma \Psi_{B, j} \wedge (dd^c_\omega \Phi)^n := \sigma \wedge  (dd^c_\omega \Phi)^n + dd^c (\Psi_{B, j} (dd^c_\omega \Phi)^n) \ge 0 \]
for $\sigma = - \mathrm{Ric} (\omega)$. 
(We note in the notation of \cite{BB}, their $\Psi_{B, j}$ stands for our $\psi + \Psi_{B, j}$, taking a local potential $\psi$ of $\sigma = dd^c \psi$. )
Thus the following is the last piece in the proof of the convexity. 

\begin{thm}
\label{key claim for convexity}
Let $\Phi$ be the $C^{1,1}$-geodesic ray and $\Psi$ be a Lipschitz function on $X \times A_\infty^\circ$ with $dd^c_\sigma \Psi \wedge (dd^c_\omega \Phi)^n := \sigma \wedge  (dd^c_\omega \Phi)^n + dd^c (\Psi (dd^c_\omega \Phi)^n) \ge 0$. 
We have the following equality: 
\[ dd^c \mathscr{A}_{\bm{\phi}}^{\bm{\psi}} = \frac{1}{n!} \varpi_* (e^{-\dot{\Phi}} dd^c_\sigma \Psi \wedge (dd^c_\omega \Phi)^n). \]
Here the product current $e^{-\dot{\Phi}} dd^c_\sigma \Psi \wedge (dd^c_\omega \Phi)^n$ is well-defined as $dd^c_\sigma \Psi \wedge (dd^c_\omega \Phi)^n$ is a current of order $0$ by the positivity assumption.  

In particular, $\mathscr{A}_{\bm{\phi}}^{\bm{\psi}}$ is weakly subharmonic (since its $dd^c$ is the push-forward of a positive current). 
\end{thm}

In the proof, we make use of the $C^{0,1}$-regularity of each $\Psi_{B, j}$, not only its continuity. 
This regularity readily follows from the construction: recall we put 
\[ \Psi_{B, j} (x, \tau) := \max \{ \log \frac{\beta_{j, \log |\tau|} (x)}{\omega^n (x)} ,-B \}, \]
where $\beta_{j, t}$ is the Bergman measure defined as follows. 
Let $(s_{j, t}^k)_k \subset H^0 (K_X +j L)$ be an orthogonal basis with respect to the hermitian metric: 
\[ (s, s')_t = (\sqrt{-1})^{n^2} \int_X s \wedge \bar{s}' e^{-j \phi_t}. \]
Then the top form $\beta_{j, t}$ is defined as 
\[ \beta_{j, t} = \frac{n!}{j^n} (\sqrt{-1})^{n^2} \sum_k s_{j, t}^k \wedge \bar{s}_{j, t}^k e^{-j \phi_t}, \]
which is $C^{1,1}$ as $\phi_t$ is so. 
This shows the Lipschitz regularity of $\Psi_{B, j}$. 

Now we prepare the following key computation for smooth $\Phi$ and $\Psi$. 

\begin{prop}
For the functional $\mathscr{A}^{\bm{\psi}}_{\bm{\phi}} (\tau) := \mathcal{A}^{\bm{\psi}}_{\bm{\phi}} (\log |\tau|)$ on the annulus $A_\infty^\circ = \{ 1 < |z| < \infty \}$, we have 
\[ d d^c \mathscr{A}^{\bm{\psi}}_{\bm{\phi}} = \varpi_* \Big{(} d_{U (1)} d_\sigma^c \Psi \wedge e^{d_{U (1)} d_\omega^c \Phi}; \pi^{-1} \eta \Big{)}. \]
Here we put $\Psi (x, \tau) := \psi_{\log |\tau|} (x)$, $d_{U (1)} d_\sigma^c \Psi := \sigma + d_{U (1)} d^c \Psi = \omega + dd^c \Psi + i_\eta d^c \Psi. \eta^\vee$ (similarly for $\Phi$) and consider the $U (1)$-action on $A_\infty$ by the anti-scalar multiplication: its fundamental vector is $\eta = - 2\pi \partial/ \partial \theta$. 
\end{prop}

Remark: If we put $\Psi (x, \tau) = \psi_{-\log |\tau|} (x)$, the $U (1)$-action on $A_\infty$ must be reversed to the usual scalar multiplication. 

\begin{proof}
For a function $f (r e^{i \theta})$ on $A$, we have 
\[ d^c f = f_r d^c r + f_\theta d^c \theta = \frac{1}{2} (r f_r d\theta - f_\theta r^{-1} dr). \]
For a $U(1)$-invariant function $\Psi (x, r e^{i \theta}) = \psi_{\log r} (x)$ on $X \times A$, we have $d^c \Psi = d^c_X \psi_{\log r} + \dot{\psi}_{\log r} \frac{d\theta}{2}$, so we get 
\begin{gather*}
i_{\eta} d^c \Psi = - \pi \dot{\psi}_{\log r},
\\
dd^c_\sigma \Psi = d_X d_{X, \sigma}^c \psi_{\log r} + \gamma_\psi, 
\end{gather*}
where we put 
\[ \gamma_\psi := -d_X^c \dot{\psi}_{\log r} \wedge r^{-1} dr + d_X \dot{\psi}_{\log r} \wedge \frac{d\theta}{2} + \ddot{\psi}_{\log r} \frac{r^{-1} d r \wedge d \theta}{2}. \]
We have 
\[ \gamma_\phi \wedge \gamma_\psi = -(d_X \dot{\phi} \wedge d_X^c \dot{\psi} + d_X \dot{\psi} \wedge d_X^c \dot{\phi}) \wedge \frac{r^{-1} dr \wedge d \theta}{2} \]
and $\gamma_\phi^2 \wedge \gamma_\psi = 0$. 
We also have $\nu \wedge \gamma_\psi = \ddot{\psi}. \nu \wedge \frac{r^{-1} dr \wedge d \theta}{2}$ for $2n$-form $\nu$ on $X$. 
We put $\sigma_\psi := d_X d^c_{X, \sigma} \psi$. 

Using this, we compute 
\begin{align*}
\Big{[} d_{U (1)} d_\sigma^c \Psi 
&\wedge e^{d_{U (1)} d_\omega^c \Phi}; x \eta \Big{]}_{n+1} = [((\sigma_\psi + \gamma_\psi) - \pi x \dot{\psi}) \wedge e^{(\omega_\phi + \gamma_\phi) - \pi x \dot{\phi}}]_{n+1}
\\
&= \frac{1}{n!} e^{- \pi x \dot{\phi}} (\sigma_\psi + \gamma_\psi) (\omega_\phi + \gamma_\phi)^n - \frac{1}{(n+1)!} \pi x \dot{\psi} e^{- \pi x \dot{\phi}} (\omega_\phi + \gamma_\phi)^{n+1}
\\
&= \frac{1}{n!} e^{- \pi x \dot{\phi}} \Big{(} n \sigma_\psi \wedge \omega_\phi^{n-1} \wedge \gamma_\phi + \binom{n}{2} \sigma_\psi \wedge \omega_\phi^{n-2} \wedge \gamma_\phi^2 
\\
&\qquad \qquad \qquad + \gamma_\psi \wedge \omega_\phi^n + n \gamma_\psi \wedge \omega_\phi^{n-1} \wedge \gamma_\phi - \pi x \dot{\psi} \omega_\phi^n \wedge \gamma_\phi - \pi x \dot{\psi} \frac{n}{2} \omega^{n-1}_\phi \wedge \gamma_\phi^2 \Big{)}
\\
&= \frac{1}{n!} e^{- \pi x \dot{\phi}} \frac{r^{-1} dr  \wedge d \theta}{2}
\\
&\qquad \quad \wedge \Big{[} n \ddot{\phi} \sigma_\psi \wedge \omega_\phi^{n-1} + \ddot{\psi} \omega_\phi^n - \mathrm{tr}_\phi (d_X \dot{\phi} \wedge d_X^c \dot{\psi} + d_X \dot{\psi} \wedge d_X^c \dot{\phi}) \omega_\phi^n
\\
&\qquad \qquad \qquad - \pi x \dot{\psi} \Big{(} \ddot{\phi} - \mathrm{tr}_\phi (d_X \dot{\phi} \wedge d_X^c \dot{\phi}) \Big{)} \omega_\phi^n
\\
&\qquad \qquad \qquad \quad - n (n-1) d_X \dot{\phi} \wedge d_X^c \dot{\phi} \wedge \sigma_\phi \wedge \omega_\phi^{n-2} \Big{]} 
\\
&= \frac{1}{n!} \frac{r^{-1} dr \wedge d \theta}{2}
\\
&\qquad \quad \wedge \Big{[} n \ddot{\phi} \sigma_\psi \wedge e^{-\pi x \dot{\phi}} \omega_\phi^{n-1} + \Big{(} \ddot{\psi} - \nabla_X \dot{\phi} (\dot{\psi}) - \pi x \dot{\psi} (\ddot{\phi} - |\partial_X^\sharp \dot{\phi}|^2) \Big{)} e^{- \pi x \dot{\phi}} \omega_\phi^n
\\
&\qquad \qquad \qquad \quad - \frac{1}{\pi x} e^{- \pi x \dot{\phi}} n (n-1) d_X d_X^c \dot{\phi} \wedge \sigma_\psi \wedge \omega_\phi^{n-2} \Big{]}
\\
&\qquad + \frac{1}{n!} \frac{r^{-1} dr  \wedge d \theta}{2} \wedge \frac{1}{\pi x} d_X (e^{- \pi x \dot{\phi}} n (n-1) d_X^c \dot{\phi} \wedge \sigma_\psi \wedge \omega_\phi^{n-2})
\\
&= \frac{r^{-1} dr  \wedge d \theta}{2}
\\
&\qquad \quad \wedge \Big{(} \sigma_\psi \ddot{\phi} + \ddot{\psi} - \nabla_X \dot{\phi} (\dot{\psi}) - \pi x \dot{\psi} (\ddot{\phi} - |\partial^\sharp \dot{\phi}|^2) - \frac{1}{\pi x} \sigma_\psi \wedge  d_X d_X^c \dot{\phi} \Big{)} \wedge e^{\omega_\phi - \pi x \dot{\phi}} 
\\
&\qquad \qquad + \frac{1}{n!} \frac{r^{-1} dr \wedge d \theta}{2} \wedge \frac{1}{\pi x} d_X (e^{- \pi x \dot{\phi}} n (n-1) d_X^c \dot{\phi} \wedge \sigma_\psi \wedge \omega_\phi^{n-2}). 
\end{align*}

Substituting $x = \pi^{-1}$, the integration along fibre is computed as follows: 
\begin{align*}
\int_X \Big{(} \sigma_\psi \ddot{\phi} 
&+ \ddot{\psi} - \nabla \dot{\phi} (\dot{\psi}) - \dot{\psi} (\ddot{\phi} - |\partial^\sharp \dot{\phi}|^2) - \sigma_\psi \wedge d_X d_X^c \dot{\phi} \Big{)} \wedge e^{\omega_\phi - \dot{\phi}}
\\
&= -\int_X \Big{(} (d_X d_X^c \dot{\psi} - \ddot{\psi}) + (\sigma_\psi - \dot{\psi}) \wedge (d_X d_X^c \dot{\phi} -\ddot{\phi}) \Big{)} \wedge e^{\omega_\phi - \dot{\phi}}
\\
&= - \frac{d}{dt} \int_X (\sigma_\psi - \dot{\psi}) \wedge e^{\omega_\phi - \dot{\phi}}
\\
&= \frac{d^2}{dt^2} \mathcal{A}^{\bm{\psi}}_{\bm{\phi}} (t), 
\end{align*}
where we used 
\begin{gather*} 
\int_X \nabla \dot{\phi} (\dot{\psi}) e^{\omega_\phi -\dot{\phi}} = -2 \int_X \bar{\Box} \dot{\psi} e^{\omega_\phi -\dot{\phi}} = 2 \int_X dd^c \dot{\psi} \wedge e^{\omega_\phi - \dot{\phi}},  
\\
\ddot{\phi} - |\partial^\sharp \dot{\phi}|^2 = (\ddot{\phi} + \bar{\Box} \dot{\phi}) - (\bar{\Box} \dot{\phi} + |\partial^\sharp \dot{\phi}|^2), 
\\
\int_X \dot{\psi} (\ddot{\phi} + \bar{\Box} \dot{\phi}) e^{\omega_\phi - \dot{\phi}} = \int_X \dot{\psi} (\ddot{\phi} - dd^c \dot{\phi}) \wedge e^{\omega_\phi - \dot{\phi}}, 
\\
\int_X \dot{\psi} (\bar{\Box} \dot{\phi} + |\partial^\sharp \dot{\phi}|^2) e^{\omega_\phi -\dot{\phi}} = - \int_X \bar{\Box} \dot{\psi} e^{\omega_\phi -\dot{\phi}}  = \int_X dd^c \dot{\psi} \wedge e^{\omega_\phi -\dot{\phi}}. 
\end{gather*}
This proves the claim. 
\end{proof}

\begin{proof}[Proof of Theorem \ref{key claim for convexity}]
We firstly note for a smooth ray $\bm{\psi}$ and for a $C^{1,1}_{\mathrm{loc}}$-approximation $\{ \bm{\phi}_i \}_i$ of the geodesic ray $\bm{\phi}$ by smooth rays, we have the locally uniform convergence $\mathcal{A}_{\bm{\phi}_i}^{\bm{\psi}} \to \mathcal{A}_{\bm{\phi}}^{\bm{\psi}}$. 
So we get the weak convergence $dd^c \mathcal{A}_{\bm{\phi}_i}^{\bm{\psi}} \to dd^c \mathcal{A}_{\bm{\phi}}^{\bm{\psi}}$ of currents. 
On the right hand side, we get the weak convergence of currents
\[ ( d_{U (1)} d_\sigma^c \Psi \wedge e^{d_{U (1)} d_\omega^c \Phi_i}; \pi^{-1} \eta ) \to e^{-\dot{\Phi}} dd_\sigma^c \Psi \wedge (dd_\omega^c \Phi)^n \]
by the continuity of the Monge--Ampere operator with respect to the uniform convergence, which especially implies $(dd_\omega^c \Phi_i)^{n+1} \to (dd_\omega^c \Phi)^{n+1} = 0$. 
Thus the theorem holds for smooth $\bm{\psi}$. 

Next, for a Lipschitz ray $\bm{\psi}$, take a $C^{0,1}_{\mathrm{loc}}$-approximation $\{ \bm{\psi}_i \}_i$ of $\bm{\psi}$ by smooth rays. 
We easily see the locally uniform convergence $\mathcal{A}_{\bm{\phi}}^{\bm{\psi}_i} \to \mathcal{A}_{\bm{\phi}}^{\bm{\psi}}$, hence also the weak convergence $dd^c \mathcal{A}_{\bm{\phi}}^{\bm{\psi}_i} \to dd^c \mathcal{A}_{\bm{\phi}}^{\bm{\psi}}$ of currents. 

On the other hand, we compute the right hand side as 
\[ e^{-\dot{\Phi}} d d^c_\sigma \Psi_i \wedge (dd_\omega^c \Phi)^n = e^{-\dot{\Phi}} \sigma \wedge (dd_\omega^c \Phi)^n + d (e^{-\dot{\Phi}} d^c \Psi_i \wedge (dd_\omega^c \Phi)^n) + e^{-\dot{\Phi}} d \dot{\Phi} \wedge d^c \Psi_i \wedge (dd^c_\omega \Phi)^n. \]
Since $\{ \Psi_i \}_i$ is a $C^{0,1}_{\mathrm{loc}}$-approximation of $\Psi$ and $\dot{\Phi}$ is Lipschitz, we have  
\begin{gather*} 
e^{-\dot{\Phi}} d^c \Psi_i \wedge (dd_\omega^c \Phi)^n \to e^{-\dot{\Phi}} d^c \Psi \wedge (dd_\omega^c \Phi)^n, 
\\
e^{-\dot{\Phi}} d \dot{\Phi} \wedge d^c \Psi_i \wedge (dd^c_\omega \Phi)^n \to e^{-\dot{\Phi}} d \dot{\Phi} \wedge d^c \Psi \wedge (dd^c_\omega \Phi)^n
\end{gather*}
as currents. 
Thus we get 
\[ e^{-\dot{\Phi}} d d_\sigma \Psi_i \wedge (dd_\omega^c \Phi)^n \to e^{-\dot{\Phi}} \sigma \wedge (dd_\omega^c \Phi)^n + d (e^{-\dot{\Phi}} d^c \Psi \wedge (dd_\omega^c \Phi)^n) + e^{-\dot{\Phi}} d \dot{\Phi} \wedge d^c \Psi \wedge (dd^c_\omega \Phi)^n \]
as currents. 

Finally, we note 
\[ G S + d (G d^c F \wedge T) = dG \wedge d^c F \wedge T + G (S+dd^c (F T)) \]
for a closed $(k,k)$-current $T$ with $L^\infty$-coefficients, a $(k+1,k+1)$-current $S$ with $L^\infty$-coefficients and Lipschitz functions $F, G$ satisfying $S + dd^c (FT) \ge 0$ (hence the product $G (S+dd^c (FT))$ is well-defined). 
We can check this as follows. 
If we take a $C^{0,1}$-approximation $G_i \to G$ by smooth functions, $dG_i$ converges to $dG$ in the weak$^*$ topology of $L^\infty = (L^1)^*$, so we have the following weak convergence of each term by the regularity/order assumptions: 
\begin{gather*} 
G_i S \to GS, \quad G_i d^c F \wedge T \to G d^c F \wedge T, 
\\ 
dG_i \wedge d^c F \wedge T \to dG \wedge d^c F \wedge T, \quad G_i (S+dd^c (F T)) \to G (S+dd^c (F T)). 
\end{gather*}
So by regularizing $G$, we may assume $G$ is smooth. 
Then the formula reduces to $d^c F \wedge T = d^c (F T)$. 
Similarly regularizing $F$, we may assume $F$ is smooth, in which case the formula is evident. 

Using this, we compute 
\begin{align*} 
e^{-\dot{\Phi}} \sigma \wedge (dd_\omega^c \Phi)^n 
&+ d (e^{-\dot{\Phi}} d^c \Psi \wedge (dd_\omega^c \Phi)^n) + e^{-\dot{\Phi}} d \dot{\Phi} \wedge d^c \Psi \wedge (dd_\omega^c \Phi)^n
\\
&= e^{-\dot{\Phi}} (\sigma \wedge (dd_\omega^c \Phi)^n + dd^c (\Psi (dd^c_\omega \Phi)^n))
\\
&= e^{-\dot{\Phi}} dd_\sigma^c \Psi \wedge (dd^c_\omega \Phi)^n 
\end{align*}
and complete the proof. 
\end{proof}

\begin{rem}
For a uniformly convergent sequence $\Psi_i \to \Psi$, $\Psi_i (dd^c \Phi)^n$ converges weakly to $\Psi (dd^c \Phi)^n$ as $(dd^c \Phi)^n$ is order $0$. 
This implies that $dd^c_\sigma \Psi_i \wedge (dd^c \Phi)^n = dd^c_\sigma (\Psi_i (dd^c \Phi)^n)$ weakly converges to $dd^c_\sigma \Psi \wedge (dd^c \Phi)^n := dd^c_\sigma (\Psi (dd^c \Phi)^n)$ as currents. 
However, this does not readily imply the product $e^{-\dot{\Phi}} dd^c_\sigma \Psi_i \wedge (dd^c \Phi)^n$ converges to the product $e^{-\dot{\Phi}} dd^c_\sigma \Psi \wedge (dd^c \Phi)^n$ as $\dot{\Phi}$ is only Lipschitz continuous. 
In the above proof, we discussed with ''the integration by parts'' and made use of the Lipschitz regularity of $\dot{\Phi}, \Psi$ to show this convergence. 

We note the above proof can be relaxed to $C^{1, \bar{1}}$-geodesics by carefully tracking with the boundedness of $(dd^c \Phi)^n$ and $L^p_2$-regularity of $\Phi$ for every $1 < p < \infty$. 
\end{rem}

\subsection{Slope formula of the action functional}
\label{Slope formula}

\subsubsection{Preparations}

\begin{lem}[Equivariant $\partial \bar{\partial}$-lemma]
Let $X$ be a compact K\"ahler manifold endowed with a Hamiltonian holomorphic aciton of a compact Lie group $K$, where Hamiltonian means that there is a $K$-invariant K\"ahler form $\omega$ admitting a moment map. 
\begin{enumerate}
\item For a $d_K$-exact equivariant $(1,1)$-form $\sigma + \nu$, the smooth function $f$ with $\sigma = dd^c f$, which exists by the usual $\partial \bar{\partial}$-lemma, satisfies $\sigma + \nu = d_K d^c f$. 

\item For a $K$-invariant divisor $D$, take a smooth $K$-equivariant closed $(1,1)$-form $\Delta +\delta$ in the equivariant cohomology class $[D^K]$. 
Then there is an integrable function $G$ on $X$ which is smooth away from $D$ such that $(\Delta +\delta) + d_K d^c G = D^K$ as equivariant currents. 
\end{enumerate}
\end{lem}

\begin{rem}
For a $K$-invariant divisor $D$, $D + \mu$ for a constant $\mu \in (\mathfrak{k}^\vee)^K$ gives a $d_K$-closed equivariant current. 
Indeed, for $\xi \in \mathfrak{k}$, we have $i_\xi D = 0$ as currents: for a test $2n-1$-form $\varphi$, we compute 
\[ (i_\xi D, \varphi) = \pm (D, i_\xi \varphi) = \pm \int_D (i_\xi \varphi)|_D = \pm \int_D i_{\xi|_D} (\varphi|_D) = 0, \]
where we used the $K$-invariance of $D$. 

We can construct equivariant cohomology via equivariant currents (cf.  \cite{Ino3, GS}), so that the $d_K$-closed equivariant current $D^K := D + 0$ defines the equivariant cohomology class $[D^K]$. 
\end{rem}

\begin{proof}
(1) For a $K$-equivariant $1$-form, which is just a $K$-invariant $1$-form, we have $(d_K \gamma; \xi) = d \gamma + \gamma (\xi)$ for $\xi \in \mathfrak{k}$. 
Suppose $d \gamma = 0$, then by $K$-invariance, we have $L_\xi \gamma = d (\gamma (\xi)) = 0$, so that $\gamma (\xi)$ is constant. 
Since the action is Hamiltonian, $\xi$ has a non-empty zero set. 
It follows that a $d_K$-exact equivariant $2$-form $\sigma + \nu$ is zero iff $\sigma = 0$. 
Since $\sigma' + \nu' := (\sigma + \nu) - d_K d^c f$ is a $d_K$-exact equivariant $2$-form with $\sigma' = 0$, it must be zero as equivariant form. 

We note the proof works also for $d_K$-closed equivariant $(1,1)$-current. 

(2) The claim is well-known in the non-equivariant case: we have an integrable function $G$ smooth away from $D$ such that $\Delta + dd^c G = D$. 
Then since $\sigma' + \nu' = (\Delta + \delta) + d_K d^c G - D^K$ is $d_K$-exact with $\sigma' = 0$, we must have $\nu' = 0$ by the above argument, which proves the claim. 
\end{proof}

\begin{lem}[Equivariant Stokes theorem]
Let $\mathcal{X}$ be a manifold with boundary and $K$ be a compact Lie group acting smoothly on $\mathcal{X}$. 
Let $\alpha, \gamma$ be a $K$-equivariant differential form with Lipschitz coefficients and $\beta$ be a $K$-equivariantly closed differential form with $C^\infty$-coefficients. 
Then we have 
\[ \int_{\mathcal{X}} d_K \alpha \wedge (\beta + d_K \gamma)^l = \int_{\partial \mathcal{X}} \alpha \wedge (\beta + d_K (\gamma|_{\partial \mathcal{X}}))^l. \]
\end{lem}

\begin{proof}
This is well-known when $\alpha$ and $\gamma$ is smooth. 
For Lipschitz $\alpha$ and $\gamma$, we can take approximation by smooth equivariant forms $\{ \alpha_i \}_i, \{ \gamma_i \}_i$ in $C^{0,1}$. 
When restricted to the boundary, $\alpha_i|_{\partial \mathcal{X}}$, $\gamma_i|_{\partial \mathcal{X}}$ are still approximation in $C^{0,1}$ of $\alpha|_{\partial \mathcal{X}}$, $\gamma|_{\partial \mathcal{X}}$, respectively. 
Thus we have 
\begin{gather*}
\int_X d_K \alpha_i \wedge (\beta + d_K \gamma_i)^l \to \int_X d_K \alpha \wedge (\beta + d_K \gamma)^l, 
\\
\int_{\partial \mathcal{X}} \alpha_i \wedge (\beta + d_K (\gamma_i|_{\partial \mathcal{X}}))^l \to \int_{\partial \mathcal{X}} \alpha \wedge (\beta + d_K (\gamma|_{\partial \mathcal{X}}))^l. 
\end{gather*} 
Now the claim follows from the smooth case. 
\end{proof}

\begin{rem}
We will apply this to $\gamma = d^c \phi$ with $C^{1,1}$-regular $\phi$. 
For the weaker regularity $C^{1, \bar{1}}$, the current $d ((d^c \phi)|_{\partial \mathcal{X}})$ may not be bounded. 
Actually, it is not even clear if the restriction $(d^c \phi)|_{\partial \mathcal{X}}$ is in $L^p_1$ as the trace operator $L^p_1 (\mathcal{X}) \to W^{1-1/p, p} (\partial \mathcal{X})$ losses regularity. 
\end{rem}

Let $X$ be a compact K\"ahler manifold. 
Let $\mathcal{X}$ be a test configuration of $X$ and $\mathcal{M}$ be a $U(1)$-equivariant class in $H^{1,1}_{U (1)} (\mathcal{X}, \mathbb{R})$. 
Take a resolution $\beta: \tilde{\mathcal{X}} \to \mathcal{X}$ so that $\tilde{\mathcal{X}}$ dominates the trivial configuration by $\rho: \tilde{\mathcal{X}} \to X \times \mathbb{C}P^1$. 
Take a smooth $U (1)$-equivariant $(1,1)$-form $\Sigma + \nu \in \beta^* \bar{\mathcal{M}}$ on $\tilde{\mathcal{X}}$ and a smooth $(1,1)$-form $\sigma$ on $X$ such that $[\sigma] = i^* \mathcal{M} \in H^2 (X, \mathbb{R})$. 

Let $j_-: X \times \mathbb{C}_- \hookrightarrow \bar{\mathcal{X}}$ denote the natural inclusion. 
We can construct a smooth function $\Psi_0$ on $X \times \mathbb{C}_{-}$ so that $\sigma + d_{U (1)} d^c \Psi_0 = j_-^* (\Sigma + \nu)$ on $X \times \mathbb{C}_{-}$ as follows. 
Note by the following exact sequence (cf. \cite{Ino3})
\[ H^{\mathrm{lf}}_{2n} (\mathcal{X}_0) \to H_{U (1)}^2 (\tilde{\mathcal{X}}) \to H^2_{U (1)} (X \times \mathbb{C}_{-}), \]
the equivariant form $\beta^* \bar{\mathcal{M}} - \rho^* [\sigma]$ can be written as $[D^{U (1)}]$ for a $U (1)$-invariant Cartier divisor $D$ on $\tilde{\mathcal{X}}$ supported on the central fibre. 
By the above lemma, we have a smooth $U (1)$-equivariant $(1,1)$-form $\Delta + \delta$ on $\tilde{\mathcal{X}}$ representing $[D^{U (1)}]$ and an integrable function $G$ on $\tilde{\mathcal{X}}$ which is smooth away from the central fibre such that $(\Delta + \delta) + d_{U (1)} d^c G = D$ as currents. 
On the other hand, we have a smooth function $F$ on $\tilde{\mathcal{X}}$ such that $(\beta^* \Sigma + \beta^* \nu) - \rho^* \sigma - (\Delta + \delta) = d_{U (1)} d^c F$. 
Thus we get $(\beta^* \Sigma + \beta^* \nu) - \rho^* \sigma - (\Delta + \delta) - d_{U (1)} d^c G = d_{U (1)} d^c (F -G)$ as currents on $\tilde{\mathcal{X}}$. 
Since $(\Delta + \delta) + d_{U (1)} d^c G$ is zero away from the central fibre, we get $j_-^* (\Sigma + \nu) = \sigma + d_{U (1)} d^c (F -G)$ on $\tilde{\mathcal{X}} \setminus \tilde{\mathcal{X}}_0 = X \times \mathbb{C}_{-}$. 
Put $\Psi_0 := F - G$. 

\begin{defin}
We say 
\begin{gather} 
\bm{\psi} \in C^{1,1} (X, \sigma; \mathcal{M})
\end{gather}
if for $\Psi (x, \tau) := \psi_{-\log |\tau|} (x)$, $\Psi - \Psi_0$ extends to a $C^{1,1}$-function on $\tilde{\mathcal{X}}_\Delta$ for some $\tilde{\mathcal{X}}$. 
The notion is independent of the choice of equivariant forms $\Sigma + \nu \in \beta^* \bar{\mathcal{M}}$, $\Delta + \delta \in [D^{U(1)}]$. 
Note we assume the regularity across the central fibre $\mathcal{X}_0$ in this notation. 
\end{defin}

For a smooth ample test configuration $(\mathcal{X}, \mathcal{L}; \rho)$, we have $C^{1,1}$-regularity across the central fibre $\mathcal{X}_0$. 

\subsubsection{Equivariant tensor calculus}

\begin{thm}
\label{limit formula}
Let $(\mathcal{X}, \mathcal{L})$ be a test configuration and $\mathcal{M}$ be a $U(1)$-equivariant class in $H^{1,1}_{U (1)} (\mathcal{X}, \mathbb{R})$. 
Take smooth $(1,1)$-forms $\omega \in i^* \mathcal{L}$, $\sigma \in i^* \mathcal{M}$. 
For rays of functions $\bm{\phi} \in C^{1, 1} (X, \omega; \mathcal{L})$ and $\bm{\psi} \in C^{1,1} (X, \sigma; \mathcal{M})$, we have 
\begin{gather*} 
(e^{\mathcal{L}|_{\mathcal{X}_0}}; \rho) = \lim_{t \to \infty} \int_X e^{\omega_{\phi_{\rho t}} - \pi \rho \dot{\phi}_{\rho t}}, 
\\
(\mathcal{M}|_{\mathcal{X}_0}. e^{\mathcal{L}|_{\mathcal{X}_0}}; \rho) = \lim_{t \to \infty} \int_X (\sigma_{\psi_{\rho t}} - \pi \rho \dot{\psi}_{\rho t}) e^{\omega_{\phi_{\rho t}}- \pi \rho \dot{\phi}_{\rho t}}. 
\end{gather*}
We note $d\phi_{\rho t}/dt = \rho \dot{\phi}_{\rho t}$. 
\end{thm} 

By the localization formula, we have 
\begin{gather*} 
(e^{\mathcal{L}|_{\mathcal{X}_0}}; \rho) = (e^L) - \rho (e^{\bar{\mathcal{L}}}; \rho), 
\\
(\mathcal{M}|_{\mathcal{X}_0}. e^{\mathcal{L}|_{\mathcal{X}_0}}; \rho) = (M. e^L) - \rho (\bar{\mathcal{M}}. e^{\bar{\mathcal{L}}}; \rho). 
\end{gather*}
We will relate the limit to the right hand side. 

We do not need any positivity on $\mathcal{L}$ and $\omega_{\phi_t}$ in the above theorem as we compute the integration directly, by applying equivariant Stokes theorem under $C^{1,1}$-regularity. 

\begin{rem}
Take another test configuration $\mathcal{X}'$ dominating $\mathcal{X}$, we can also consider the intersection $(\bar{\mathcal{M}}'. e^{\bar{\mathcal{L}}}; \rho)$ for $\mathcal{M}' \in H^{1,1}_{U (1)} (\mathcal{X}', \mathbb{R})$ by either pushing $\bar{\mathcal{M}}'$ to $\bar{\mathcal{X}}$ as divisor or pulling back $\bar{\mathcal{L}}$ to $\bar{\mathcal{X}}'$. 
The outputs coincide by the projection formula. 
The above theorem holds also for $\bm{\psi}' \in C^{1,1} (X, \sigma, \mathcal{M}')$ with this intersection. 

\end{rem}

We prepare the following key formula. 

\begin{prop}
\label{limit as equivariant intersection}
In the setup of the above theorem, we have 
\[ (\bar{\mathcal{M}}. \bar{\mathcal{L}}^{n+k}; \rho) = \lim_{t \to \infty} - \frac{1}{\rho} \int_X (\sigma_{\psi_t} - \pi \rho \dot{\psi}_t) (\omega_{\phi_t} - \pi \rho \dot{\phi}_t)^{n+k} \]
for $k \ge 0$. 
\end{prop}

\begin{proof}
Replacing $\mathcal{X}$ with the resolution $\tilde{\mathcal{X}}$, we may assume $\mathcal{X}$ is smooth dominating the trivial test configuration (as the equivariant intersection can be computed on $\tilde{\mathcal{X}}$ thanks to the projection formula). 
Take smooth equivariant forms $\Omega + \mu \in \bar{\mathcal{L}}$, $\Sigma + \nu \in \bar{\mathcal{M}}$. 
Then by the assumption there is a $C^{1,1}$-function $\tilde{\Phi}, \tilde{\Psi}$ on $\bar{\mathcal{X}}$ such that $\omega + d_{U (1)} d^c \Phi = (\Omega + \mu) + d_{U (1)} d^c \tilde{\Phi}$, $\sigma + d_{U (1)} d^c \Psi = (\Sigma + \nu) + d_{U (1)} d^c \tilde{\Psi}$ on $X \times \mathbb{C}_{-}$. 
\begin{align*}
(\bar{\mathcal{M}}. \bar{\mathcal{L}}^{n+k})
&= \int_{\bar{\mathcal{X}}} (\Sigma + \nu) (\Omega + \mu)^{n+k} 
\\
&= \int_{\bar{\mathcal{X}}} (\Sigma + \nu + d_{U (1)} d^c \tilde{\Psi}) (\Omega + \mu + d_{U (1)} d^c \tilde{\Phi})^{n+k}
\end{align*}
by the equivariant Stokes theorem. 
Since the equivariant forms have $L^\infty$-coeffiecients, we have 
\[ \int_{\bar{\mathcal{X}}} (\Sigma + \nu + d_{U (1)} d^c \tilde{\Psi}) (\Omega + \mu + d_{U (1)} d^c \tilde{\Phi})^{n+k} = \lim_{t \to \infty} \int_{\varpi^{-1} (\Delta_t)} (\Sigma + \nu + d_{U (1)} d^c \tilde{\Psi}) (\Omega + \mu + d_{U (1)} d^c \tilde{\Phi})^{n+k}. \]
Again by the equivariant Stokes theorem, we have 
\begin{align*} 
\int_{\varpi^{-1} (\Delta_t)} 
&(\Sigma + \nu + d_{U (1)} d^c \tilde{\Psi}) (\Omega + \mu + d_{U (1)} d^c \tilde{\Phi})^{n+k} 
\\
&= \int_{X \times \Delta_t} (\sigma + d_{U (1)} d^c \Psi) \wedge (\omega + d_{U (1)} d^c \Phi)^{n+k} 
\\
&= \int_{X \times \Delta_t} d_{U (1)} d^c \Psi \wedge (\omega+ d_{U (1)} d^c \Phi)^{n+k} 
\\
&\qquad + \sum_{l=0}^{n-1} \binom{n+k}{l} \int_{X \times \Delta_t} d_{U (1)} d^c \Phi \wedge \sigma \wedge \omega^l \wedge (d_{U (1)} d^c \Phi)^{n-1-l+k}
\\
&= \int_{X \times \partial \Delta_t} (d^c \Psi)|_{X \times \partial \Delta_t} \wedge (\omega + d_{U (1)} (d^c \Phi)|_{X \times \partial \Delta_t})^{n+k} 
\\
&\qquad +\sum_{l=0}^{n-1} \binom{n+k}{l} \int_{X \times \partial \Delta_t} (d^c \Phi)|_{X \times \partial \Delta_t} \wedge \sigma \wedge \omega^l \wedge (d_{U (1)} (d^c \Phi)|_{X \times \partial \Delta_t})^{n-1-l+k}
\\
&= \int_{X \times \partial \Delta_t} (d^c_X \psi + \dot{\psi} \frac{d\theta}{2}) \wedge (\omega_\phi + d_X \dot{\phi} \wedge \frac{d\theta}{2} - \pi \dot{\phi}. \eta^\vee)^{n+k} 
\\
&\qquad +\sum_{l=0}^{n-1} \binom{n+k}{l} \int_{X \times \partial \Delta_t} (d^c_X \phi + \dot{\phi} \frac{d\theta}{2}) \wedge \sigma \wedge \omega^l \wedge (d_X d_X^c \phi + d_X \dot{\phi} \wedge \frac{d\theta}{2} - \pi \dot{\phi}. \eta^\vee)^{n-1-l+k}. 
\end{align*}

(1) We compute the first term in the last line as 
\begin{align*}
\int_{X \times \partial \Delta_t} 
&(d^c_X \psi + \dot{\psi} \frac{d\theta}{2}) \wedge (\omega_\phi + d_X \dot{\phi} \wedge \frac{d\theta}{2} - \pi \dot{\phi}. \eta^\vee)^{n+k} 
\\
&= (n+k) \int_X d^c \psi \wedge d (\pi \dot{\phi}) \wedge (\omega_\phi - \pi \dot{\phi}. \eta^\vee)^{n+k-1} + \int_X \pi \dot{\psi} (\omega_\phi - \pi \dot{\phi}. \eta^\vee)^{n+k} 
\\
&= (n+k) \binom{n+k-1}{k} \int_X d^c \psi \wedge d (\pi \dot{\phi}) \wedge \omega_\phi^{n-1} (- \pi \dot{\phi}. \eta^\vee)^k + \binom{n+k}{k} \int_X \pi \dot{\psi} \omega_\phi^n (- \pi \dot{\phi}. \eta^\vee)^k. 
\end{align*}
We can compute the integrand of the first term in two ways as 
\begin{align*} 
[d^c \psi \wedge 
&d (\pi \dot{\phi}) \wedge (\omega_\phi - \pi \dot{\phi}. \eta^\vee)^{n+k-1} ]_{2n}
\\
&= -d ([(\pi \dot{\phi}) d^c \psi \wedge (\omega_\phi - \pi \dot{\phi}. \eta^\vee)^{n+k-1}]_{2n-1}) + [(\pi \dot{\phi}) dd^c \psi \wedge (\omega_\phi - \pi \dot{\phi}. \eta^\vee)^{n+k-1}]_{2n}
\\
& \qquad - (\pi \dot{\phi}) d^c \psi \wedge \binom{n+k-1}{k} \omega_\phi^{n-1} \wedge k (- \pi \dot{\phi}. \eta^\vee)^{k-1}. d (- \pi \dot{\phi}). \eta^\vee 
\end{align*}
and 
\begin{align*}
[d^c \psi \wedge d (\pi \dot{\phi}) 
&\wedge (\omega_\phi - \pi \dot{\phi}. \eta^\vee)^{n+k-1} ]_{2n}
\\
&= \binom{n+k-1}{k} d^c \psi \wedge d (\pi \dot{\phi}) \wedge \omega_\phi^{n-1}. (-\pi \dot{\phi}. \eta^\vee)^k. 
\end{align*}
Putting these together, we obtain 
\begin{align*} 
(n+k) \int_X d^c \psi \wedge d (\pi \dot{\phi}) \wedge (\omega_\phi - \pi \dot{\phi}. \eta^\vee)^{n+k-1} 
&= \frac{n+k}{k+1} \int_X (\pi \dot{\phi}) dd^c \psi \wedge (\omega_\phi - \pi \dot{\phi}. \eta^\vee)^{n+k-1}
\\
&= - (\eta^\vee)^{-1} \int_X dd^c \psi \wedge (\omega_\phi - \pi \dot{\phi}. \eta^\vee)^{n+k}
\end{align*}
and so we get 
\begin{align*}
\int_{X \times \partial \Delta_t} 
&(d^c_X \psi + \dot{\psi} \frac{d\theta}{2}) \wedge (\omega_\phi + d_X \dot{\phi} \wedge \frac{d\theta}{2} - \pi \dot{\phi}. \eta^\vee)^{n+k} 
\\
&= - (\eta^\vee)^{-1} \int_X dd^c \psi \wedge (\omega_\phi - \pi \dot{\phi}. \eta^\vee)^{n+k} + (\eta^\vee)^{-1} \int_X (\pi \dot{\psi}. \eta^\vee) (\omega_\phi - \pi \dot{\phi}. \eta^\vee)^{n+k}.  
\end{align*}

(2) Similarly, we compute the second term as 
\begin{align*}
\int_{X \times \partial \Delta_t} 
&(d^c_X \phi + \dot{\phi} \frac{d\theta}{2}) \wedge \sigma \wedge \omega^l \wedge (d_X d_X^c \phi + d_X \dot{\phi} \wedge \frac{d\theta}{2} - \pi \dot{\phi}. \eta^\vee)^{n-1-l+k}
\\
&= (n-1-l+k) \int_X d^c \phi \wedge \sigma \wedge \omega^l \wedge d (\pi \dot{\phi}) \wedge (dd^c \phi - \pi \dot{\phi}. \eta^\vee)^{n-1-l+k-1}
\\
&\qquad + \binom{n-1-l+k}{k} \int_X (\pi \dot{\phi}) \sigma \wedge \omega^l \wedge (dd^c \phi)^{n-1-l} (- \pi \dot{\phi}. \eta^\vee)^k, 
\end{align*}
unless $l= n-1$ and $k=0$, in which case we have 
\[ \int_{X \times \partial \Delta_t} (d^c_X \phi + \dot{\phi} \frac{d\theta}{2}) \wedge \sigma \wedge \omega^l \wedge (d_X d_X^c \phi + d_X \dot{\phi} \wedge \frac{d\theta}{2} - \pi \dot{\phi}. \eta^\vee)^{n-1-l+k} = \int_X (\pi \dot{\phi}) \sigma \wedge \omega^{n-1}. \]

When $l \le n-2$, we compute the integrand of the first part in two ways as 
\begin{align*} 
[d^c \phi \wedge 
&\sigma \wedge \omega^l \wedge d (\pi \dot{\phi}) \wedge (dd^c \phi - \pi \dot{\phi}. \eta^\vee)^{n-1-l+k-1}]_{2n} 
\\
&= - d [(\pi \dot{\phi}) d^c \phi \wedge \sigma \wedge \omega^l \wedge (dd^c \phi - \pi \dot{\phi}. \eta^\vee)^{n-1-l+k-1}]_{2n} 
\\
&\qquad + [(\pi \dot{\phi}) dd^c \phi \wedge \sigma \wedge \omega^l \wedge (dd^c \phi - \pi \dot{\phi}. \eta^\vee)^{n-1-l+k-1}]_{2n}
\\
&\qquad \quad - (\pi \dot{\phi}) d^c \phi \wedge \sigma \wedge \omega^l \wedge \binom{n-1-l+k-1}{k} (dd^c \phi)^{n-2-l} \wedge k (- \pi \dot{\phi}. \eta^\vee)^{k-1}. d (-\pi \dot{\phi}). \eta^\vee 
\end{align*}
and 
\begin{align*}
[d^c \phi \wedge 
&\sigma \wedge \omega^l \wedge d (\pi \dot{\phi}) \wedge (dd^c \phi - \pi \dot{\phi}. \eta^\vee)^{n-1-l+k-1}]_{2n} 
\\
&= \binom{n-1-l+k-1}{k} d^c \phi \wedge \sigma \wedge \omega^l \wedge d (\pi \dot{\phi}) \wedge (dd^c \phi)^{n-2-l} (- \pi \dot{\phi}. \eta^\vee)^k. 
\end{align*}
When $l= n-1$, it is zero. 

Combining these together, we obtain 
\begin{align*} 
(n-1-l
&+k) \int_X d^c \phi \wedge \sigma \wedge \omega^l \wedge d (\pi \dot{\phi}) \wedge (dd^c \phi - \pi \dot{\phi}. \eta^\vee)^{n-1-l+k-1} 
\\
&= \binom{n-1-l+k}{k+1} \int_X (\pi \dot{\phi}) \sigma \wedge \omega^l \wedge (dd^c \phi)^{n-1-l} (- \pi \dot{\phi}. \eta^\vee)^k 
\end{align*}
for $l \ge n-2$ and $= 0$ when $l=n-1$. 
Thus we get  
\begin{align*}
\int_{X \times \partial \Delta_t} 
&(d^c_X \phi + \dot{\phi} \frac{d\theta}{2}) \wedge \sigma \wedge \omega^l \wedge (d_X d_X^c \phi + d_X \dot{\phi} \wedge \frac{d\theta}{2} - \pi \dot{\phi}. \eta^\vee)^{n-1-l+k}
\\
&= \binom{n-l+k}{k+1} \int_X (\pi \dot{\phi}) \sigma \wedge \omega^l \wedge (dd^c \phi)^{n-1-l} (- \pi \dot{\phi}. \eta^\vee)^k, 
\end{align*}
which holds also for $l=n-1$. 

Finally, using 
\[ \sum_{l=0}^{n-1} \binom{n+k}{l} \binom{n-l+k}{k+1} \omega^l \wedge (dd^c \phi)^{n-1-l} = \binom{n+k}{k+1} \omega_\phi^{n-1}, \]
we obtain 
\begin{align*}
\sum_{l=0}^{n-1} \binom{n+k}{l} 
&\int_{X \times \partial \Delta_t} (d^c_X \phi + \dot{\phi} \frac{d\theta}{2}) \wedge \sigma \wedge \omega^l \wedge (d_X d_X^c \phi + d_X \dot{\phi} \wedge \frac{d\theta}{2} - \pi \dot{\phi}. \eta^\vee)^{n-1-l+k}
\\
&=\binom{n+k}{k+1} \int_X (\pi \dot{\phi}) \sigma \wedge \omega_\phi^{n-1} (- \pi \dot{\phi}. \eta^\vee)^k
\\
&= -(\eta^\vee)^{-1} \int_X \sigma \wedge (\omega_\phi - \pi \dot{\phi}. \eta^\vee)^{n+k}. 
\end{align*}

\vspace{2mm}
Combining the all, we get 
\[ (\bar{\mathcal{M}}. \bar{\mathcal{L}}^{n+k}) = \lim_{t \to \infty} - (\eta^\vee)^{-1} \int_X (\sigma_\psi - \pi \dot{\psi}. \eta^\vee) (\omega_\phi - \pi \dot{\phi}. \eta^\vee)^{n+k}. \]
\end{proof}

\begin{proof}[Proof of Theorem \ref{limit formula}]
By the above proposition, we get 
\begin{gather*}
(e^{\mathcal{L}|_{\mathcal{X}_0}}; \rho) = \frac{1}{n!} \int_X \omega_{\phi_t}^n + \sum_{k=1}^\infty \frac{1}{(n+k)!} \lim_{t \to \infty} \int_X (\omega_{\phi_t} - \pi \rho \dot{\phi}_t)^{n+k}
\\
(\mathcal{M}|_{\mathcal{X}_0}. e^{\mathcal{L}|_{\mathcal{X}_0}}; \rho) = \frac{1}{(n-1)!} \int_X \sigma_{\psi_t} \wedge \omega_{\phi_t}^{n-1} + \sum_{k=0}^\infty \frac{1}{(n+k)!} \lim_{t \to \infty} \int_X (\sigma_{\psi_t} - \pi \rho \dot{\psi}_t) (\omega_{\phi_t} - \pi \rho \dot{\phi}_t)^{n+k}. 
\end{gather*}

Since $\dot{\phi}_t = i_\eta d^c \Phi|_{\mathcal{X}_{-\log t}} = (\mu_\eta + i_\eta d^c \tilde{\Phi})|_{\mathcal{X}_{-\log t}}$, it is uniformly bounded. 
Thus we have 
\[ \left| \int_X (\omega_{\phi_t} - \pi \rho \dot{\phi}_t)^{n+k} \right| \le \binom{n+k}{k} \int_X |\pi \rho \dot{\phi}_t|^k \omega_{\phi_t}^n \le \binom{n+k}{k} (L^{\cdot n}) C^k. \]
Then by the dominated convergence theorem, we get 
\begin{align*} 
\sum_{k=1}^\infty \frac{1}{(n+k)!} \lim_{t \to \infty} \int_X (\omega_{\phi_t} - \pi \rho \dot{\phi}_t)^{n+k} 
&= \lim_{t \to \infty} \sum_{k=1}^\infty \frac{1}{(n+k)!} \int_X (\omega_{\phi_t} - \pi \rho \dot{\phi}_t)^{n+k}
\\
&= \lim_{t \to \infty} \int_X \sum_{k=1}^\infty \frac{1}{(n+k)!} (\omega_{\phi_t} - \pi \rho \dot{\phi}_t)^{n+k}
\\
&= - \frac{1}{n!} \int_X \omega_{\phi_t}^n + \lim_{t \to \infty} \int_X e^{\omega_{\phi_t} - \pi \rho \dot{\phi}_t}.  
\end{align*}
This proves the first equality in the claim. 
We can similarly discuss on $\psi$. 
\end{proof}

\begin{cor}
\label{slope formula}
Let $(\mathcal{X}, \mathcal{L})$ be a smooth test configuration and $\bm{\phi}$ be the $C^{1,1}$-geodesic ray subordinate to $(\mathcal{X}, \mathcal{L})$ emanating from a smooth metric. 

\begin{enumerate}
\item We have 
\begin{gather*} 
\frac{1}{n!} \int_X e^{-\rho \dot{\phi}_{\rho t}} \omega_{\phi_{\rho t}}^n = (e^{\mathcal{L}|_{\mathcal{X}_0}}; \pi^{-1} \rho), 
\\
\frac{1}{n!} \int_X \rho \dot{\phi}_{\rho t} e^{-\rho \dot{\phi}_{\rho t}} \omega_{\phi_{\rho t}}^n = n (e^{\mathcal{L}|_{\mathcal{X}_0}}; \pi^{-1} \rho) - (\mathcal{L}|_{\mathcal{X}_0}. e^{\mathcal{L}|_{\mathcal{X}_0}}; \pi^{-1} \rho). 
\end{gather*}

\item For $\bm{\psi} \in C^\infty (X, -\mathrm{Ric} (\omega), 2\pi K^{\log}_{\bar{\mathcal{X}}/\mathbb{P}^1})$ and $\bm{\psi}' \in C^\infty (X, -\mathrm{Ric} (\omega), 2\pi K^{\log}_{\bar{\mathcal{X}}/\mathbb{P}^1})$, we have 
\begin{gather*} 
\lim_{t \to \infty} \frac{d}{dt} \mathcal{A}_{\bm{\phi}_\rho}^{\bm{\psi}_\rho} (t) = -2\pi (\kappa_{\mathcal{X}_0}. e^{\mathcal{L}|_{\mathcal{X}_0}}; \pi^{-1} \rho), 
\\
\lim_{t \to \infty} \frac{d}{dt} \mathcal{A}_{\bm{\phi}_\rho}^{\bm{\psi}'_\rho} (t) = -2\pi ((K_X. e^L) - \pi^{-1} \rho (K^{\log}_{\bar{\mathcal{X}}/\mathbb{P}^1}. e^{\bar{\mathcal{L}}}; \pi^{-1} \rho)). 
\end{gather*}
\end{enumerate}
\end{cor}

\begin{proof}
As we already noted, we may assume $\bar{\mathcal{L}}$ is ample. 
Take a K\"ahler form $\Omega$ in $\bar{\mathcal{L}}$, which can be written as $\Omega = \omega + dd^c \Phi_\Omega$ on $\bar{\mathcal{X}} \setminus \mathcal{X}_0$. 
It is known by \cite[Corollary 1.3]{CTW1} that there exists a unique $\Omega$-psh function $\tilde{\Phi}$ on $\mathcal{X}_\Delta$ which is $C^{1,1}$ on $\mathcal{X}_\Delta$ and solves 
\[ (\Omega + dd^c \tilde{\Phi})^{n+1}  = 0, \quad \tilde{\Phi}|_{\partial \mathcal{X}_\Delta} = \tilde{\phi}_0. \]
Thus the geodesic ray $\bm{\phi} = \{ \phi_t = \Phi_\Omega (\cdot, e^{-t}) + \tilde{\Phi} (\cdot, e^{-t}) \}_{t \in [0, \infty)}$ is of class $C^{1,1} (X, \omega; \mathcal{L})$. 
Thus we can apply the above theorem, which shows the limits are expressed by the corresponding equivariant intersections. 
By Lemma \ref{affine functional}, the integrations in (1) is independent of $t$, so it is equal to the limit $t \to \infty$. 
\end{proof}

\subsubsection{Slope formula of the action functional}

\begin{thm}
Let $(\mathcal{X}, \mathcal{L})$ be an snc smooth test configuration which dominates the trivial configuration. 
For the geodesic ray $\bm{\phi}$ subordinate to $(\mathcal{X}, \mathcal{L})$ emanating from a smooth metric, we have 
\[ \lim_{t \to \infty} \frac{d}{dt}_+ \mathcal{A}_{\bm{\phi}_\rho} (t) = -2\pi ((K_X. e^L) - \pi^{-1} \rho (K^{\log}_{\bar{\mathcal{X}}/\mathbb{P}^1}. e^{\bar{\mathcal{L}}}; \pi^{-1} \rho)). \]
\end{thm}

\begin{proof}
We firstly note $\bm{\phi}_{(\mathcal{X}, \mathcal{L}), \rho} = \bm{\phi}_{(\mathcal{X}_d, \mathcal{L}_d), d^{-1} \rho}$ for the geodesic $\bm{\phi}_{(\mathcal{X}, \mathcal{L}), \rho}$ subordinate to $(\mathcal{X}, \mathcal{L}; \rho)$. 
On the right hand side, we have 
\[ \rho (K^{\log}_{\bar{\mathcal{X}}/\mathbb{P}^1}. e^{\bar{\mathcal{L}}}; \rho) = d^{-1} \rho (K^{\log}_{\bar{\mathcal{X}}_d/\mathbb{P}^1}. e^{\bar{\mathcal{L}}_d}; d^{-1} \rho) \]
for the normalized base change $(\mathcal{X}_d, \mathcal{L}_d)$, as we already explained in the introduction. 
For sufficiently divisible integer $d > 0$, the central fibre of $\mathcal{X}_d$ is reduced, so that we have $\rho (K^{\log}_{\bar{\mathcal{X}}/\mathbb{P}^1}. e^{\bar{\mathcal{L}}}; \rho) = d^{-1} \rho (K_{\bar{\mathcal{X}}_d/\mathbb{P}^1}. e^{\bar{\mathcal{L}}}_d; d^{-1} \rho)$ for sufficiently divisible $d$. 

Therefore, it suffices to show 
\[ \mathcal{A}_{\bm{\phi}_\rho} - \mathcal{A}_{\bm{\phi}_\rho}^{\bm{\psi}_\rho} \le O (1) \]
and 
\[ o (t) \le \mathcal{A}_{\bm{\phi}_\rho} - \mathcal{A}_{\bm{\phi}_\rho}^{\bm{\psi}'_\rho} \]
for smooth $\bm{\psi}, \bm{\psi'}$ in the previous corollary. 

As in the proof of \cite[Theorem 5.1]{S-D}, we have a uniform upper bound on $\log (\omega^n_{\phi_t}/e^{-\psi_t} \omega^n)$. 
So we get 
\[ \mathcal{A}_{\bm{\phi}_\rho} - \mathcal{A}_{\bm{\phi}_\rho}^{\bm{\psi}_\rho} = \frac{1}{n!} \int_X \log (\omega^n_{\phi_{\rho t}}/e^{-\psi_{\rho t}} \omega^n) e^{- \rho \dot{\phi}_{\rho t}} \omega_{\phi_{\rho t}}^n \le C (e^{\mathcal{L}|_{\mathcal{X}_0}}; \pi^{-1} \rho) \]
as desired. 

To see 
\[ o (t) \le \mathcal{A}_{\bm{\phi}_\rho} - \mathcal{A}_{\bm{\phi}_\rho}^{\bm{\psi}'}, \]
we put $d\mu_t := e^{-\rho \dot{\phi}_{\rho t}} \omega^n_{\phi_{\rho t}}/(n! (e^{\mathcal{L}|_{\mathcal{X}_0}}))$, $d\nu_t := e^{-\psi'_{\rho t}} \omega^n/\int_X e^{-\psi'_{\rho t}} \omega^n$ and compute 
\begin{align*} 
\mathcal{A}_{\bm{\phi}_\rho} - \mathcal{A}_{\bm{\phi}_\rho}^{\bm{\psi}'_\rho} 
&= \frac{1}{n!} \int_X \log (\omega^n_{\phi_{\rho t}}/e^{-\psi'_{\rho t}} \omega^n) e^{- \rho \dot{\phi}_{\rho t}} \omega_{\phi_{\rho t}}^n
\\ 
&= (e^{\mathcal{L}|_{\mathcal{X}_0}}) \log (e^{\mathcal{L}|_{\mathcal{X}_0}}) + \frac{1}{n!} \int_X \rho \dot{\phi}_{\rho t} e^{-\rho \dot{\phi}_{\rho t}} \omega_{\phi_{\rho t}}^n 
\\
&\qquad + (e^{\mathcal{L}|_{\mathcal{X}_0}}) \int_X \log \frac{d\mu_t}{d\nu_t} d\mu_t - (e^{\mathcal{L}|_{\mathcal{X}_0}}) \log \frac{1}{n!} \int_X e^{-\psi'_{\rho t}} \omega^n 
\\
& \ge C - (e^{\mathcal{L}|_{\mathcal{X}_0}}) \log \frac{1}{n!} \int_X e^{-\psi'_{\rho t}} \omega^n, 
\end{align*}
where we used $\int_X \log \frac{d\mu_t}{d\nu_t} d\mu_t \ge 0$ in the last inequality. 
The last term is of class $O (\log t) \subset o (t)$ by \cite[Lemma 3.11]{BHJ2}. 
This proves the claim. 
\end{proof}

\subsubsection{Proof of main theorems}

Now we can prove the rest of our main theorems. 

\begin{cor}
For $\lambda \in \mathbb{R}$, we have 
\[ \sup_{(\mathcal{X}, \mathcal{L}), \rho \ge 0} \cmu^\lambda (\mathcal{X}, \mathcal{L}; \rho) \le \inf_{\omega_\varphi \in \mathcal{H} (X, L)} \cmu^\lambda (\omega_\varphi), \]
where $(\mathcal{X}, \mathcal{L}; \rho)$ runs over all test configurations. 
\end{cor}

\begin{proof}
Now by Lemma \ref{W plus}, Lemma \ref{affine functional}, the convexity of $\mathcal{A}_{\bm{\phi}}$, Corollary \ref{slope formula} and the above theorem, we get 
\[ \cmu (\omega) \ge \cW (\omega, - \rho \dot{\phi}_0) \ge - \frac{\frac{d}{dt}_+|_{t=0} \mathcal{A}_{\bm{\phi}_\rho}}{\int_X e^{-\rho \dot{\phi}_0} \omega^n/n!} \ge \lim_{t \to \infty} - \frac{\frac{d}{dt}_+ \mathcal{A}_{\bm{\phi}_\rho} (t)}{\int_X e^{-\rho \dot{\phi}_{\rho t}} \omega_{\phi_{\rho t}}^n/n!} = \NAmu (\mathcal{X}, \mathcal{L}; \pi^{-1} \rho) \]
for an snc smooth test configuration $(\mathcal{X}, \mathcal{L})$ and $\rho \ge 0$. 

On the other hand, by Lemma \ref{affine functional} and Corollary \ref{slope formula}, we have 
\[ \check{S} (\omega, -\rho \dot{\phi}_0) = \check{S} (\omega_{\phi_{\rho t}}, -\rho \dot{\phi}_{\rho t}) = \bm{\check{\sigma}} (\mathcal{X}, \mathcal{L}; \pi^{-1} \rho). \]
Thus the claim holds for snc smooth test configurations. 

For a general test configuration $(\mathcal{X}, \mathcal{L})$, shifting $\mathcal{L}$ by a weight so that $\bar{\mathcal{L}}$ is ample and then taking a resolution $\tilde{\mathcal{X}} \to \mathcal{X}$ and approximating the pull-back $\tilde{\mathcal{L}}$ by ample series $\bar{\mathcal{L}}_\epsilon$, we can approximate $\NAmu (\mathcal{X}, \mathcal{L}; \rho)$ by $\NAmu (\tilde{\mathcal{X}}, \mathcal{L}_\epsilon; \rho)$. 
This shows $\cmu (\omega) \ge \NAmu (\mathcal{X}, \mathcal{L}; \rho)$ for general $(\mathcal{X}, \mathcal{L}; \rho)$. 
\end{proof}

\begin{cor}
\label{minimizer}
Suppose $\lambda \le 0$. 
If $\omega$ is a $\mu^\lambda$-cscK metric, then it minimizes $\cmu^\lambda$. 
\end{cor}

\begin{proof}
We firstly note $\cmu^\lambda (X, L; \xi) \le \cmu^\lambda (\omega_\varphi)$ for every vector field $\xi$ and metric $\omega_\varphi$ (not necessarily $\xi$-invariant). 
Indeed, when $\xi$ is rational, i.e. generating a $\mathbb{C}^\times$-action, this is a consequence of the above corollary. 
Since $\cmu^\lambda (X, L; \bullet)$ is continuous on any torus $\mathfrak{t}$ and rational $\xi$ is dense in $\mathfrak{t}$, we obtain the inequality for general $\xi$. 

If $\omega$ is a $\mu^\lambda$-cscK metric for $\lambda \le 0$, then there is $\xi$ such that $\cmu^\lambda (\omega) = \cmu^\lambda (X, L; \xi)$. 
Thus we have $\cmu^\lambda (\omega) \le \cmu^\lambda (\omega_\varphi)$ for every metric $\omega_\varphi$. 
\end{proof}

If $\omega$ is a $\check{\mu}^\lambda_\xi$-cscK metric for $\lambda \le 0$, then $\xi$ maximizes $\cmu^\lambda (X, L; \bullet)$ among all vectors. 
Thus we get the following slight refinement of \cite[Corollary 3.10]{Ino2}. 

\begin{cor}
Conjecture \ref{uniqueness} reduces to the uniqueness of the maximizer of $\cmu^\lambda (X, L; \bullet)$ on a maximal torus $\mathfrak{t}$ or on the center of a maximal compact (cf. \cite[Corollary 3.19]{Ino2}). 
\end{cor}

\section{Observations}
\label{Observations}

Here we briefly observe relation of our framework with He and Calabi functional. 

\subsection{K\"ahler--Ricci soliton and He functional}
\label{Observations, He}

\subsubsection{$H$-functional and $\mu$-entropy}

It is observed in \cite{Per, TZ2} that on a Fano manifold $X$ the critical points of the $\mu$-entropy for the polarization $L = -K_X$ and $\lambda = 2\pi$ are precisely K\"ahler--Ricci solitons, implicitly assuming the smoothness of the $\mu$-entropy in this ``mild temperature'' case $\lambda = 2 \pi$. 
Though the author could not recover the actual proof of the smoothness in this case, we can rephrase this observation in the following two ways. 
\begin{itemize}
\item The critical points of $\cW^{2\pi}$ are precisely K\"ahler--Ricci solitons. 

\item The minimizers of $\cmu^{2\pi}$ are precisely K\"ahler--Ricci solitons. 
\end{itemize}

The first one follows from the main theorem of this article and the fact that $\mu^{2\pi}$-cscK metrics in the polarization $L=-K_X$ on a Fano manifold are precisely K\"ahler--Ricci solitons (cf. \cite{Ino1, Ino2}). 

If $\omega$ is a K\"ahler--Ricci soliton, i.e. $\mathrm{Ric} (\omega) - L_{J\xi} \omega = 2\pi \omega$, we have the equality $\cmu^{2\pi} (\omega) = \cW^{2\pi} (\omega, \theta_\xi) = \cmu^{2\pi} (-2\xi)$ as observed in \cite{TZ2}. 
Since we have $\cmu^{2\pi} (\omega_\varphi) \ge \cmu^{2\pi} (-2\xi)$ for general metric $\omega_\varphi$, the K\"ahler--Ricci soliton $\omega$ minimizes the $\mu$-entropy. 
To see the reverse implication, we discuss as follows: if $\omega$ minimizes $\cmu^{2\pi}$, it also minimizes the $H$-entropy by the equality $\inf_\omega \cmu^{2\pi} (\omega) = \inf_\omega H (\omega)$ from \cite{DS}, so that it is a K\"ahler--Ricci soliton. 
(Note the sign of our $\mu$-entropy is reversed from Perelman's original convention. ) 

The \textit{$H$-entropy} of a K\"ahler metric $\omega$ in $\mathcal{H} (X, -K_X)$ is given by 
\[ \frac{1}{2\pi} H (\omega) := \int_X h e^h \omega^n \Big{/} \int_X e^h \omega^n - \log  \int_X e^h \frac{\omega^n}{n!}, \]
where $h$ is a Ricci potential: $\sqddbar h = \Ric (\omega) - 2\pi \omega$. 
The critical points of $H$-entropy are K\"ahler--Ricci solitons. 
As noticed in \cite{He}, we have 
\[ \cW^{2\pi} (\omega, h) = H (\omega), \]
so that the following inequality holds 
\begin{equation} 
\cmu^{2\pi} (\omega) \ge H (\omega), 
\end{equation}
which is analogous to the inequality $M \ge D$ on Mabuchi and Ding functional, and the equality holds iff $\omega$ is a K\"ahler--Ricci soliton (cf. \cite{He}, \cite[Theorem 2.4.3]{Fut}). 

\subsubsection{Legendre duality}

The following formula on the $H$-entropy is known as Legendre duality: 
\begin{equation} 
\frac{1}{2\pi} H (\omega) = \sup_{f \in C^0 (X)} \Big{(} \int_X f e^h \omega^n \Big{/} \int_X e^h \omega^n - \log \int_X e^f \frac{\omega^n}{n!} \Big{)}, 
\end{equation}
which follows from Jensen's inequality: 
\[ \int_X (f- h) \frac{e^h \omega^n}{\int_X e^h \omega^n} \le \log \int_X e^{f-h} \frac{e^h \omega^n}{\int_X e^h \omega^n}. \]

Now we consider the following functional 
\begin{equation} 
\frac{1}{2 \pi} \check{\mathcal{L}} (\omega, f) := \int_X f e^h \omega^n \Big{/} \int_X e^h \omega^n - \log \int_X e^f \frac{\omega^n}{n!}
\end{equation}
defined on the tangent bundle $T \mathcal{H} (X, -K_X)$. 
We easily see that critical points of $\check{\mathcal{L}}$ are K\"ahler--Ricci solitons. 
One can consider this functional as an extension of Tian--Zhu's volume functional \cite{TZ1}. 

It is shown in \cite{DS} that this functional is monotonic along geodesics. 
They then proved the following equality for $H$-invariant: 
\[ \sup_{(\mathcal{X}, \mathcal{L})} H (\mathcal{X}, \mathcal{L}) = \inf_{\omega_\varphi \in \mathcal{H} (X, L)} H (\omega_\varphi), \]
based on the work \cite{CSW} on the limiting behavior of K\"ahler--Ricci flow. 
The monotonicity is used for the inequality $\sup \le \inf$. 

We note Berman gave another bound on $\mu$-entropy for Fano manifolds (cf. \cite[Corollary 4.5]{Ber}). 

\subsection{Extremal metric and Calabi functional}
\label{Observations, Calabi}

\subsubsection{Extremal limit $\lambda^{-1} \to 0$}

As in \cite{Ino2}, we consider the following rescaled $W$-entropy for $\kappa \neq 0$: 
\begin{align} 
\cW^\dagger_\kappa (\omega, f) 
&:= \kappa^{-1} \Big{(} \cW^{\kappa^{-1}} (\omega, \kappa f) - \cW^{\kappa^{-1}} (\omega, 0) \Big{)}
\\ \notag
&= \kappa^{-1} \Big{(} \frac{\int_X (s (\omega) + \kappa \bar{\Box} f) e^{\kappa f} \omega^n}{\int_X e^{\kappa f} \omega^n} - \bar{s} \Big{)} 
\\ \notag
&\qquad- \kappa^{-2} \Big{(} \frac{\int_X \kappa f e^{\kappa f} \omega^n}{\int_X e^{\kappa f} \omega^n} - \log \int_X e^{\kappa f} \omega^n \Big{)}.  
\end{align}
The computation in \cite[Section 5.2]{Ino2} shows the limit as $\kappa \to 0$ is given by 
\begin{equation} 
\mathcal{W}_{\mathrm{ext}} (\omega, f) := \lim_{\kappa \to 0} \cW^\dagger_\kappa (\omega, f) = - \frac{1}{2} \frac{\int_X (\hat{s} (\omega) - \hat{f})^2 \omega^n}{\int_X \omega^n} + \frac{1}{2} \frac{\int_X \hat{s}^2 (\omega) \omega^n}{\int_X \omega^n}, 
\end{equation}
where we put $\hat{u} := u - \int_X u \omega^n / \int_X \omega^n$. 
We easily see that $f$ is a critical point of $\mathcal{W}_{\mathrm{ext}} (\omega, \cdot)$ iff $f = s (\omega) + \mathrm{const}.$, which obviously maximizes $\mathcal{W}_{\mathrm{ext}} (\omega, \cdot)$. 
Since 
\[ C (\omega) := \sup_{f \in C^\infty (X)} \mathcal{W}_{\mathrm{ext}} (\omega, f) = \frac{1}{2} \frac{\int_X \hat{s}^2 (\omega) \omega^n}{\int_X \omega^n} \]
is nothing but the Calabi functional, the critical points of $\mathcal{W}_{\mathrm{ext}}$ are precisely extremal metrics. 

We can similarly check the derivative of the functionals $\cW^\dagger_\kappa$ also converge to that of $\mathcal{W}_{\mathrm{ext}}$ as in \cite{Ino2} again. 
It is also observed in the paper when $\kappa$ approaches to zero from the negative side $\kappa < 0$ ($\lambda \to -\infty$), the holomorphic vector fields $\partial^\sharp f_\kappa$ associated to some critical points $(\omega_\kappa, f_\kappa)$ of $\cW^\dagger_\kappa$ always converge to an extremal vector field modulo holomorphic gauges: it suffices to take gauges $g_\kappa \in \mathrm{Aut} (X, L)$ so that $g_\kappa^* \omega_\kappa$ are invariant with respect to the center of a fixed maximal compact. 
In contrast, when $\kappa$ approaches to zero from the positive side $\kappa > 0$ ($\lambda \to + \infty$), there may be critical points whose associated holomorphic vector fields $\partial^\sharp f_\kappa$ fade away to infinity, so that $(\omega_\kappa, f_\kappa)$ never converges. 

On can also derive the functional $\mathcal{W}_{\mathrm{ext}}$ in the spirit of volume minimization. 
The functional $\mathcal{W}_{\mathrm{ext}} (\omega, \mu_\xi^\omega)$ defines a convex functional on the Lie algebra $\mathfrak{k}$ of a compact Lie group acting on $(X, L)$. 
The derivative at $\xi$ gives the relative Futaki invariant relative to $\xi$. 
We can recover the uniqueness of extremal vector by the convexity, which rephrases the proof in \cite{FM} in view of volume minimization. 

\subsubsection{Donaldson's lower bound on Calabi functional --- a simple proof}

Standing on our Lagrangian perspective, we exhibit a simple proof of Donaldson's inequality \cite{Don}: 
\[ \frac{2\pi^2}{(L^{\cdot n})} \frac{\mathrm{DF} (\mathcal{X}, \mathcal{L})^2}{\| (\mathcal{X}, \mathcal{L}) \|^2} \le C (\omega) \]
for any test configuration $(\mathcal{X}, \mathcal{L})$ with $\mathrm{DF} (\mathcal{X}, \mathcal{L}) \le 0$. 
The proof here shares the spirit with the proof in \cite{DS} for the $H$-functional and that in this article for the $W$-entropy. 
The readers should compare the proof with the Kempf--Ness approach in \cite{Der1}, which pursues the idea of \cite{Chen}. 
The author is grateful to R. Dervan for informing the author these reference. 

Now we put 
\begin{align*}
M_{\mathrm{NA}} (\mathcal{X}, \mathcal{L}; \rho) 
&:= \rho \Big{(} (K^{\log}_{\bar{\mathcal{X}}/\mathbb{C}P^1}. \bar{\mathcal{L}}^{\cdot n}) - \frac{n}{n+1} \frac{(K_X. L^{\cdot n-1})}{(L^{\cdot n})} (\bar{\mathcal{L}}^{\cdot n}) \Big{)}, 
\\
\| (\mathcal{X}, \mathcal{L}; \rho) \|^2 
&:= - \Big{(} \frac{2 \rho (\bar{\mathcal{L}}_{\mathbb{C}^{\times}}^{\cdot n+2}; \rho)}{(n+2)(n+1)} + \frac{1}{(L^{\cdot n})} \Big{(} \frac{\rho (\bar{\mathcal{L}}^{\cdot n+1})}{n+1} \Big{)}^2 \Big{)}, 
\\
C_{\mathrm{NA}} (\mathcal{X}, \mathcal{L}; \rho) 
&:= \frac{-1}{2 (L^{\cdot n})} \left( 4\pi M_{\mathrm{NA}} (\mathcal{X}, \mathcal{L}; \rho) + \| (\mathcal{X}, \mathcal{L}; \rho) \|^2 \right). 
\end{align*}
We have $M_{\mathrm{NA}} (\mathcal{X}, \mathcal{L}; \rho) = \rho M_{\mathrm{NA}} (\mathcal{X}, \mathcal{L}) \le \rho \mathrm{DF} (\mathcal{X}, \mathcal{L})$ as in \cite{BHJ1}. 
By the equivariant localization, we have 
\begin{align*} 
\frac{2\rho (\bar{\mathcal{L}}_{\mathbb{C}^{\times}}^{\cdot n+2}; \rho)}{(n+2)!} + \frac{n!}{(L^{\cdot n})} \Big{(} \frac{\rho (\bar{\mathcal{L}}^{\cdot n+1})}{(n+1)!} \Big{)}^2 
&= - 2! \frac{(\mathcal{L}_{\mathbb{C}^\times}|_{\mathcal{X}_0}^{\cdot n+2}; \rho)}{(n+2)!} + \frac{n!}{(L^{\cdot n})} \Big{(} \frac{(\mathcal{L}_{\mathbb{C}^\times}|_{\mathcal{X}_0}^{\cdot n+1}; \rho)}{(n+1)!} \Big{)}^2 
\\
&= - \int_\mathbb{R} (-\rho t)^2 \mathrm{DH} +  \frac{(\int_\mathbb{R} (-\rho t) \mathrm{DH})^2}{\int_\mathbb{R} \mathrm{DH}} 
\\
&= - \rho^2 \int_\mathbb{R} (t- b)^2 \mathrm{DH} < 0 
\end{align*}
for $b := \int_\mathbb{R} t \mathrm{DH}/\int_\mathbb{R} \mathrm{DH}$,  so that $\| (\mathcal{X}, \mathcal{L}; \rho) \|^2$ is positive unless $(\mathcal{X}, \mathcal{L})$ is trivial. 
Since we have 
\[ \frac{d}{d\rho}\Big{|}_{\rho = 0} C_{\mathrm{NA}} (\mathcal{X}, \mathcal{L}; \rho) = - \frac{2\pi}{(L^{\cdot n})} M_{\mathrm{NA}} (\mathcal{X}, \mathcal{L}), \]
$(X, L)$ is K-semistable if $C_{\mathrm{NA}}$ is maximized at the trivial configuration among all test configurations. 
In the same way as we will do for the non-archimedean $\mu$-entropy in the subsequent article, we can also extend this to the extremal case: $(X, L)$ is relatively K-semistable if $C_{\mathrm{NA}}$ is maximized at a product configuration among all test configurations (cf. \cite{Der2}). 
Especially when $(X, L)$ is K-semistable, we can show the reverse implication, since if $\mathrm{DF} (\mathcal{X}, \mathcal{L}) \ge 0$ then the maximum of $C_{\mathrm{NA}} (\mathcal{X}, \mathcal{L}; \rho)$ on $\rho \ge 0$ is achieved at $\rho = 0$, so the trivial configuration maximizes $C_{\mathrm{NA}}$. 

When $\mathrm{DF} (\mathcal{X}, \mathcal{L}) \le 0$, the maximum is achieved at $\rho = -2\pi M_{\mathrm{NA}} (\mathcal{X}, \mathcal{L})/\| (\mathcal{X}, \mathcal{L}) \|^2$ and is given by 
\[ \max_{\rho \ge 0} C_{\mathrm{NA}} (\mathcal{X}, \mathcal{L}; \rho) = \frac{2\pi^2}{(L^{\cdot n})} \frac{M_{\mathrm{NA}} (\mathcal{X}, \mathcal{L})^2}{\| (\mathcal{X}, \mathcal{L}) \|^2} > 0. \]
Therefore Donaldson's theorem \cite{Don} on the lower bound of the Calabi functional can be rephrased as 
\[ \sup_{(\mathcal{X}, \mathcal{L}; \rho)} C_{\mathrm{NA}} (\mathcal{X}, \mathcal{L}; \rho) \le \inf_{\omega \in \mathcal{H} (X, L)} C (\omega). \]
It is conjectured that these values are indeed the same. (cf. \cite{His, Xia, Li2})

The following proof is based on the developments on the convexity (\cite{BB}) of the Mabuchi functional along $C^{1,1}$-geodesic rays and the slope formula (\cite{S-D}. See also \cite{Xia, Li2}. ) for geodesic rays with algebraic singularities. 

Firstly, we note that 
\[ \int_0^t \mathcal{W}_{\mathrm{ext}} (\omega_s, \dot{\phi}_s) ds = - \frac{1}{2 (L^{\cdot n})} \left( 2 \mathcal{M} (\omega_t) + \int_0^t ds \int_X \hat{\dot{\phi}}_s^2 \omega_s^n \right) \]
for smooth paths $\omega_s$. 
Along $C^{1, 1}$-geodesics, the integrand $\int_X \hat{\dot{\phi}}_s^2 \omega_s^n$ is constant and equal to $\| (\mathcal{X}, \mathcal{L}) \|^2$ by Proposition \ref{limit as equivariant intersection} (cf. \cite{His16}). 

Since the right derivative $\frac{d}{dt}_+ \Big{|}_{t=s} \mathcal{M} (\omega_t)$ exists for a $C^{1,1}$ geodesic $\omega_t$ by the convexity of the Mabuchi functional, we can define $\mathcal{W}_{\mathrm{ext}}^+ (\omega_s, \dot{\phi}_s)$ for a $C^{1,1}$-geodesic by 
\[ \mathcal{W}_{\mathrm{ext}}^+ (\omega_s, \dot{\phi}_s) := - \frac{1}{2\int_X \omega^n} \left( 2 \frac{d}{dt}_+ \Big{|}_{t=s} \mathcal{M} (\omega_t) + \int_X \hat{\dot{\phi}}_s^2 \omega_s^n \right), \]
which is monotonically decreasing again by the convexity. 
Meanwhile, if $\omega_0$ is smooth, $\mathcal{W}_{\mathrm{ext}} (\omega_0, \dot{\phi}_0)$ is well-defined in the usual sense as $\dot{\phi}_0$ is $L^2$-integrable. 
We can relate these as follows. 
Recall the Chen--Tian expression of the Mabuchi functional
\[ \mathcal{M} (\omega_t) = \int_X \log \frac{\omega_t^n}{\omega_0^n} \omega_t^n - n \int_0^t ds \int_X \dot{\phi}_s \mathrm{Ric} (\omega_0) \wedge \omega_s^{n-1}. \] 
Since the entropy part $\int_X \log \frac{\omega_t^n}{\omega_0^n} \omega_t^n$ is non-negative, we have 
\[ \frac{d}{dt}_+ \Big{|}_{t=0} \mathcal{M} (\omega_t) \ge - \int_X \dot{\phi}_0 s (\omega_0) \omega_0^n. \] 
Thus we get $\mathcal{W}_{\mathrm{ext}} (\omega_0, \dot{\phi}_0) \ge \mathcal{W}_{\mathrm{ext}}^+ (\omega_0, \dot{\phi}_0)$. 

On the other hand, the slope formula on the Mabuchi functional shows 
\[ C_{\mathrm{NA}} (\mathcal{X}, \mathcal{L}; \rho) = \lim_{t \to \infty} \mathcal{W}_{\mathrm{ext}}^+ (\omega_{\rho t}, \dot{\phi}_{\rho t}) \]
for the $C^{1,1}$-geodesic ray subordinate to the test configuration $(\mathcal{X}, \mathcal{L})$ with any smooth initial metric $\omega_0$. 
Now we obtain the inequality by 
\[ C_{\mathrm{NA}} (\mathcal{X}, \mathcal{L}; \rho) = \lim_{t \to \infty} \mathcal{W}_{\mathrm{ext}}^+ (\omega_{\rho t}, \dot{\phi}_{\rho t}) \le \mathcal{W}_{\mathrm{ext}}^+ (\omega_0, \rho \dot{\phi}_0) \le \mathcal{W}_{\mathrm{ext}} (\omega_0, \rho \dot{\phi}_0) \le C (\omega_0) \]
for every $(\mathcal{X}, \mathcal{L}; \rho)$ and $\omega_0$. 
This proves Donaldson's inequality.

\subsection{Are they Lagrangian?}

We saw that $\cW^\lambda$, $\mathcal{W}_{\mathrm{ext}}$ and $\check{\mathcal{L}}$ in different frameworks are monotonic along geodesics. 
It is natural to ask if there is some common universal background for these functionals. 
The author have not yet reached a good answer to this naive question. 
Here we describe such attempt, without mentioning detail. 
Though the observation does not conclude anything so far, the author thinks it is worth to share. 

The functionals are defined on the tangent bundle of the space of K\"ahler metrics. 
Lagrangian in physics is also a functional on tangent bundle. 
This motivates us to observe the Euler--Lagrange equation of these functional. 
It turns out that there is a common feature for $\mathcal{W}_{\mathrm{ext}}$ and $\check{\mathcal{L}}$. 

We briefly review Lagrangian formalism in physics. 
It explains the motion of a particle by the principle of least action. 
The formalism is formulated in a coordinate free expression, so that it is suitable to deal with holonomic constraints, such as pendulum. 
A Lagrangian is a functional $\mathcal{L}: T M \to \mathbb{R}$ on the tangent bundle of a space $M$. 
The action functional $\mathscr{A}$ is a functional on the path space $C^\infty ([a,b], M)$ which is defined as 
\[ \mathscr{A} (\phi) := \int_a^b \mathcal{L} (\phi (t), \dot{\phi} (t)) dt. \]

The principle of least action states that the motion of a particle minimizes this functional, among all the curve $\phi$ with the fixed initial and end positions $\phi (a) = x, \phi (b) = y$. 
Critical points of the functional is characterized by the Euler--Lagrange equation: 
\begin{equation}
\frac{\partial \mathcal{L}}{\partial q} (\phi (t), \dot{\phi} (t)) - \frac{d}{dt} \Big{(} \frac{\partial \mathcal{L}}{\partial \dot{q}} (\phi (t), \dot{\phi} (t)) \Big{)} = 0, 
\end{equation}
where $(q, \dot{q})$ is a local (canonical) coordinate of $T M$. 
This equation is known to be coordinate free. 
In the Cartesian coordinate, Lagrangian is given by $\mathcal{L} = T -V$, where $T = \frac{1}{2} m |v|^2$ is the kinematic energy and $V$ is a potential energy of the system. 

Lagrangian formalism appears also in mathematics. 
Geodesics in Riemannian geometry can be characterized by the Euler--Lagrange equation on the kinematic Lagrangian $\mathcal{L} (x, v) = \frac{1}{2} |v|_x^2$. 
Perelman's $\mathscr{L}$-geometry \cite{Per} may be regarded as a Lagrangian formalism on the spacetime of Ricci flow. 

Lempert \cite{Lem} observed that parallel fibrewise convex Lagrangians on $T \mathcal{H} (X, L)$ is minimized by geodesics. 
We note our `Lagrangians' $\cW^\lambda, \mathcal{W}_{\mathrm{ext}}, \check{\mathcal{L}}$ are defined on the same space, but are not parallel (as for $\cW^\lambda$, not even fibrewise convex). 
Now, let us observe what happens for our `Lagrangians'. 

\subsubsection{$\check{\mathcal{L}}$ is a Lagrangian}

Recall we put 
\[ \frac{1}{2 \pi} \check{\mathcal{L}} (\omega, f) = \int_X f e^h \omega^n \Big{/} \int_X e^h \omega^n - \log \int_X e^f \frac{\omega^n}{n!}. \]
We regard $\check{\mathcal{L}} (\omega, - f)$ as a Lagrangian. 
Then the Euler--Lagrange equation turns into 
\begin{equation}
\label{EL}
\frac{\partial \check{\mathcal{L}}}{\partial \omega} (\omega_t, - \dot{\phi}_t) + \frac{d}{dt} \frac{\partial \check{\mathcal{L}}}{\partial f} (\omega_t, -\dot{\phi}_t) = 0. 
\end{equation}

A simple calculation shows the following. 

\begin{prop}
The Euler--Lagrange equation (\ref{EL}) is equivalent to the geodesic equation. 
\end{prop}

The \textit{energy} of this Lagrangian system is given by 
\begin{equation} 
\mathcal{U} (\omega, f) := \langle \frac{\partial \check{\mathcal{L}}}{\partial f} (\omega, f), f \rangle - \check{\mathcal{L}} (\omega, f) = -2\pi \Big{(} \frac{\int_X f e^f \omega^n}{\int_X e^f \omega^n} - \log \int_X e^f \frac{\omega^n}{n!} \Big{)}. 
\end{equation}
It explains the reason that this functional is conserved along geodesics.

\subsubsection{$\mathcal{W}_{\mathrm{ext}}$ is a Lagrangian}

Recall we put 
\[ \mathcal{W}_{\mathrm{ext}} (\omega, f) = - \frac{1}{2} \frac{\int_X (\hat{s} (\omega) - \hat{f})^2 \omega^n}{\int_X \omega^n} + \frac{1}{2} \frac{\int_X \hat{s}^2 (\omega) \omega^n}{\int_X \omega^n}.  \]
We regard $\mathcal{W}_{\mathrm{ext}} (\omega, f)$ as Lagrangian. 
Similarly as before, we can observe the following. 

\begin{prop}
The Euler--Lagrange equation 
\[ \frac{\partial \mathcal{W}_{\mathrm{ext}}}{\partial \omega} (\omega_t, \dot{\phi}_t) - \frac{d}{dt} \frac{\partial \mathcal{W}_{\mathrm{ext}}}{\partial f} (\omega_t, \dot{\phi}_t) = 0 \] 
is equivalent to the geodesic equation. 
\end{prop}

%
%

The \textit{energy} of this Lagrangian system is given by 
\begin{equation} 
\mathcal{H}_{\mathrm{ext}} (\omega, f) := \langle \frac{\partial \mathcal{W}_{\mathrm{ext}}}{\partial f} (\omega, f), f \rangle - \mathcal{W}_{\mathrm{ext}} (\omega, f) = -\frac{1}{2} \frac{\int_X \hat{f}^2 \omega^n}{\int_X \omega^n}. 
\end{equation}
It again explains the reason that this functional is conserved along geodesics. 

%

\subsubsection{$\cW^\lambda$ is ...}

Recall we put 
\begin{align*} 
 \cW (\omega, f) 
 &= - \frac{\int_X (s (\omega) + \bar{\Box} f) e^f \omega^n}{\int_X e^f \omega^n}, 
\\
\check{S} (\omega, f) 
&= \frac{\int_X (n+f) e^f \omega^n}{\int_X e^f \omega^n} - \log \int_X e^f \frac{\omega^n}{n!}. 
\end{align*}
We regard $\cW^\lambda (\omega, - f) = \cW (\omega, - f) + \lambda \check{S} (\omega, -f)$ as Lagrangian. 

\begin{prop}
We have 
\[ \Big{(} \frac{\partial \check{S}}{\partial \omega} (\omega_t, -\dot{\phi}_t) + \frac{d}{dt} \frac{\partial \check{S}}{\partial f} (\omega_t, -\dot{\phi}_t) \Big{)} (\varphi) = \frac{\int_X (1 + \hat{f} ) \widehat{(\dot{f} + |\partial f|^2)} \hat{\varphi} ~ e^f \omega^n}{\int_X e^f \omega^n} \]
and 
\begin{align*}
\Big{(} \frac{\partial \cW}{\partial \omega} (\omega_t, -\dot{\phi}_t) 
&+ \frac{d}{dt} \frac{\partial \cW}{\partial f} (\omega_t, -\dot{\phi}_t) \Big{)} (\varphi) 
\\
&= - \frac{\int_X \Big{(} |\mathcal{D} f|^2 + (\Delta - \nabla f) \widehat{(\dot{f} +|\partial f|^2)} + \hat{s}_f \widehat{(\dot{f} + |\partial f|^2)} \Big{)} \hat{\varphi} ~e^f \omega^n}{\int_X e^f \omega^n},
\end{align*}
where we put $f := -\dot{\phi}$. 
\end{prop}

Therefore, we conclude that $\cW^\lambda$ does not characterize geodesics by the principle of least action, obstructed by the integration $\int_X |\mathcal{D} \dot{\phi}|^2 \hat{\varphi} ~e^{-\dot{\phi}} \omega^n$. 

\begin{quest}
What does this mean? 
\end{quest}

\subsubsection{A remark by Laszlo Lempert}

The author was leant from Laszlo Lempert that the functionals $\check{\mathcal{L}}$ and $\mathcal{W}_{\mathrm{ext}}$ are both the sum of parallel Lagrangian and a closed 1-form, which is given by the differential of Ding and Mabuchi functional, respectively.  
The Euler--Lagrange equation of a functional does not change when adding a closed 1-form, so that our observation that the Euler--Lagrange equation of these functionals give geodesic equation can be derived from his result on parallel Lagrangians \cite{Lem}.


\begin{thebibliography}{100}

\bibitem[Ber]{Ber} R. J. Berman, \textit{K-polystability of $\mathbb{Q}$-Fano varieties admitting K\"ahler--Einstein metrics}, Invent. Math. \textbf{203}, 3 (2016), 973--1025. 

\bibitem[BB]{BB} R. J. Berman, B. Berndtsson, \textit{Convexity of the $K$-energy on the space of K\"ahler metrics and uniqueness of extremal metrics}, J. Amer. Math. Soc. \textbf{30} (2017), 1165--1196. 

\bibitem[BBJ]{BBJ} R. J. Berman, S. Boucksom, M. Jonsson, \textit{A variational approach to the Yau--Tian--Donaldson conjecture}, \verb|arXiv:1509.04561v3|. 

\bibitem[BHJ1]{BHJ1} S. Boucksom, T. Hisamoto, M. Jonsson, \textit{Uniform K-stability, Duistermaat--Heckman measures and singularities of pairs}, Ann. Inst. Fourier (Grenoble) \textbf{67} (2017), 87--139. 

\bibitem[BHJ2]{BHJ2} S. Boucksom, T. Hisamoto, M. Jonsson, \textit{Uniform K-stability and asymptotics of energy functionals in K\"ahler geometry}, J. Eigne Angrew. Math. \textbf{751} (2019), 27--89. 

\bibitem[BJ1]{BJ1} S. Boucksom, M. Jonsson, \textit{Singular semipositive metrics on line bundles on varieties over trivially valued fields}, \verb|arXiv:1801.08229|. 

\bibitem[BJ2]{BJ2} S. Boucksom, M. Jonsson, \textit{A non-archimedean approach to K-stability}, \verb|arXiv:1805.11160|. 

\bibitem[Chen]{Chen} X. Chen, \textit{Space of K\"ahler metrics. III. On the lower bound of the Calabi energy and geodesic distance}, Invent. Math. \textbf{175}, 3 (2009), 453--503.  

\bibitem[CC]{CC} X. Chen, J. Cheng, \textit{On the constant scalar curvature K\"ahler metrics (I) -- a priori estimate}, \verb|arXiv:|1712.06697. 

\bibitem[CSW]{CSW} X. Chen, S. Sun, B. Wang, \textit{K\"ahler-Ricci flow, K\"ahler-Einstein metric, and K-stability}, Geom. Topol. \textbf{22} (2018), 3145--3173. 

\bibitem[CTW1]{CTW1} J. Chu, V. Tossati, B. Weinkove, \textit{On the $C^{1,1}$-regularity of geodesics in the space of K\"ahler metrics}, Ann. PDE \textbf{3}, 2 (2017), paper No. 15, 12pp. 

\bibitem[CTW2]{CTW2} J. Chu, V. Tossati, B. Weinkove, \textit{$C^{1,1}$ regularity for degenerate complex Monge-Amp\`ere equations and geodesic rays}, Comm. Partial Differential Equations \textbf{3} (2018), 292--312. 

\bibitem[Der1]{Der1} R. Dervan, \textit{Relative K-stability for K\"ahler manifolds}, Math. Ann. \textbf{372}, No. 3--4 (2018), 859--889. 

\bibitem[Der2]{Der2} R. Dervan, \textit{K-semistability of optimal degenerations}, Q. J. Math. \textbf{71}, 3 (2020), 989--995. 

\bibitem[DS]{DS} R. Dervan, G. Sz\'ekelyhidi, \textit{The K\"ahler-Ricci flow and optimal degenerations}, J. Diff. Geom. \textbf{116}, 1 (2020), 187--203. 

\bibitem[Don]{Don} S. Donaldson, \textit{Lower bounds on the Calabi functional}, J. Diff. Geom. \textbf{70}, 3 (2005), 453--472. 

\bibitem[EG1]{EG1} D. Edidin, W. Graham, \textit{Equivariant Intersection Theory}, Invent. Math. \textbf{131}, 3 (1998), 595--634. 

\bibitem[EG2]{EG2} D. Edidin, W. Graham, \textit{Riemann--Roch for equivariant Chow groups}, Duke Math. J. \textbf{102}, 3 (2000), 567--594. 

\bibitem[Ful]{Ful} Fulton, \textit{Intersection Theory, second edition}, Springer-Verlag, Berlin, 1998. 

\bibitem[Fut]{Fut} A. Futaki, \textit{K\"ahler--Einstein Metrics and Integral Invariants}, Invent. Math. \textbf{73} (1983), 437--443. 

\bibitem[FM]{FM} A. Futaki, T. Mabuchi, \textit{Bilinear forms and extremal K\"ahler vector fields}, Math. Ann. \textbf{301} (1995), 2, 199--210. 


\bibitem[GGK]{GGK} V. Ginzburg, V. Guillemin, Y. Karshon, \textit{Moment maps, cobordisms, and Hamiltonian group actions}, Mathematical Surveys and Monographs, 98. American Mathematical Society, Providence, RI, 2002. 

\bibitem[GS]{GS} V. W. Guillemin, A. Sternberg, \textit{Supersymmetry and Equivariant deRham Theory}, Mathematics Past and Present. Springer-Verlag, Berlin, 1999. 

\bibitem[HL1]{HL1} J. Han, C. Li, \textit{On the Yau-Tian-Donaldson conjecture for generalized K\"ahler-Ricci soliton equations}, \verb|arXiv:2006.00903|. 

\bibitem[HL2]{HL2} J. Han, C. Li, \textit{Algebraic uniqueness of K\"ahler-Ricci flow limits and optimal degenerations of Fano varieties}, \verb|arXiv.2009.01010|. 

\bibitem[He]{He} W. He, \textit{K\"ahler--Ricci soliton and $H$-functional}, Asian J. Math. \textbf{20} (2016), 645--664. 

\bibitem[His1]{His16} T. Hisamoto, \textit{On the limit of spectral measures associated to a test configuration of a polarized K\"ahler manifold}, J. Reigne Angew. Math. \textbf{713} (2016), 129--148

\bibitem[His2]{His} T. Hisamoto, \textit{Geometric flow, multiplier ideal sheaves and optimal destabilizer for a Fano manifold}, \verb|arXiv:1901.08480|. 

\bibitem[Ino1]{Ino1} E. Inoue, \textit{The moduli space of Fano manifolds with K\"ahler--Ricci solitons}, Adv. in Math. \textbf{357} (2019), 106841. 

\bibitem[Ino2]{Ino2} E. Inoue, \textit{Constant $\mu$-scalar curvature K\"ahler metric - formulation and foundational results}, J. Geom. Anal. \textbf{32} (2022), Article number 145. 

\bibitem[Ino3]{Ino3} E. Inoue, \textit{Equivariant calculus on $\mu$-character and $\mu$K-stability of polarized schemes}, \verb|arXiv:2004.06393|. 

\bibitem[Ino4]{Ino4} E. Inoue, \textit{Entropies in $\mu$-framework of canonical metrics and K-stability, II -- Non-archimedean aspect: non-archimedean $\mu$-entropy and $\mu$K-stability}, preprint. 

\bibitem[Lah1]{Lah1} A. Lahdili, \textit{K\"ahler metrics with constant weighted scalar curvature and weighted K-stability}, Proc. London Math. Soc. (3) \textbf{119} (2019), 1065--1114. 

\bibitem[Lah2]{Lah2} A. Lahdili, \textit{Convexity of the weighted Mabuchi functional and the uniqueness of weighted extremal metrics}, \verb|arXiv:2007.01345|. 

\bibitem[Leg]{Leg} E. Legendre, \textit{Localizing the Donaldson--Futaki invariant}, \verb|arXiv:2006.08987|. 

\bibitem[Lem]{Lem} L. Lempert, \textit{The principle of least action in the space of K\"ahler potentials}, \verb|arXiv:2009.09949|. 

\bibitem[Li1]{Li1} C. Li, \textit{K-semistability of equivariant volume minimization}, Duke Math. J. \textbf{166}, 16 (2017), 3147--3218. 

\bibitem[Li2]{Li2} C. Li, \textit{Geodesic rays and stability in the cscK problem}, \verb|arXiv:2001.01366|. 


\bibitem[MSY]{MSY} D. Martelli, J. Sparks, S.-T. Yau, \textit{Sasaki-Einstein manifolds and volume minimization}, Comm. Math. Phys. \textbf{280} (2008), 611--673. 


\bibitem[Per]{Per} G. Perelman, \textit{The entropy formula for the Ricci flow and its geometric applications}, \verb|arXiv:0211159|. 

\bibitem[PS1]{PS1} D. H. Phong, J. Sturm, \textit{Test configurations for K-stability and geodesic rays}, J. Symp. Geom. \textbf{5}, 2 (2007), 221--247. 

\bibitem[PS2]{PS2} D. H. Phong, J. Sturm, \textit{Regularity of geodesic rays and Monge--Ampere equations}, Proc. Amer. Math. Soc. \textbf{138}, 10 (2010), 3637--3650. 

\bibitem[Rot]{Rot} Rothaus, \textit{Logarithmic Sobolev inequalities and the supremum of Schr\"odinger operators}, J. Func. Anal. \textbf{42} (1981), 110--120. 

\bibitem[S-D]{S-D} Z. Sj\"ostr\"om Dyrefelt, \textit{K-semistability of cscK manifolds with transcendental cohomology class}, J. Geom. Anal. \textbf{28}, 4 (2018), 2927--2960. 


\bibitem[TZZZ]{TZZZ} G. Tian, S. Zhang, Z. Zhang, X. Zhu, \textit{Supremum of Perelman's entropy and K\"ahler--Ricci flow on a Fano manifold}, Trans. Amer. Math. Soc. \textbf{365}, 12 (2013), 6669--6695. 

\bibitem[TZ1]{TZ1} G. Tian, X. Zhu, \textit{A new holomorphic invariant and uniqueness of K\"ahler--Ricci solitons}, Comment. Math. Helv. \textbf{77} (2002), 297--325. 

\bibitem[TZ2]{TZ2} G. Tian, X. Zhu, \textit{Convergence of the K\"ahler--Ricci flow on Fano manifolds}, J. Amer. Math. Soc. \textbf{20} (3) (2007), 675--699. 


\bibitem[Xia]{Xia} M. Xia, \textit{On sharp lower bounds for Calabi type functionals and destabilizing properties of gradient flows}, \verb|arXiv:1901.07889|. 

\end{thebibliography}
\end{document}